\documentclass[reqno,a4paper,11pt]{article}
\usepackage{amsmath,amsthm,dsfont,amsfonts,amssymb,fancyhdr}
\usepackage{enumerate}
\usepackage[usenames,dvipsnames]{color}
\usepackage{bbm,soul,supertabular,longtable,verbatim,extarrows}
\usepackage{titlesec}
\usepackage{graphicx}
\usepackage[titletoc,toc,title]{appendix}

\usepackage{tikz}
\usetikzlibrary{arrows,decorations.pathmorphing,backgrounds,positioning,fit,petri}

\usepackage{caption}

\usepackage{float}
\usepackage{BOONDOX-cal}
\usepackage{tikz}
\usepackage{booktabs,multirow,makecell}
\usepackage{authblk}
\usepackage[titletoc]{appendix}
\usetikzlibrary[trees]
\usepackage{array}
\usepackage{lipsum}
\usepackage{mathrsfs}
\usepackage{indentfirst}
\usepackage[colorlinks=true, allcolors=blue]{hyperref}
\usepackage[top=2.4cm,bottom=2.2cm,left=2.6cm,right=2cm]{geometry}

\newtheorem{thm}{Theorem}[section]
\newtheorem{lem}[thm]{Lemma}
\newtheorem{pro}[thm]{Proposition}
\newtheorem{de}[thm]{Definition}
\newtheorem{re}[thm]{Remark}
\newtheorem{remark}[thm]{Remark}

\newtheorem{cor}[thm]{Corollary}

\newtheorem{conj}[thm]{Conjecture}

\newcommand{\Ho}{\mathfrak{H}}
\newcommand{\ho}{\mathfrak{h}}
\newcommand{\Go}{\mathfrak{G}}
\newcommand{\go}{\mathfrak{g}}

\newcommand{\Bo}{\mathfrak{B}}
\newcommand{\Po}{\mathfrak{P}}
\newcommand{\Co}{\mathfrak{C}}

\newcommand\keywords[1]{\textbf{Keywords}: #1}
\newcommand\MSC[1]{\textbf{MSC2020}: #1}

\soulregister\cite7
\soulregister\citep7
\soulregister\citet7
\soulregister\ref7
\soulregister\pageref7
\sethlcolor{yellow}
\setstcolor{red}
\soulregister{\em}{0}

\begin{document}
	\title{\bf A Bose-Laskar-Hoffman theory for $\mu$-bounded graphs with fixed smallest eigenvalue}

    \author[a,b]{Jack H. Koolen}
	\author[a,c]{Hong-Jun Ge}
	\author[a]{Chenhui Lv}
	\author[d,e]{Qianqian Yang\thanks{Corresponding author}}
	\affil[a]{\footnotesize{School of Mathematical Sciences, University of Science and Technology of China, Hefei, 230026, People's Republic of China}}
	\affil[b]{\footnotesize{CAS Wu Wen-Tsun Key Laboratory of Mathematics, University of Science and Technology of China, Hefei, 230026, People's Republic of China}}
	\affil[c]{\footnotesize{Department of Econometrics and O.R., Tilburg University, Tilburg, Netherlands}}
	\affil[d] {\footnotesize{Department of Mathematics, Shanghai University, Shanghai 200444, People's Republic of China}}
	\affil[e] {\footnotesize{Newtouch Center for Mathematics of Shanghai University, Shanghai 200444, People's Republic of China}}
	
	\maketitle
	\pagestyle{plain}
	
	\newcommand\blfootnote[1]{%
		\begingroup
		\renewcommand\thefootnote{}\footnote{#1}%
		\addtocounter{footnote}{-1}%
		\endgroup}
%	\blfootnote{2020 Mathematics Subject Classification. 05C50, 05C75, 05E30} 
	\blfootnote{E-mail addresses: {\tt koolen@ustc.edu.cn} (J.H. Koolen), {\tt gehj22@mail.ustc.edu.cn} (H.-J. Ge), {\tt lch1994@mail.ustc.edu.cn} (C. Lv), {\tt  qqyang@shu.edu.cn} (Q. Yang).}

\vspace{-50pt}
\begin{abstract}
%{\bf We need to do better, I changed the title, Please check the references as some are outdated}
	In 2018, by Ramsey and Hoffman theory, Koolen, Yang, and Yang presented a structural result on graphs with smallest eigenvalue at least $-3$ and large minimum degree. 
	In this study, we depart from the conventional use of Ramsey theory and instead employ a novel approach that combines the Bose-Laskar type argument with Hoffman theory to derive structural insights into $\mu$-bounded graphs with fixed smallest eigenvalue. Our method establishes a reasonable bound on the minimum degree.
	Note that local graphs of distance-regular graphs are $\mu$-bounded. We apply these results to characterize the structure for any local graph 
	of a distance-regular graph with classical parameters $(D,b,\alpha,\beta)$. Consequently, we show that the parameter $\alpha$ is bounded by a cubic polynomial in $b$ if $D \geq 9$ and $b \geq 2$. 
	We { also} show that $\alpha \leq 2$ if $b =2$ and $D \geq 12$.
\end{abstract}
\MSC{05C50, 05C75, 05E30}

\vspace{-20pt}
\section{Introduction}

All graphs mentioned in this paper are finite, undirected and simple. The eigenvalues of a graph are the eigenvalues of its adjacency matrix.
Let $G$ be a graph. We denote by $\lambda_{\min}(G)$ the smallest eigenvalue of $G$, $N(x)$ the neighbors of vertex $x$, and $d(x)$ the degree of vertex $x$.
For a non-negative integer $a$, let $[a]:=\{1,\ldots,a\}$ and $[0]:=\emptyset$.
 For undefined notations, see \cite{BCN}.
 
A graph is called \emph{$\mu$-bounded} with parameter $c$ if each pair of distinct non-adjacent vertices has at most $c$ common neighbors. Note that the local graph of distance-regular graphs at every vertex is $\mu$-bounded.
  The notions of distance-regular graphs and distance-regular graphs with classical parameters are introduced in Section \ref{2}.

 In this paper, we will give some structural results on $\mu$-bounded graphs with fixed smallest eigenvalue. After that, we apply these results to describe the structure for any local graph 
  of a distance-regular graph with classical parameters $(D,b,\alpha,\beta)$. As a consequence, Theorem \ref{introalpha} gives { an upper} bound on $\alpha$ in terms of $b$. Although the bound is similar { to} the bound in \cite{KLPY}, the main difference here is that we derive a reasonable bound on the diameter, whereas in \cite{KLPY}, they need a super large bound on the diameter.
  
  \begin{thm}\label{introalpha}
  	Let $D$ and $b$ be positive integers.
  	If $\Gamma$ is a distance-regular graph with classical parameters $(D,b,\alpha,\beta)$ such that $D\geq 9$ and $b\geq 2$, then $\alpha<b^2(b+1)+1$ holds.
  \end{thm}

  In particular, Theorem \ref{a5} shows that $\alpha\leq2$ if $b=2$ and $D\geq 12$.
 
   \begin{thm}\label{a5}
   	If $\Gamma$ is a distance-regular graph with classical parameters $(D,b,\alpha,\beta)$ such that $b=2$ and $D\geq 12$, then $\alpha\leq 2$ and  $\alpha\in\{0,\frac{1}{3},\frac{2}{3},1,\frac{4}{3},2\}$.
   \end{thm}
 
 Moreover, we can show the following bound for $\alpha$ if $b\in[100]$ .
 \begin{thm}\label{alphab}
 	If $\Gamma$ is a distance-regular graph with classical parameters $(D,b,\alpha,\beta)$ such that $b\in[100]$ and $D\geq14$, then the following statements hold:
 	\begin{enumerate}
 		\item If $b$ is not a perfect square, then $\alpha\leq b$.
 		\item If $b$ is a perfect square, then $\alpha\leq b $ or $\alpha=b+\sqrt{b}$.
 	\end{enumerate}
 \end{thm}  
 
 It is worth mentioning that the half dual polar graph and the distance 1-or-2 graph of symplectic dual polar graph are distance-regular graphs with classical parameters $(D,b,\alpha,\beta)$ such that $b=q^2$ and $\alpha=q(q+1)=b+\sqrt{b}$, where $q$ is a prime power. %Hence, when $b\in\{4,9,16\}$, there are distance-regular graphs with classical parameters $(D,b,\alpha,\beta)$ such that $\alpha=b+\sqrt{b}>b$. 
 {For} definition of these two kinds of distance-regular graphs, we refer readers to \cite[Chapter 9.4]{BCN}.
 
 \begin{conj}
 	Let $\Gamma$ be a distance-regular graph with classical parameters $(D,b,\alpha,\beta)$. There is a constant $\ell$ such that if $D\geq\ell$, then one of the following statements holds:
 	\begin{enumerate}
 		\item If $b$ is a perfect square, then $\alpha\leq b+\sqrt{b}$.
 		\item If $b$ is not a perfect square, then $\alpha\leq b$.
 	\end{enumerate}
 \end{conj}

In 1976, Cameron et al. \cite{Cameron} showed that a connected graph $G$ with {smallest eigenvalue} $\lambda_{\min}(G)\geq -2$ is a generalized line graph if the number of vertices is at least 37, by using the classification of the irreducible root lattices.
 In 1977, Hoffman \cite{Hoff1977} showed that a connected graph $G$ with smallest eigenvalue $\lambda_{\min}(G)>-1-\sqrt{2}$ is a generalized line graph if the {minimum degree} of $G$ is large enough. 
Hoffman did not use the classification of the irreducible root lattices for the proof, but instead he had to assume large minimum degree. 
In terms of line Hoffman graphs, a generalized line graph is the slim graph of a   \big\{\raisebox{-1ex}{\begin{tikzpicture}[scale=0.3]
		
		\tikzstyle{every node}=[draw,circle,fill=black,minimum size=10pt,scale=0.3,
		inner sep=0pt]
		
		\draw (-2.1,0) node (1f1) [label=below:$$] {};
		\draw (-1.1,0) node (1f2) [label=below:$$] {};

		\tikzstyle{every node}=[draw,circle,fill=black,minimum size=5pt,scale=0.3,
		inner sep=0pt]

		\draw (-1.6,1) node (1s1) [label=below:$$] {};
		
		\draw (1f1) -- (1s1) -- (1f2);
\end{tikzpicture}},\hspace{-0.08cm}
\raisebox{-1ex}{\begin{tikzpicture}[scale=0.3]
		\tikzstyle{every node}=[draw,circle,fill=black,minimum size=10pt,scale=0.3,
		inner sep=0pt]

		\draw (1.5,0) node (3f2) [label=below:$$] {};

		\tikzstyle{every node}=[draw,circle,fill=black,minimum size=5pt,scale=0.3,
		inner sep=0pt]

		\draw (1,1) node (3s1) [label=below:$$] {};
		\draw (2,1) node (3s2) [label=below:$$] {};
		
		\draw (3s1) -- (3f2) -- (3s2);
\end{tikzpicture}}\big\}
-line Hoffman graph. 
 The notions of Hoffman graphs and line Hoffman graphs are  introduced in Section \ref{2}, and the notion of associated Hoffman graphs in Section \ref{3}.

In 2018, Koolen et al. \cite{KYY19} showed the following result on graphs with smallest eigenvalue at least $-3$, which generalized the result by Hoffman \cite{Hoff1977}.
To state the result, we need to introduce the concept {of a} $t$-plex. Given a positive integer $t$, a \emph{$t$-plex} is a graph such that each vertex is adjacent to all but at most $p-1$ of the other vertices.
Especially, a \emph{clique} is a $1$-plex.
{The notions of Hoffman graphs, line Hoffman graphs, and slim graphs mentioned below will be introduced 
	in Sections \ref{hoffsec1} and \ref{hoffsec2}.
}

%In order to state the result, we introduce three Hoffman graphs $\mathfrak{f}_1$, $\mathfrak{f}_2$ and $\mathfrak{f}_3$ as follows:
%	\begin{figure}[H]
%		\centering
%		\includegraphics[scale=1.0]{f}
%		{\caption{}\label{figure1}}
%	\end{figure}
%\vspace{-12pt}
%(For the notions of Hoffman graphs and line Hoffman graphs, we refer to the next section.)

\begin{thm}[{\cite[Theorem 1.4]{KYY19}}]\label{intro1}
	There exists a positive integer $K$ such that if a graph $G$ satisfies the following conditions:
	\begin{enumerate}
		\item $d(x)>K$ for each $x\in V(G)$,
		\item any $5$-plex containing $x$ has order at most $d(x)-K$ for all $x \in V(G)$,
		\item $\lambda_{\min} (G) \geq -3$,
	\end{enumerate}
	then $G$ is the slim graph of a $2$-fat  \big\{\raisebox{-1ex}{\begin{tikzpicture}[scale=0.3]
			
			\tikzstyle{every node}=[draw,circle,fill=black,minimum size=10pt,scale=0.3,
			inner sep=0pt]
			
			\draw (-2.1,0) node (1f1) [label=below:$$] {};
			\draw (-1.6,0) node (1f2) [label=below:$$] {};
			\draw (-1.1,0) node (1f3) [label=below:$$] {};

			\tikzstyle{every node}=[draw,circle,fill=black,minimum size=5pt,scale=0.3,
			inner sep=0pt]

			\draw (-1.6,1) node (1s1) [label=below:$$] {};
			
			\draw (1f1) -- (1s1) -- (1f2);
			\draw (1f3) -- (1s1);
	\end{tikzpicture}},\hspace{-0.08cm}
	\raisebox{-1ex}{\begin{tikzpicture}[scale=0.3]
			\tikzstyle{every node}=[draw,circle,fill=black,minimum size=10pt,scale=0.3,
			inner sep=0pt]

			\draw (-0.5,0) node (2f1) [label=below:$$] {};
			\draw (0.5,0) node (2f2) [label=below:$$] {};
			\draw (-0.5,1) node (2f3) [label=below:$$] {};
			\draw (0.5,1) node (2f4) [label=below:$$] {};

			\tikzstyle{every node}=[draw,circle,fill=black,minimum size=5pt,scale=0.3,
			inner sep=0pt]

			\draw (0,0.2) node (2s1) [label=below:$$] {};
			\draw (0.3,0.5) node (2s2) [label=below:$$] {};
			\draw (-0.3,0.5) node (2s3) [label=below:$$] {};
			\draw (0,0.8) node (2s4) [label=below:$$] {};

			\draw (2f1) -- (2s1) -- (2f2) -- (2s2) -- (2f4) -- (2s4) -- (2f3) -- (2s3) -- (2f1);
			\draw (2s1) -- (2s4);
			\draw (2s2) -- (2s3);
	\end{tikzpicture}},\hspace{-0.08cm}
	\raisebox{-1ex}{\begin{tikzpicture}[scale=0.3]
			\tikzstyle{every node}=[draw,circle,fill=black,minimum size=10pt,scale=0.3,
			inner sep=0pt]

			\draw (1.5,0) node (3f1) [label=below:$$] {};
			\draw (1.5,1) node (3f2) [label=below:$$] {};

			\tikzstyle{every node}=[draw,circle,fill=black,minimum size=5pt,scale=0.3,
			inner sep=0pt]

			\draw (1,0.5) node (3s1) [label=below:$$] {};
			\draw (2,0.5) node (3s2) [label=below:$$] {};
			
			\draw (3s1) -- (3s2);
			\draw (3f1) -- (3s1) -- (3f2) -- (3s2) -- (3f1);
	\end{tikzpicture}}\big\}-line Hoffman graph.
\end{thm}
For the proof, it is important to have large cliques. To obtain those large cliques in the proof of Theorem \ref{intro1}, they needed to use Ramsey theory, {implying} that they do not have a reasonable bound for $K$.  

By combining the Bose-Laskar type argument and Hoffman theory, we give the following result which  is a version of Theorem \ref{intro1} for $\mu$-bounded graphs.  {As a} consequence, we derive a  reasonable value for $K$ in this case. The main reason that we can obtain a reasonable bound for $K$ is  Lemma \ref{blhlem}.
This lemma guarantees the existence of large cliques without relying on Ramsey theory {which}  plays a key role in the proof of Theorem \ref{intro2}.
%we will discuss $\mu$-bounded graph with parameter $c$ and smallest eigenvalue at least$-3$ without assuming large minimum degree by using the Bose-Laskar method and Hoffman Theory. 

\begin{thm}\label{intro2}
	Let $c$ be a positive integer. Let $\tilde{c}:=\min\{c,6\}$ and $q:=\max\{c+5, 50\tilde{c}+16\}$. If a graph $G$ satisfies the following conditions:
	\begin{enumerate}
		\item $G$ is a $\mu$-bounded graph with parameter $c$,
		\item every clique $C$ has order at most $\min_{x\in V(C)}d(x)-K$, where $K=\max\{36c+400\tilde{c}+83, 44c-5 \}$,
		\item $\lambda_{\min} (G) \geq -3$,
	\end{enumerate}
	then the associated Hoffman graph $\go(G,q)$ is a $2$-fat \big\{\raisebox{-1ex}{\begin{tikzpicture}[scale=0.3]
			
			\tikzstyle{every node}=[draw,circle,fill=black,minimum size=10pt,scale=0.3,
			inner sep=0pt]
			
			\draw (-2.1,0) node (1f1) [label=below:$$] {};
			\draw (-1.6,0) node (1f2) [label=below:$$] {};
			\draw (-1.1,0) node (1f3) [label=below:$$] {};

			\tikzstyle{every node}=[draw,circle,fill=black,minimum size=5pt,scale=0.3,
			inner sep=0pt]

			\draw (-1.6,1) node (1s1) [label=below:$$] {};
			
			\draw (1f1) -- (1s1) -- (1f2);
			\draw (1f3) -- (1s1);
	\end{tikzpicture}},\hspace{-0.08cm}
	\raisebox{-1ex}{\begin{tikzpicture}[scale=0.3]
			\tikzstyle{every node}=[draw,circle,fill=black,minimum size=10pt,scale=0.3,
			inner sep=0pt]

			\draw (1.5,0) node (3f1) [label=below:$$] {};
			\draw (1.5,1) node (3f2) [label=below:$$] {};

			\tikzstyle{every node}=[draw,circle,fill=black,minimum size=5pt,scale=0.3,
			inner sep=0pt]

			\draw (1,0.5) node (3s1) [label=below:$$] {};
			\draw (2,0.5) node (3s2) [label=below:$$] {};
			
			\draw (3s1) -- (3s2);
			\draw (3f1) -- (3s1) -- (3f2) -- (3s2) -- (3f1);
	\end{tikzpicture}}\big\}
	-line Hoffman graph.
\end{thm}
 
 Notice that if a $\mu$-bounded graph $G$ with parameter $c$ is the slim graph of a $2$-fat \big\{\raisebox{-1ex}{\begin{tikzpicture}[scale=0.3]
 		
 		\tikzstyle{every node}=[draw,circle,fill=black,minimum size=10pt,scale=0.3,
 		inner sep=0pt]
 		
 		\draw (-2.1,0) node (1f1) [label=below:$$] {};
 		\draw (-1.6,0) node (1f2) [label=below:$$] {};
 		\draw (-1.1,0) node (1f3) [label=below:$$] {};

 		\tikzstyle{every node}=[draw,circle,fill=black,minimum size=5pt,scale=0.3,
 		inner sep=0pt]

 		\draw (-1.6,1) node (1s1) [label=below:$$] {};
 		
 		\draw (1f1) -- (1s1) -- (1f2);
 		\draw (1f3) -- (1s1);
 \end{tikzpicture}},\hspace{-0.08cm}
 \raisebox{-1ex}{\begin{tikzpicture}[scale=0.3]
 		\tikzstyle{every node}=[draw,circle,fill=black,minimum size=10pt,scale=0.3,
 		inner sep=0pt]

 		\draw (1.5,0) node (3f1) [label=below:$$] {};
 		\draw (1.5,1) node (3f2) [label=below:$$] {};

 		\tikzstyle{every node}=[draw,circle,fill=black,minimum size=5pt,scale=0.3,
 		inner sep=0pt]

 		\draw (1,0.5) node (3s1) [label=below:$$] {};
 		\draw (2,0.5) node (3s2) [label=below:$$] {};
 		
 		\draw (3s1) -- (3s2);
 		\draw (3f1) -- (3s1) -- (3f2) -- (3s2) -- (3f1);
 \end{tikzpicture}}\big\}-line Hoffman graph, then its structure is characterized in
Propositions \ref{muassle} and  \ref{muassle2}, and thus $c\leq 18$ and $\lambda_{\min}(G)\geq -3$. In particular, $c\leq 12$ if $G$ is regular. 
 
%In particular, describe the structure of the slim graph $G$  of  a 
%Which implies that the parameter $c$ is at most $18$, moreover, $c$ is at most $12$ if $G$ is regular.
%	Note that the smallest eigenvalue of $G$ satisfies $\lambda_{\min}(G) \geq -3$. 
%	If $G$ is $\mu$-bounded with  parameter $c$, then we show that $c \leq 18$. If moreover $G$ is regular and $\mu$-bounded, then $c \leq 12$. 
	 
%	In fact, for fixed $\epsilon>0$, we can generalize the condition (ii) of Theorem \ref{intro2} to $\lambda_{\min} (G) \geq -2-\sqrt{2}+\epsilon$ in Theorem \ref{corep}, which implies that $\lambda_{\min}(G)\geq -3$. Of course, the bound for $K$ would depend on $\epsilon$. 
%	Note that the condition $\epsilon>0$ is necessary, as $H:=L(K_{p,p})\square K_{1,2}$ is a $\mu$-bounded graph with parameter $6$, every vertex lies in two distinct cliques of order $p$, and $\lambda_{\min}(H)=-2-\sqrt{2}$ and $p$ can be arbitrary large. But $H$ is not the slim graph of  a $2$-fat 

%If $G$ is $\mu$-bounded with  parameter $c$, then we show that $c \leq 18$. If moreover $G$ is regular and $\mu$-bounded, then $c \leq 12$. 

In fact, for fixed $\epsilon>0$, the condition (iii) of Theorem \ref{intro2} can be generalized to $\lambda_{\min} (G) \geq -2-\sqrt{2}+\epsilon$ (see Theorem \ref{corep}), and  the value of $K$ in (ii) would depend on $\epsilon$. 
It is worth mentioning that the condition $\epsilon>0$ is necessary.
For example, let $H:=L(K_{p,p})\square K_{1,2}$ be the Cartesian product of $L(K_{p,p})$ and $K_{1,2}$, where $L(K_{p,p})$ is the line graph of complete bipartite graph $K_{p,p}$. It follows that $H$
 is a $\mu$-bounded graph with parameter $6$, every vertex lies in two distinct cliques of order $p$, and $\lambda_{\min}(H)=-2-\sqrt{2}$ and $p$ can be arbitrary large. But $H$ is not the slim graph of  a $2$-fat
\big\{\raisebox{-1ex}{\begin{tikzpicture}[scale=0.3]
			
			\tikzstyle{every node}=[draw,circle,fill=black,minimum size=10pt,scale=0.3,
			inner sep=0pt]
			
			\draw (-2.1,0) node (1f1) [label=below:$$] {};
			\draw (-1.6,0) node (1f2) [label=below:$$] {};
			\draw (-1.1,0) node (1f3) [label=below:$$] {};

			\tikzstyle{every node}=[draw,circle,fill=black,minimum size=5pt,scale=0.3,
			inner sep=0pt]

			\draw (-1.6,1) node (1s1) [label=below:$$] {};
			
			\draw (1f1) -- (1s1) -- (1f2);
			\draw (1f3) -- (1s1);
	\end{tikzpicture}},\hspace{-0.08cm}
	\raisebox{-1ex}{\begin{tikzpicture}[scale=0.3]
			\tikzstyle{every node}=[draw,circle,fill=black,minimum size=10pt,scale=0.3,
			inner sep=0pt]

			\draw (1.5,0) node (3f1) [label=below:$$] {};
			\draw (1.5,1) node (3f2) [label=below:$$] {};

			\tikzstyle{every node}=[draw,circle,fill=black,minimum size=5pt,scale=0.3,
			inner sep=0pt]

			\draw (1,0.5) node (3s1) [label=below:$$] {};
			\draw (2,0.5) node (3s2) [label=below:$$] {};
			
			\draw (3s1) -- (3s2);
			\draw (3f1) -- (3s1) -- (3f2) -- (3s2) -- (3f1);
	\end{tikzpicture}}\big\}-line Hoffman graph.

As an application of Theorem \ref{intro2}, we give a characterization for the structure of local graphs for some distance-regular graphs with classical parameters.  
In particular, Theorem \ref{a5} is derived from Corollary \ref{incor}.

\begin{cor}\label{incor}
	Let $\Gamma$ be a distance-regular graph with classical parameters $(D,b,\alpha,\beta)$ such that $b=2$ and $\alpha\neq 0$.
	 Let $c:=3\alpha+2$, $\tilde{c}:=\min\{c,6\}$ and $q:=\max\{c+5, 50\tilde{c}+16\}$.
	 For each vertex $x$, the local graph $\Gamma_x$ at $x$  is $\mu$-bounded with parameter $c $ and the associated Hoffman graph $\go(\Gamma_x,q)$ is a $2$-fat \big\{\raisebox{-1ex}{\begin{tikzpicture}[scale=0.3]
	 		
	 		\tikzstyle{every node}=[draw,circle,fill=black,minimum size=10pt,scale=0.3,
	 		inner sep=0pt]
	 		
	 		\draw (-2.1,0) node (1f1) [label=below:$$] {};
	 		\draw (-1.6,0) node (1f2) [label=below:$$] {};
	 		\draw (-1.1,0) node (1f3) [label=below:$$] {};

	 		\tikzstyle{every node}=[draw,circle,fill=black,minimum size=5pt,scale=0.3,
	 		inner sep=0pt]

	 		\draw (-1.6,1) node (1s1) [label=below:$$] {};
	 		
	 		\draw (1f1) -- (1s1) -- (1f2);
	 		\draw (1f3) -- (1s1);
	 \end{tikzpicture}},\hspace{-0.08cm}
	 \raisebox{-1ex}{\begin{tikzpicture}[scale=0.3]
	 		\tikzstyle{every node}=[draw,circle,fill=black,minimum size=10pt,scale=0.3,
	 		inner sep=0pt]

	 		\draw (1.5,0) node (3f1) [label=below:$$] {};
	 		\draw (1.5,1) node (3f2) [label=below:$$] {};

	 		\tikzstyle{every node}=[draw,circle,fill=black,minimum size=5pt,scale=0.3,
	 		inner sep=0pt]

	 		\draw (1,0.5) node (3s1) [label=below:$$] {};
	 		\draw (2,0.5) node (3s2) [label=below:$$] {};
	 		
	 		\draw (3s1) -- (3s2);
	 		\draw (3f1) -- (3s1) -- (3f2) -- (3s2) -- (3f1);
	 \end{tikzpicture}}\big\}-line Hoffman graph if $\alpha(2^D-2)\geq \max\{36c+400\tilde{c}+84, 44c-4 \}$.
\end{cor}

\begin{remark}
If $\Gamma$ is a distance-regular graph with classical parameters $(D, 2, 2, 2^{D+2}-2)$, then the above result {provides structural insight} for any local graph of $\Gamma$.
On this moment there are two infinite families known with these parameters, namely the Grassmann graphs $J_2(2D+1, D)$ and the twisted Grassmann graphs 
$\tilde{J}_2(2D+1, D)$, but we do not know whether these are the only distance-regular graphs with these parameters for large $D$.
\end{remark}

%Using the same method as above we can also show a bound on $\alpha$ for given $b \geq 2$, similar to the bound given in Koolen et al. \cite{KLPY}. The main advantage of the 
%current method is that we have a reasonable bound on $D$. 

This paper is organized as follows. In Section 2, we give definitions and basic theory of Hoffman graphs. We also give some results about the forbidden matrices introduced by Koolen et al. \cite{KYY19}.
In Section 3, we give a theory for the associated Hoffman graph of $\mu$-bounded graphs. 
This is a simplified version of the original method of Hoffman as given in \cite{KKY}.
We also show some properties for associated Hoffman graphs.
In Section 4, we firstly give the crucial Lemma \ref{blhlem} with an argument inspired by the Bose-Laskar method. After that, we give a proof for the main result Theorem \ref{intro2}. In Section \ref{aDRG}, we give some applications on distance-regular graphs with classical parameters and prove Corollary \ref{incor}, Theorem \ref{a5},
and Theorem \ref{alphab}. 
We also give a proof of Theorem \ref{introalpha}. This will be a consequence of Theorem \ref{beta}, whose proof will be also given in that section. 

\section{Definitions and preliminaries}\label{2}

\subsection{Distance-regular graphs}\label{DRGc}
A connected graph $\Gamma$ with diameter $D$ is called {\em{distance-regular}} if for $0\leq i\leq D$ there are integers $b_i$ and $c_i$  such that for every pair of vertices $x, y \in V(\Gamma)$ with $d(x, y)=i$, among the neighbors of $y$,
there are precisely $c_i$ (resp. $b_i$) at distance $i-1$ (resp. $i+1$) from $x$.
This pivotal concept led to significant results in graph theory and algebraic
combinatorics \cite{BCN}.

 In this paper, we denote a distance-regular graph by $\Gamma$.
Every distance-regular graph  is regular with degree $k= b_0$.  We define $a_i := k-b_i-c_i$ for notational convenience.   Note that local graphs of $\Gamma$ are $a_1$-regular and $\mu$-bounded with parameter $c_2-1$ on $b_0$ vertices.
 The numbers $c_i, a_i$, and $b_i$ $(0\leq i\leq D)$ are called the {\em intersection numbers} of $\Gamma$ with the following property.

\begin{pro}[{\cite[Theorem 4.1.7]{BCN}}]\label{417}
	If $i$ and $h$ are positive integers such that $i+h\leq D$, then
	$p^{i+h}_{ih}:=\frac{c_{i+1}\cdots c_{i+h}}{c_1\cdots c_h}$ is a non-negative {integer}. 
\end{pro}

Here we introduce the Delsarte bound for distance-regular graphs, which gives an upper bound on the order of  cliques. 
\begin{lem}[{\cite[Proposition 4.4.6]{BCN}}]\label{dels}
	Let $\Gamma$ be a distance-regular graph with degree $k$ and smallest eigenvalue $\lambda_{\min}(\Gamma)$.
	If $C$ is a clique in $\Gamma$, then the order of the clique {satisfies} $$|V(C)| \leq 1 + \frac{k}{-\lambda_{\min}(\Gamma)}.$$
\end{lem}

The following theorem is a well-known result for the smallest eigenvalue of local graphs of distance-regular graphs, and is due to Terwilliger \cite{Ternotes}.

\begin{thm}[{\cite[Theorem 4.4.3]{BCN}}]\label{bcnth}
	If $\Gamma$ is a distance-regular graph of diameter $D\geq 3$ with second largest eigenvalue $\theta$, then $b^+:=\frac{b_1}{\theta+1}>0$, and the local graph of $\Gamma$ at each vertex has smallest eigenvalue at least $-1-b^+$.
\end{thm}

\begin{de}
	A distance-regular graph $\Gamma$ of diameter $D$ has \emph{classical parameters} $(D, b, \alpha, \beta)$ if the intersection numbers of $\Gamma$ satisfy 
		\begin{equation}\label{drgc}
		c_i=\Big[\substack{i\\ \\ 1}\Big]_b(1+\alpha\Big[\substack{i-1\\ \\ 1}\Big]_b)
	\end{equation} and
	\begin{equation}\label{drgb}
		b_i=(\Big[\substack{D\\ \\ 1}\Big]_b-\Big[\substack{i\\ \\ 1}\Big]_b)(\beta-\alpha\Big[\substack{i\\ \\ 1}\Big]_b)
	\end{equation} for $0\leq i\leq D$, where $\Big[\substack{i\\ \\ 1}\Big]_b=b^{i-1}+b^{i-2}+\dots+b+1$.
\end{de}
	There are many examples of distance-regular graphs with classical parameters such as the Johnson graphs, the Hamming graphs, the Grassmann graphs, the bilinear forms graphs and so on. For more comprehensive background on distance-regular graphs, we refer readers to \cite{BCN,DKT2016}.

 Brouwer et al. \cite[Proposition 6.2.1]{BCN} showed that the parameter $b$ is an integer such that $b\notin\{0,-1\}$. Weng \cite{Weng99} classified the distance-regular graphs with classical parameters $(D,b,\alpha,\beta)$ for $b\leq-2$. Terwilliger, referred to \cite[Theorem 6.1.1]{BCN}, has classified the distance-regular graphs with classical parameters $(D, b, \alpha, \beta)$, where $b=1$.
 Therefore, we focus on $b\geq2$ in this paper.

 The following lemma  gives explicit formulas for eigenvalues of distance-regular graphs with classical parameters. 
 
 \begin{lem}[{\cite[Corollary 8.4.2]{BCN}}]\label{ei}
 	The eigenvalues of the distance-regular graph $\Gamma$ with classical parameters $(D, b, \alpha, \beta)$ are
 	$\lambda_i:=\Big[\substack{D-i\\ \\ 1}\Big]_b(\beta-\alpha\Big[\substack{i\\ \\ 1}\Big]_b)-\Big[\substack{i\\ \\ 1}\Big]_b=\frac{b_i}{b^i}-\Big[\substack{i\\ \\ 1}\Big]_b$ for $0\leq i\leq D$.
 \end{lem}
 Note that, if $b\geq 1$, then these eigenvalues have natural ordering $k=\lambda_0>\lambda_1>\dots>\lambda_D$.
Lemma \ref{dels} and Lemma \ref{ei} imply that the Delsarte bound for distance-regular graphs with classical parameters is $1+\beta$ if $b\geq 1$. We also have
\begin{equation}\label{ie1}
	 \beta-1\geq\alpha\Big[\substack{D-1\\ \\ 1}\Big]_b,
\end{equation}
 since $a_D=k-b_D-c_D=\Big[\substack{D\\ \\ 1}\Big]_b\beta-0-\Big[\substack{D\\ \\ 1}\Big]_b(1+\alpha\Big[\substack{D-1\\ \\ 1}\Big]_b)\geq0$.	%\Big[\substack{D\\ \\ 1}\Big]_b\beta-0-\Big[\substack{D\\ \\ 1}\Big]_b(1+\alpha\Big[\substack{D-1\\ \\ 1}\Big]_b)\geq0$.

\subsection{Hoffman graphs and special matrices}\label{hoffsec1}
Now we give definitions and preliminaries of Hoffman graphs. For more details, see \cite{JKMT14,KKY,KYY19,WN1995}.
\begin{de}\label{hoffmangraph}
	A \emph{Hoffman graph} $\ho$ is a pair of $(H, \ell)$, where $H=(V,E)$ is a graph and $\ell:V\rightarrow\{{\rm fat},{\rm slim}\}$ is a labeling map on $V$, such that no two vertices with label \emph{fat} are adjacent and every vertex with label \emph{fat} has at least one neighbor with label \emph{slim}.
\end{de}

The vertices with label {fat} (resp. {slim}) are called \emph{fat vertices} (resp. \emph{slim vertices}). We denote the set of slim (resp. fat) vertices by $V_{{\rm slim}}(\ho)$ (resp.  $V_ {\rm fat}(\ho)$). For a vertex $x$ of $\mathfrak{h}$, the set of slim (resp. fat) neighbors of $x$ in $\mathfrak{h}$ is denoted by $N_{\mathfrak{h}}^{{\rm slim}}(x)$ (resp. $N_{\mathfrak{h}}^{{\rm fat}}(x)$).
A Hoffman graph is called \emph{$t$-fat} if every slim vertex has at least $t$ fat neighbors.
% Especially, a \emph{fat} Hoffman graph is a $1$-fat Hoffman graph.
The \emph{slim graph} of $\ho$ is the subgraph of $H$ induced by $V_{\rm slim}(\ho)$. Note that the slim graph of a Hoffman graph is an ordinary graph.

\begin{de}
 A Hoffman graph $\ho_1= (H_1, \ell_1)$ is called an \emph{induced Hoffman subgraph} of $\ho= (H, \ell)$, if $H_1$ is an induced subgraph of $H$ and $\ell_1(x) = \ell(x)$ for all vertices $x$ of $H_1$.
		%\begin{enumerate}	
	%	\item Let $W$ be a subset of $V_{\rm slim}(\ho)$. The {\bf induced Hoffman subgraph of $\ho$  generated by $W$}, denoted by $\langle W\rangle_{\ho}$, is the Hoffman subgraph of $\ho$ induced by $W \cup \{f \in V_{\rm fat}(\ho)$ $|$ $f \sim w$ for some $w \in W \}$.
	%\end{enumerate}
\end{de}

\begin{de}Two Hoffman graphs $\ho= (H, \ell)$ and $\ho^\prime=(H^\prime, \ell^\prime)$ are \emph{isomorphic} if there exists an isomorphism from $H$ to $H^\prime$ which preserves the labeling.
\end{de}

For a Hoffman graph $\ho= (H, \ell)$, there exists a matrix $D$ such that the adjacency matrix $A$ of $H$ satisfies 	$A= \footnotesize \begin{pmatrix}A_{\rm slim}  & D^T\\
	D  & O	\end{pmatrix}$,
where $A_{\rm slim}$ is the adjacency matrix of the slim graph of $\mathfrak{h}$. The  symmetric matrix $S(\ho):=A_{\rm slim}-D^TD$ is called the \emph{special matrix} of $\ho$. The
{eigenvalues of $\ho$}  are the eigenvalues of $S(\ho)$, and the smallest eigenvalue of $\ho$ is denoted by $\lambda_{\min}(\ho)$.
Woo and Neumaier\cite[Corollary 3.3]{WN1995} showed that $\lambda_{\min}(\go)\geq\lambda_{\min}(\ho)$ if $\go$ is an induced Hoffman subgraph of $\ho$.

 The $t$-fat Hoffman graph with one slim vertex and $t$ fat vertices, 
denoted by $\ho^{(t)}$,
is the unique Hoffman graph with the special matrix $(-t)$.
However, in general Hoffman graphs are not determined by their special matrices. For example,
\raisebox{-1ex}{\begin{tikzpicture}[scale=0.3]
		
		\tikzstyle{every node}=[draw,circle,fill=black,minimum size=10pt,scale=0.3,
		inner sep=0pt]
		
		\draw (-2.2,0) node (1f1) [label=below:$$] {};
		\draw (-1.6,0) node (1f2) [label=below:$$] {};
		\draw (-1.0,0) node (1f3) [label=below:$$] {};

		\tikzstyle{every node}=[draw,circle,fill=black,minimum size=5pt,scale=0.3,
		inner sep=0pt]

		\draw (-1.9,1) node (1s1) [label=below:$$] {};
		\draw (-1.3,1) node (1s2) [label=below:$$] {};
		
		\draw (1f1) -- (1s1) -- (1f2);
		\draw (1f2) -- (1s2) -- (1f3);
\end{tikzpicture}} and \hspace{-0.08cm}
\raisebox{-1ex}{\begin{tikzpicture}[scale=0.3]
		\tikzstyle{every node}=[draw,circle,fill=black,minimum size=10pt,scale=0.3,
		inner sep=0pt]

		\draw (1.5,0) node (3f1) [label=below:$$] {};
		\draw (1.5,1) node (3f2) [label=below:$$] {};

		\tikzstyle{every node}=[draw,circle,fill=black,minimum size=5pt,scale=0.3,
		inner sep=0pt]

		\draw (1,0.5) node (3s1) [label=below:$$] {};
		\draw (2,0.5) node (3s2) [label=below:$$] {};
		
		\draw (3s1) -- (3s2);
		\draw (3f1) -- (3s1) -- (3f2) -- (3s2) -- (3f1);
\end{tikzpicture}}
are non-isomorphic Hoffman graphs with the same special matrix.

%For the smallest eigenvalue of a Hoffman graph $\ho$ and its induced Hoffman subgraph $\go$,

%\begin{lem}\label{interlacing}
%	If $\go$ is an induced Hoffman subgraph of a Hoffman graph $\ho$,
%	then $\lambda_{\min}(\go)\geq\lambda_{\min}(\ho)$.
%\end{lem}

Let $\ho$ be a Hoffman graph.
Let $G(\ho,p)$ denote the graph obtained from $\ho$ by replacing fat vertex $f$ by a slim $p$-clique $K^f$ and joining all the neighbors of $f$ in $\ho$ with all the vertices in $K^f$ for each fat vertex $f$.
For the relation between  $\lambda_{\min}(\ho)$ and $\lambda_{\min}(G(\ho,p))$, Hoffman and Ostrowski showed the following result. For a proof, we refer to \cite[Theorem 2.14]{JKMT14} and \cite[Theorem 3.4]{KYY18}.
\begin{thm}\label{Ostrowski} 
	If  $\ho$ is a Hoffman graph and $p$ a positive integer, then
	$\lambda_{\min}(G(\ho,p))\geq \lambda_{\min}(\ho)$
	and
	$\lim_{p \rightarrow \infty}\lambda_{\min}(G(\ho,p))= \lambda_{\min}(\ho).$
	
\end{thm}

\subsection{Decompositions of Hoffman graphs}\label{hoffsec2}
Now, we introduce decompositions of Hoffman graphs and line Hoffman graphs.

For two vertices $x$ and $y$, we denote by $x\sim y$ if they are adjacent.
Let $W$ be a subset of $V_{\rm slim}(\ho)$.
 The {\em induced Hoffman subgraph of $\ho$  generated by $W$} is the Hoffman subgraph of $\ho$ induced by $W \cup \{f \in V_{\rm fat}(\ho)$ $|$ $f \sim w$ for some $w \in W \}$.

\begin{de}\label{directsummatrix}
	A family of Hoffman graphs $\{\ho^i\}_{i=1}^r$ is a \emph{decomposition} of a Hoffman graph $\ho$, if there exists a partition $\big\{V_{\rm slim}^1(\ho),V_{\rm slim}^2(\ho),\ldots,V_{\rm slim}^r(\ho) \big\}$ of $V_{\rm slim}(\ho)$ such that  the induced Hoffman subgraph generated by $V_{\rm slim}^i(\ho)$ is $\ho^i$ for $i\in[r]$ and
	$$S(\ho)=
	\begin{pmatrix}
		S(\ho^1) &&&\\
		& S(\ho^2)&&\\
		&&\ddots&\\
		&&&S(\ho^r)
	\end{pmatrix}
	$$ is a block diagonal matrix with respect to this partition of $V_{\rm slim}(\ho)$.
\end{de}

If a Hoffman graph $\ho$ has a decomposition $\{\ho^i\}_{i=1}^r$, then we write $\ho=\uplus_{i=1}^r\ho^i$.
A Hoffman graph $\ho$ is said to be \emph{decomposable}, if $\ho$ has a decomposition $\{\ho^i\}_{i=1}^r$ with $r\geq2$. Otherwise, $\ho$ is called \emph{indecomposable}. Note that $\ho$ is decomposable if and only if its special matrix $S(\ho)$ is a block diagonal matrix with at least two blocks.
Here we give a combinatorial way to describe the decomposability of Hoffman graphs.

\begin{lem}[{\cite[Lemma 2.11]{KYY19}}]\label{combi}                                                                   
	If $\ho$ is a Hoffman graph and $\{\ho^i\}_{i=1}^r$ a family of induced Hoffman subgraphs of $\ho$, then $\{\ho^i\}_{i=1}^r$ is a decomposition of $\ho$
	, if it satisfies the following conditions:
	\begin{enumerate}
		\item $V(\ho)=\bigcup_{i=1}^rV(\ho^i)$,
		\item $V_{\rm slim}(\ho^i)\cap V_{\rm slim}(\ho^j)=\emptyset$ for $i\neq j$,
		\item if $x \in V_{\rm slim}(\ho^i),~f \in V_{\rm fat}(\ho)$ and $x\sim f$, then $f\in V_{\rm fat}(\ho^i)$,
		\item if $x \in V_{\rm slim}(\ho^i)$, $y \in V_{\rm slim}(\ho^j)$ and $i\neq j$, then $x$ and $y$ have at most one common fat neighbor, and they have one if and only if they are adjacent.
	\end{enumerate}
\end{lem}

\begin{de}\label{line}
	Let $\Go$ be a family of pairwise non-isomorphic Hoffman graphs. A Hoffman graph $\ho$ is called a \emph{$\Go$-line Hoffman graph} if there exists a Hoffman graph $\ho^\prime$ satisfying the following conditions:
	\begin{enumerate}
		\item $\ho^\prime$ has $\ho$ as an induced Hoffman subgraph,
		\item $\ho^\prime$ has the same slim graph as $\ho$,
		\item $\ho^\prime=\uplus_{i=1}^{r}\ho_i^\prime$, where $\ho_i^\prime$ is isomorphic to an induced Hoffman subgraph of a Hoffman graph in $\Go$ for $i\in[r]$.
	\end{enumerate}
\end{de}

%We denote it by $\mathfrak{h}^{(t)}$.

\subsection{Forbidden matrices}\label{Forbiddenmatrice}

In this subsection, we introduce a structural theory for $t$-fat Hoffman graphs with smallest eigenvalue at least $-t-1$. For more details, we refer readers to  \cite{KYY19}

% A symmetric matrix is said to be \emph{reducible} if and only if it is equivalent to a block diagonal matrix with at least two blocks. If a symmetric matrix is not reducible, it is said to be \emph{irreducible}.

%	Let $t$ and $a$ be integers such that $t>0$. 
%	Let $I_n$ be the identity matrix of size $n$. Let $J_n$ be the all-one matrix of size $n$.
%	Let $m_{1,a}(t):=(-t+a)I_1$, $m_{2,a}(t):=aJ_2-(t+a)I_2$, $m_{3,a}(t):=\footnotesize \begin{pmatrix}
%		-t-1 & a \\
%		a & -t
%	\end{pmatrix}$,
%	$m_{5}(t):=\footnotesize\begin{pmatrix}-t & -1 & -1\\-1 & -t & -1 \\-1 & -1 & -t\end{pmatrix}$,
%	$m_{6}(t):=\footnotesize\begin{pmatrix}-t & 1 & 1\\1 & -t & -1 \\1 & -1 & -t\end{pmatrix}$,
%	$m_7(t):=\footnotesize\begin{pmatrix}-t & 0 & 1\\0 & -t & -1 \\1 & -1 & -t\end{pmatrix}$,
%	$m_8(t):=\footnotesize\begin{pmatrix}-t & 0 & 1\\0 & -t & 1 \\1 & 1 & -t\end{pmatrix}$, 
%	$m_8(t):= J_3-m_6(t)+m_7(t)-(t+1)I_3$
%	$m_9(t):= -m_8(t)-2tI_3$ and
%	$m_9(t):=\footnotesize\begin{pmatrix}-t & 0 & -1\\0 & -t & -1 \\-1 & -1 & -t\end{pmatrix}$
%	$m_{5}(t):=-(t-1)I_3-J_3$.

		Let $t$ and $a$ be integers, where $t>0$. 
	Let $m_{1,a}(t):=(-t+a)$, 
	$m_{2,a}(t):= \footnotesize \begin{pmatrix}-t & a \\ a & -t	\end{pmatrix}$, 
	$m_{3,a}(t):=\footnotesize \begin{pmatrix}	-t-1 & a \\	a & -t	\end{pmatrix}$,
	$m_{4,a}(t):= \footnotesize \begin{pmatrix}-t-1 & a \\ a & -t-1	\end{pmatrix}$, 
		$m_{5}(t):=\footnotesize\begin{pmatrix}-t & -1 & -1\\-1 & -t & -1 \\-1 & -1 & -t\end{pmatrix}$,
	$m_{6}(t):=\footnotesize\begin{pmatrix}-t & 1 & 1\\1 & -t & -1 \\1 & -1 & -t\end{pmatrix}$,
	$m_7(t):=\footnotesize\begin{pmatrix}-t & 0 & 1\\0 & -t & -1 \\1 & -1 & -t\end{pmatrix}$,
		$m_8(t):=\footnotesize\begin{pmatrix}-t & 0 & 1\\0 & -t & 1 \\1 & 1 & -t\end{pmatrix}$, and
%	$m_8(t):= J_3-m_6(t)+m_7(t)-(t+1)I_3$
%	$m_9(t):= -m_8(t)-2tI_3$ and
		$m_9(t):=\footnotesize\begin{pmatrix}-t & 0 & -1\\0 & -t & -1 \\-1 & -1 & -t\end{pmatrix}$.
%	$m_{5}(t):=-(t-1)I_3-J_3$.

%\begin{de}
%	Let $t$ and $a$ be integers such that $t>0$. The matrices $m_{1,a}(t)$, $m_{2,a}(t)$, $m_{3,a}(t)$, $m_{4,a}(t)$, $m_{5}(t)$, $m_{6}(t)$, $m_{7}(t)$, $m_{8}(t)$ and $m_{9}(t)$ are defined as follows:
%\vspace{-4pt}
%		\begin{equation*}
%			\begin{array}[c]{cccc}
				
%		m_{1,a}(t)= \begin{pmatrix} -t+a \end{pmatrix},&		
%		 m_{2,a}(t)=\begin{pmatrix} -t & a \\ a & -t \end{pmatrix},&
%		m_{3,a}(t)=\begin{pmatrix}	-t-1 & a \\	a & -t \end{pmatrix},&\\
		
%		\\

%		m_{4,a}(t)=\begin{pmatrix} -t-1 & a \\ a & -t-1 \end{pmatrix},&				
%				m_{5}(t)=\begin{pmatrix}-t & -1 & -1\\-1 & -t & -1 \\-1 & -1 & -t\end{pmatrix},&
%				m_{6}(t)=\begin{pmatrix}-t & 1 & 1\\1 & -t & -1 \\1 & -1 & -t\end{pmatrix},&\\
				
%				\\				
%				m_{7}(t)=\begin{pmatrix}-t & 0 & 1\\0 & -t & -1 \\1 & -1 & -t\end{pmatrix},&				
%				m_{8}(t)=\begin{pmatrix}-t & 0 & 1\\0 & -t & 1 \\1 & 1 & -t\end{pmatrix}, &
%				m_{9}(t)=\begin{pmatrix}-t & 0 & -1\\0 & -t & -1 \\-1 & -1 & -t\end{pmatrix}.\\
%			\end{array}
%		\end{equation*}

%\end{de}

%, and the Hoffman graph $\ho^{(t-a)}$ is the unique Hoffman graph with the special matrix $m_{1,a}(t)=(-t+a)$.

 The set of irreducible symmetric matrices $\mathcal{M}(t)$ is the union of the sets $M_1(t),M_2(t)$, and $M_3(t)$, where
		$M_1(t) :=\{m_{1,a}(t), m_{2,a}(t) \mid a=-2,-3,\ldots\}$,  
		$M_2(t) :=\{m_{3,a}(t), m_{4,a}(t) \mid a=1,-1,-2,\ldots\}$,
		$M_3(t) :=\{m_{5}(t),m_{6}(t),m_{7}(t),m_{8}(t),m_{9}(t)\}$.

%\begin{de}\label{F(t)} Let $t$ be a positive integer. The set of irreducible symmetric matrices $\mathcal{M}(t)$ is the union of the sets $M_1(t),M_2(t)$, $M_3(t)$, $M_4(t)$ and $M_5(t)$, where
%	\begin{equation*}
%		\begin{split}
%			M_1(t)& :=\{m_{1,a}(t) \mid a=-2,-3,\ldots\},    \\
%			M_2(t)& :=\{m_{2,a}(t) \mid a=-2,-3,\ldots,-t\}, \\
%			M_3(t)& :=\{m_{3,a}(t) \mid a=1,-1,-2,\ldots, -t\}, \\
%			M_4(t)& :=\{m_{4,a}(t) \mid a=1,-1,-2,\ldots, -t-1\}, \\
%			M_5(t)& :=\{m_{5}(t),m_{6}(t),m_{7}(t),m_{8}(t),m_{9}(t)\}.
%		\end{split}
%	\end{equation*}
%\end{de}

Note that $\lambda_{\min}(m_{1,a}(t))=-t+a,~\lambda_{\min}(m_{2,a}(t))=-t-|a|$, $ \lambda_{\min}(m_{3,a}(t))=-t-\frac{1+\sqrt{1+4a^2}}{2}$,  $\lambda_{\min}(m_{4,a}(t))=-t-1-|a|$, $\lambda_{\min}(m_{5}(t))=\lambda_{\min}(m_{6}(t))=-t-2$, and $\lambda_{\min}(m_{7}(t))=\lambda_{\min}(m_{8}(t))=\lambda_{\min}(m_{9}(t))=-t-\sqrt{2}$.
It means that the smallest eigenvalue of each matrix in $\mathcal{M}(t)$ is at most $-t-\sqrt{2}$. Therefore, the following result can be derived from Theorem \ref{Ostrowski}.

%Note that the smallest eigenvalue of each matrix in $\mathcal{M}(t)$ is at most $-t-\sqrt{2}$. Hence we have the following result by Theorem \ref{Ostrowski}.

\begin{pro}[{\cite[Proposition 3.3]{KYY19}}]\label{epsilon}
	Let $\ho$ be a Hoffman graph with special matrix in $\mathcal{M}(t)$.
	For fixed $\epsilon>0$, there exists a positive integer $p:=p(\epsilon,t)$ such that $\lambda_{\min}(G(\ho,p))< -t-\sqrt{2}+\epsilon$.
\end{pro}

To simplify the narrative, we introduce nine Hoffman graphs in the following.

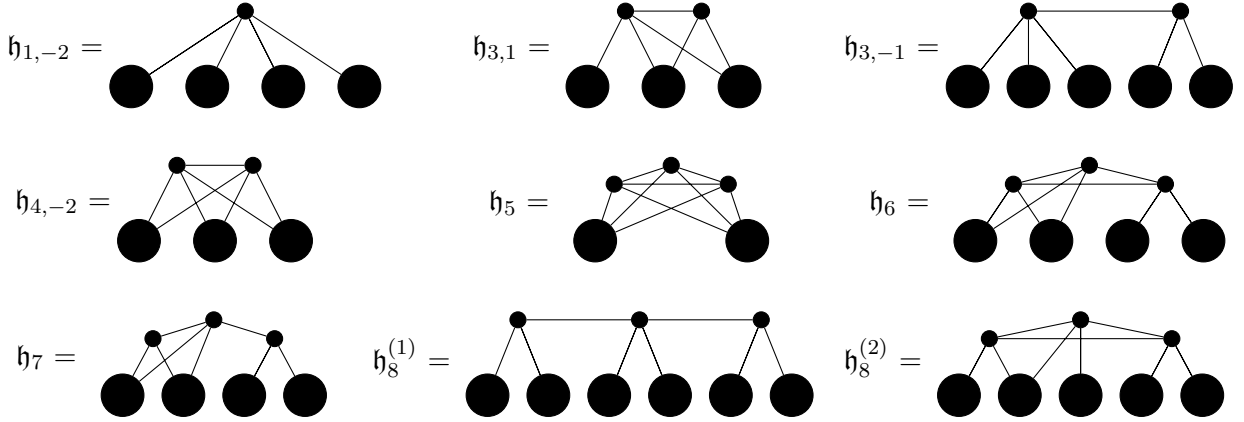
\begin{figure}[H]
	\centering
	\begin{tikzpicture}
		\draw (-9,5) node {$\mathfrak{h}_{1,-2}$ =};
		\draw (-3,5) node {$\mathfrak{h}_{3,1}$ =};
		\draw (2,5) node {$\mathfrak{h}_{3,-1}$ =};
		\tikzstyle{every node}=[draw,circle,fill=black,minimum size=16pt,
		inner sep=0pt]
		{every label}=[\tiny]
		\draw (-8,4.5) node (1f1) [label=below:$$] {};
		\draw (-7,4.5) node (1f2) [label=below:$$] {};
		\draw (-6, 4.5) node (1f3) [label=below:$$] {};
		\draw (-5,4.5) node (1f4) [label=below:$$] {};
		\draw (-2,4.5) node (2f1) [label=below:$$] {};
		\draw (-1,4.5) node (2f2) [label=below:$$] {};
		\draw (0,4.5) node (2f3) [label=below:$$] {};
		\draw (3,4.5) node (3f1) [label=below:$$] {};
		\draw (3.8,4.5) node (3f2) [label=below:$$] {};		
		\draw (4.6,4.5) node (3f3) [label=below:$$] {};		
		\draw (5.4,4.5) node (3f4) [label=below:$$] {};		
		\draw (6.2,4.5) node (3f5) [label=below:$$] {};

		\tikzstyle{every node}=[draw,circle,fill=black,minimum size=6pt,
		inner sep=0pt]
		{every label}=[\tiny]
		
		\draw (-6.5,5.5) node (1s1) [label=below:$$] {};
		\draw (-1.5,5.5) node (2s1) [label=below:$$] {};
		\draw (-0.5,5.5) node (2s2) [label=below:$$] {};
		\draw (3.8,5.5) node (3s1) [label=below:$$] {};
		\draw (5.8,5.5) node (3s2) [label=below:$$] {};
		
		\draw (1f2) -- (1s1) -- (1f1) -- (1s1) -- (1f3) -- (1s1) --(1f4);
		\draw (2f1) -- (2s1) -- (2f2) -- (2s2) --(2f3) -- (2s1) -- (2s2);
		\draw (3f2) -- (3s1) -- (3f1) -- (3s1) --(3f3) -- (3s1) -- (3s2) -- (3f4) -- (3s2) -- (3f5);
	\end{tikzpicture}
	\vspace{-0.8cm}
\end{figure}

\begin{figure}[H]
	\centering
	\begin{tikzpicture}
		\draw (-9,5) node {$\mathfrak{h}_{4,-2}$ =};
		\draw (-3,5) node {$\mathfrak{h}_{5}$ =};
		\draw (2,5) node {$\mathfrak{h}_{6}$ =};
		\tikzstyle{every node}=[draw,circle,fill=black,minimum size=16pt,
		inner sep=0pt]
		{every label}=[\tiny]
		\draw (-8,4.5) node (1f1) [label=below:$$] {};
		\draw (-7,4.5) node (1f2) [label=below:$$] {};
		\draw (-6, 4.5) node (1f3) [label=below:$$] {};
		\draw (-2,4.5) node (2f1) [label=below:$$] {};
		\draw (0,4.5) node (2f2) [label=below:$$] {};
		\draw (3,4.5) node (3f1) [label=below:$$] {};
		\draw (4,4.5) node (3f2) [label=below:$$] {};		
		\draw (5,4.5) node (3f3) [label=below:$$] {};		
		\draw (6,4.5) node (3f4) [label=below:$$] {};

		\tikzstyle{every node}=[draw,circle,fill=black,minimum size=6pt,
		inner sep=0pt]
		{every label}=[\tiny]
		
		\draw (-6.5,5.5) node (1s1) [label=below:$$] {};
		\draw (-7.5,5.5) node (1s2) [label=below:$$] {};
		\draw (-1.75,5.25) node (2s1) [label=below:$$] {};
		\draw (-0.25,5.25) node (2s2) [label=below:$$] {};
		\draw (-1,5.5) node (2s3) [label=below:$$] {};
		\draw (4.5,5.5) node (3s2) [label=below:$$] {};
		\draw (5.5,5.25) node (3s3) [label=below:$$] {};
		\draw (3.5,5.25) node (3s1) [label=below:$$] {};
		
		\draw (1f1) -- (1s1) -- (1f2) -- (1s2) -- (1f3) -- (1s1) --(1s2) -- (1f1);
		\draw (2f1) -- (2s1) -- (2f2) -- (2s2) --(2s3) -- (2s1) -- (2s2) -- (2f1) -- (2s3) -- (2f2);
		\draw (3f1) -- (3s2) -- (3f2) -- (3s1) -- (3f1) -- (3s1) --(3s3) -- (3f3) -- (3s3) -- (3f4) -- (3s3) -- (3s2) -- (3s1);
	\end{tikzpicture}
	\vspace{-0.8cm}
\end{figure}

\begin{figure}[H]
	\centering
	\begin{tikzpicture}
		\draw (-9,5) node {$\mathfrak{h}_{7}$ =};
		\draw (-4.2,5) node {$\mathfrak{h}_{8}^{(1)}$ =};
		\draw (2,5) node {$\mathfrak{h}_{8}^{(2)}$ =};
		\tikzstyle{every node}=[draw,circle,fill=black,minimum size=16pt,
		inner sep=0pt]
		{every label}=[\tiny]
		\draw (-8,4.5) node (1f1) [label=below:$$] {};
		\draw (-7.2,4.5) node (1f2) [label=below:$$] {};
		\draw (-6.4, 4.5) node (1f3) [label=below:$$] {};
		\draw (-5.6,4.5) node (1f4) [label=below:$$] {};
		\draw (-3.2,4.5) node (2f1) [label=below:$$] {};
		\draw (-2.4,4.5) node (2f2) [label=below:$$] {};
		\draw (-1.6,4.5) node (2f3) [label=below:$$] {};
		\draw (-0.8,4.5) node (2f4) [label=below:$$] {};
		\draw (0,4.5) node (2f5) [label=below:$$] {};
		\draw (0.8, 4.5) node (2f6) [label=below:$$] {};
		\draw (3,4.5) node (3f1) [label=below:$$] {};
		\draw (3.8,4.5) node (3f2) [label=below:$$] {};		
		\draw (4.6,4.5) node (3f3) [label=below:$$] {};		
		\draw (5.4,4.5) node (3f4) [label=below:$$] {};		
		\draw (6.2,4.5) node (3f5) [label=below:$$] {};

		\tikzstyle{every node}=[draw,circle,fill=black,minimum size=6pt,
		inner sep=0pt]
		{every label}=[\tiny]
		
		\draw (-6.8,5.5) node (1s2) [label=below:$$] {};
		\draw (-7.6,5.25) node (1s1) [label=below:$$] {};
		\draw (-6,5.25) node (1s3) [label=below:$$] {};
		\draw (-2.8,5.5) node (2s1) [label=below:$$] {};
		\draw (-1.2,5.5) node (2s2) [label=below:$$] {};
		\draw (0.4,5.5) node (2s3) [label=below:$$] {};
		\draw (3.4,5.25) node (3s1) [label=below:$$] {};
		\draw (5.8,5.25) node (3s3) [label=below:$$] {};
		\draw (4.6,5.5) node (3s2) [label=below:$$] {};
		
		\draw (1s1) -- (1s2) -- (1f2) -- (1s1) -- (1f1) -- (1s2) -- (1s3) -- (1f3) -- (1s3) --(1f4);
		\draw (2f1) -- (2s1) -- (2f2) -- (2s1) --(2s2) -- (2f3) -- (2s2) -- (2f4) -- (2s2) -- (2s3) -- (2f5) -- (2s3) -- (2f6);
		\draw (3s2) -- (3f2) -- (3s1) -- (3f1) -- (3s1) --(3s3) -- (3f4) -- (3s3) -- (3f5) -- (3s3) -- (3s2) -- (3f3) -- (3s2) --(3s1);
	\end{tikzpicture}
	\vspace{-0.8cm}
\end{figure}
We denote by $\Ho:=\{\ho_{1,-2},\ho_{3,1},\ho_{3,-1},\ho_{4,-2},\ho_{5},\ho_6,\ho_7,\ho_{8}^{(1)},\ho_8^{(2)}\}$.

\begin{lem}\label{m2}
Let $\ho$ be a Hoffman graph in which two slim vertices are adjacent if they share at least one fat common neighbor.
If $\ho$ contains an induced Hoffman subgraph with special matrix in $\mathcal{M}(2)$, then it must contain a member of $\Ho$ as an induced Hoffman subgraph.
\end{lem}
\begin{proof}
 It can be confirmed by checking the {possibilities} of Hoffman graphs with special matrices in $\{m_{1,-2}(2)$, $m_{2,-2}(2)$  
	$m_{3,1}(2)$,$m_{3,-1}(2)$, $m_{3,-2}(2)$, $m_{4,1}(2)$,$m_{4,-1}(2)$, $m_{4,-2}(2)$, $m_{4,-3}(2)$,
	$m_{5}(2),m_{6}(2),m_{7}(2),m_{8}(2),m_{9}(2)\}$.
\end{proof}

By calculation, we have the following proposition.
\begin{pro}\label{cal}
	The following inequalities hold:
\begin{equation*}
	\begin{array}[c]{cccc}
		\lambda_{\min}(G(\ho_{1,-2},7))<-3,&
		\lambda_{\min}(G(\ho_{3,1},7))<-3,&
		\lambda_{\min}(G(\ho_{3,-1},13))<-3,&\\

	\lambda_{\min}(G(\ho_{4,-2},5))<-3,&
	\lambda_{\min}(G(\ho_{5},11))<-3,&
	\lambda_{\min}(G(\ho_{6},5))<-3,&\\

		\lambda_{\min}(G(\ho_{7},15))<-3,&
		\lambda_{\min}(G(\ho_{8}^{(1)},8))<-3,&
		\lambda_{\min}(G(\ho_{8}^{(2)},11))<-3.\\		
	\end{array}
\end{equation*}
\end{pro}
\begin{proof}
	By computer verification.
\end{proof}

Let $\ho_{(s)}$ the $2$-fat Hoffman graph with $s$ slim vertices and two fat vertices such that its slim graph is a complete graph, and let $\ho^{(s)}_{(2)}$ be the Hoffman graph consisting of two adjacent slim vertices, one of which is adjacent to exactly $s$ fat vertices, while the other has no fat neighbors.

\begin{pro}\label{215}
	\begin{enumerate}
	\item If $p_1:=s(s-1)+1$, then $\lambda_{\min}(G(\ho^{(s+1)},p_1))<-s$.
	\item If  $p_2:=(s-1)(2s-1)+1$, then $  \lambda_{\min}(G(\ho_{(s)},p_2))<-s$.
	\item {If $p_3=(s+1)(s-1)^2+1$, then $\lambda_{\min}(G(\ho^{(s)}_{(2)},p_3))<-s$.}
	\end{enumerate}
%	If $p_1:=s(s-1)+1$ and $p_2:=(s-1)(2s-1)+1$,
%	then $\lambda_{\min}(G(\ho^{(s+1)},p_1))<-s$ and $  \lambda_{\min}(G(\ho_{(s)},p_2))<-s$ both hold.
\end{pro}
\begin{proof}
	The graph $G(\ho^{(s+1)},p_1)$ consists of a vertex $u$ and $s+1$ pairwise disjoint cliques $C_1,\ldots,C_{s+1}$.
	Note that $\{\{u\},\bigcup_{i=1}^{s+1}V(C_i)\}$ is an equitable partition of $G(\ho^{(s+1)},p_1)$ with quotient matrix  $A_1=\footnotesize\begin{pmatrix}
		0 & (s+1)p_1 \\
		1 & p_1-1
	\end{pmatrix}$
	(for  the notations of equitable partitions and quotient matrices, see \cite[Chapter 9.3]{GD01}). The smallest eigenvalue of $A_1$ is $\frac{1}{2}(p_1-1-\sqrt{(p_1+1)^2+4sp_1})$, which is less than $-s$. Thus,
	$\lambda_{\min}(G(\ho^{(s+1)},p_1))<-s$ holds.
	
	The graph $G(\ho_{(s)},p_2)$ consists of $s$ vertices $u_1,\ldots,u_s$ and two disjoint cliques $C_1,C_2$.
	Note that $\{\{u_i\}_{i=1}^s,V(C_1)\cup V(C_2)\}$ is an equitable partition of
	$G(\ho_{(s)},p_2)$ with quotient matrix  $A_2=\footnotesize\begin{pmatrix}
		s-1 & 2p_2 \\
		s & p_2-1
	\end{pmatrix}$. 
	The smallest eigenvalue of $A_1$ is $\frac{1}{2}(s+p_2-2-\sqrt{(s+p_2)^2+4p_2s})$, which is less than $-s$. Thus, $\lambda_{\min}(G(\ho_{(s)},p_2))<-s$ holds.
	
{The graph $G(\ho^{(s)}_{(2)},p_3)$ consists of two vertices $u,v$ and $s$ pairwise disjoint cliques $C_1,\ldots,C_{s}$.
	Note that $\{\{u\},\{v\},\bigcup_{i=1}^{s}V(C_i)\}$ is an equitable partition of $G(\ho^{(s)}_{(2)},p_3)$ with quotient matrix  $A_3=\footnotesize\begin{pmatrix}
		0 & 1 & sp_3 \\
		1 & 0 & 0 \\
		1 & 0 & p_3-1
	\end{pmatrix}.$
It follows that $\det(A_3+sI)=(s+1)(s-1)^2-p_3<0$. Thus,
	$\lambda_{\min}(G(\ho^{(s+1)},p_1))<-s$ holds.}
\end{proof}

%\begin{de}\label{G(t)}
	Let $t$ be a positive integer.  We denote by $\Go(t)$ the family of pairwise non-isomorphic indecomposable $t$-fat Hoffman graphs whose special matrix is either  $\footnotesize\begin{pmatrix}
		J_{r_1} & -J_{r_1\times r_2} \\
		-J_{r_2\times r_1} & J_{r_2}
	\end{pmatrix}-(t+1)I$
	or $(-t-1)$, where $ r_1,r_2\in[t]$.
%\end{de}
Note that $\Go(t)$ is a finite family of Hoffman graphs.

\begin{lem}\label{b2linehoff}
	A $\Go(2)$-line Hoffman graph is a \big\{\raisebox{-1ex}{\begin{tikzpicture}[scale=0.3]
			
			\tikzstyle{every node}=[draw,circle,fill=black,minimum size=10pt,scale=0.3,
			inner sep=0pt]
			
			\draw (-2.1,0) node (1f1) [label=below:$$] {};
			\draw (-1.6,0) node (1f2) [label=below:$$] {};
			\draw (-1.1,0) node (1f3) [label=below:$$] {};

			\tikzstyle{every node}=[draw,circle,fill=black,minimum size=5pt,scale=0.3,
			inner sep=0pt]

			\draw (-1.6,1) node (1s1) [label=below:$$] {};
			
			\draw (1f1) -- (1s1) -- (1f2);
			\draw (1f3) -- (1s1);
	\end{tikzpicture}},\hspace{-0.08cm}
	\raisebox{-1ex}{\begin{tikzpicture}[scale=0.3]
			\tikzstyle{every node}=[draw,circle,fill=black,minimum size=10pt,scale=0.3,
			inner sep=0pt]

			\draw (-0.5,0) node (2f1) [label=below:$$] {};
			\draw (0.5,0) node (2f2) [label=below:$$] {};
			\draw (-0.5,1) node (2f3) [label=below:$$] {};
			\draw (0.5,1) node (2f4) [label=below:$$] {};

			\tikzstyle{every node}=[draw,circle,fill=black,minimum size=5pt,scale=0.3,
			inner sep=0pt]

			\draw (0,0.2) node (2s1) [label=below:$$] {};
			\draw (0.3,0.5) node (2s2) [label=below:$$] {};
			\draw (-0.3,0.5) node (2s3) [label=below:$$] {};
			\draw (0,0.8) node (2s4) [label=below:$$] {};

			\draw (2f1) -- (2s1) -- (2f2) -- (2s2) -- (2f4) -- (2s4) -- (2f3) -- (2s3) -- (2f1);
			\draw (2s1) -- (2s4);
			\draw (2s2) -- (2s3);
	\end{tikzpicture}},\hspace{-0.08cm}
	\raisebox{-1ex}{\begin{tikzpicture}[scale=0.3]
			\tikzstyle{every node}=[draw,circle,fill=black,minimum size=10pt,scale=0.3,
			inner sep=0pt]

			\draw (1.5,0) node (3f1) [label=below:$$] {};
			\draw (1.5,1) node (3f2) [label=below:$$] {};

			\tikzstyle{every node}=[draw,circle,fill=black,minimum size=5pt,scale=0.3,
			inner sep=0pt]

			\draw (1,0.5) node (3s1) [label=below:$$] {};
			\draw (2,0.5) node (3s2) [label=below:$$] {};
			
			\draw (3s1) -- (3s2);
			\draw (3f1) -- (3s1) -- (3f2) -- (3s2) -- (3f1);
	\end{tikzpicture}}\big\}-line Hoffman graph.
\end{lem}
\begin{proof}
	Note that $\Go(2)=$ \big\{\raisebox{-1ex}{\begin{tikzpicture}[scale=0.3]
			
			\tikzstyle{every node}=[draw,circle,fill=black,minimum size=10pt,scale=0.3,
			inner sep=0pt]
			
			\draw (-2.1,0) node (1f1) [label=below:$$] {};
			\draw (-1.6,0) node (1f2) [label=below:$$] {};
			\draw (-1.1,0) node (1f3) [label=below:$$] {};

			\tikzstyle{every node}=[draw,circle,fill=black,minimum size=5pt,scale=0.3,
			inner sep=0pt]

			\draw (-1.6,1) node (1s1) [label=below:$$] {};
			
			\draw (1f1) -- (1s1) -- (1f2);
			\draw (1f3) -- (1s1);
	\end{tikzpicture}},\hspace{-0.08cm}
	\raisebox{-1ex}{\begin{tikzpicture}[scale=0.3]
			\tikzstyle{every node}=[draw,circle,fill=black,minimum size=10pt,scale=0.3,
			inner sep=0pt]

			\draw (-0.5,0) node (2f1) [label=below:$$] {};
			\draw (0.5,0) node (2f2) [label=below:$$] {};
			\draw (-0.5,1) node (2f3) [label=below:$$] {};
			\draw (0.5,1) node (2f4) [label=below:$$] {};

			\tikzstyle{every node}=[draw,circle,fill=black,minimum size=5pt,scale=0.3,
			inner sep=0pt]

			\draw (0,0.2) node (2s1) [label=below:$$] {};
			\draw (0.3,0.5) node (2s2) [label=below:$$] {};
			\draw (-0.3,0.5) node (2s3) [label=below:$$] {};
			\draw (0,0.8) node (2s4) [label=below:$$] {};

			\draw (2f1) -- (2s1) -- (2f2) -- (2s2) -- (2f4) -- (2s4) -- (2f3) -- (2s3) -- (2f1);
			\draw (2s1) -- (2s4);
			\draw (2s2) -- (2s3);
	\end{tikzpicture}},\hspace{-0.08cm}
	\raisebox{-1ex}{\begin{tikzpicture}[scale=0.3]
			\tikzstyle{every node}=[draw,circle,fill=black,minimum size=10pt,scale=0.3,
			inner sep=0pt]

			\draw (-0.5,0) node (2f1) [label=below:$$] {};
			\draw (0.5,0) node (2f2) [label=below:$$] {};
			\draw (-0.5,1) node (2f3) [label=below:$$] {};
			\draw (0.5,1) node (2f4) [label=below:$$] {};

			\tikzstyle{every node}=[draw,circle,fill=black,minimum size=5pt,scale=0.3,
			inner sep=0pt]

			%	\draw (0,0.2) node (2s1) [label=below:$$] {};
			\draw (0.3,0.5) node (2s2) [label=below:$$] {};
			\draw (-0.3,0.5) node (2s3) [label=below:$$] {};
			\draw (0,0.8) node (2s4) [label=below:$$] {};

			\draw (2f2) -- (2s2) -- (2f4) -- (2s4) -- (2f3) -- (2s3) -- (2f1);
			%	\draw (2s1) -- (2s4);
			\draw (2s2) -- (2s3);
	\end{tikzpicture}},\hspace{-0.08cm}
	\raisebox{-1ex}{\begin{tikzpicture}[scale=0.3]
			\tikzstyle{every node}=[draw,circle,fill=black,minimum size=10pt,scale=0.3,
			inner sep=0pt]

			\draw (1.5,0) node (3f1) [label=below:$$] {};
			\draw (1.5,1) node (3f2) [label=below:$$] {};

			\tikzstyle{every node}=[draw,circle,fill=black,minimum size=5pt,scale=0.3,
			inner sep=0pt]

			\draw (1,0.5) node (3s1) [label=below:$$] {};
			\draw (2,0.5) node (3s2) [label=below:$$] {};
			
			\draw (3s1) -- (3s2);
			\draw (3f1) -- (3s1) -- (3f2) -- (3s2) -- (3f1);
	\end{tikzpicture}},\hspace{-0.08cm}
	\raisebox{-1ex}{\begin{tikzpicture}[scale=0.3]
			
			\tikzstyle{every node}=[draw,circle,fill=black,minimum size=10pt,scale=0.3,
			inner sep=0pt]
			
			\draw (-2.2,0) node (1f1) [label=below:$$] {};
			\draw (-1.6,0) node (1f2) [label=below:$$] {};
			\draw (-1.0,0) node (1f3) [label=below:$$] {};

			\tikzstyle{every node}=[draw,circle,fill=black,minimum size=5pt,scale=0.3,
			inner sep=0pt]

			\draw (-1.9,1) node (1s1) [label=below:$$] {};
			\draw (-1.3,1) node (1s2) [label=below:$$] {};
			
			\draw (1f1) -- (1s1) -- (1f2);
			\draw (1f2) -- (1s2) -- (1f3);
	\end{tikzpicture}}\big\}.
	As
	\raisebox{-1ex}{\begin{tikzpicture}[scale=0.3]
			\tikzstyle{every node}=[draw,circle,fill=black,minimum size=10pt,scale=0.3,
			inner sep=0pt]

			\draw (-0.5,0) node (2f1) [label=below:$$] {};
			\draw (0.5,0) node (2f2) [label=below:$$] {};
			\draw (-0.5,1) node (2f3) [label=below:$$] {};
			\draw (0.5,1) node (2f4) [label=below:$$] {};

			\tikzstyle{every node}=[draw,circle,fill=black,minimum size=5pt,scale=0.3,
			inner sep=0pt]

			%	\draw (0,0.2) node (2s1) [label=below:$$] {};
			\draw (0.3,0.5) node (2s2) [label=below:$$] {};
			\draw (-0.3,0.5) node (2s3) [label=below:$$] {};
			\draw (0,0.8) node (2s4) [label=below:$$] {};

			\draw (2f2) -- (2s2) -- (2f4) -- (2s4) -- (2f3) -- (2s3) -- (2f1);
			%	\draw (2s1) -- (2s4);
			\draw (2s2) -- (2s3);
	\end{tikzpicture}} and
	\raisebox{-1ex}{\begin{tikzpicture}[scale=0.3]
		
		\tikzstyle{every node}=[draw,circle,fill=black,minimum size=10pt,scale=0.3,
		inner sep=0pt]
		
		\draw (-2.2,0) node (1f1) [label=below:$$] {};
		\draw (-1.6,0) node (1f2) [label=below:$$] {};
		\draw (-1.0,0) node (1f3) [label=below:$$] {};

		\tikzstyle{every node}=[draw,circle,fill=black,minimum size=5pt,scale=0.3,
		inner sep=0pt]

		\draw (-1.9,1) node (1s1) [label=below:$$] {};
		\draw (-1.3,1) node (1s2) [label=below:$$] {};
		
		\draw (1f1) -- (1s1) -- (1f2);
		\draw (1f2) -- (1s2) -- (1f3);
\end{tikzpicture}}
both are induced Hoffman subgraphs of
	\raisebox{-1ex}{\begin{tikzpicture}[scale=0.3]
			\tikzstyle{every node}=[draw,circle,fill=black,minimum size=10pt,scale=0.3,
			inner sep=0pt]

			\draw (-0.5,0) node (2f1) [label=below:$$] {};
			\draw (0.5,0) node (2f2) [label=below:$$] {};
			\draw (-0.5,1) node (2f3) [label=below:$$] {};
			\draw (0.5,1) node (2f4) [label=below:$$] {};

			\tikzstyle{every node}=[draw,circle,fill=black,minimum size=5pt,scale=0.3,
			inner sep=0pt]

			\draw (0,0.2) node (2s1) [label=below:$$] {};
			\draw (0.3,0.5) node (2s2) [label=below:$$] {};
			\draw (-0.3,0.5) node (2s3) [label=below:$$] {};
			\draw (0,0.8) node (2s4) [label=below:$$] {};

			\draw (2f1) -- (2s1) -- (2f2) -- (2s2) -- (2f4) -- (2s4) -- (2f3) -- (2s3) -- (2f1);
			\draw (2s1) -- (2s4);
			\draw (2s2) -- (2s3);
	\end{tikzpicture}}. 
	Thus, a  $\Go(2)$-line Hoffman graph is a \big\{\raisebox{-1ex}{\begin{tikzpicture}[scale=0.3]
			
			\tikzstyle{every node}=[draw,circle,fill=black,minimum size=10pt,scale=0.3,
			inner sep=0pt]
			
			\draw (-2.1,0) node (1f1) [label=below:$$] {};
			\draw (-1.6,0) node (1f2) [label=below:$$] {};
			\draw (-1.1,0) node (1f3) [label=below:$$] {};

			\tikzstyle{every node}=[draw,circle,fill=black,minimum size=5pt,scale=0.3,
			inner sep=0pt]

			\draw (-1.6,1) node (1s1) [label=below:$$] {};
			
			\draw (1f1) -- (1s1) -- (1f2);
			\draw (1f3) -- (1s1);
	\end{tikzpicture}},\hspace{-0.08cm}
	\raisebox{-1ex}{\begin{tikzpicture}[scale=0.3]
			\tikzstyle{every node}=[draw,circle,fill=black,minimum size=10pt,scale=0.3,
			inner sep=0pt]

			\draw (-0.5,0) node (2f1) [label=below:$$] {};
			\draw (0.5,0) node (2f2) [label=below:$$] {};
			\draw (-0.5,1) node (2f3) [label=below:$$] {};
			\draw (0.5,1) node (2f4) [label=below:$$] {};

			\tikzstyle{every node}=[draw,circle,fill=black,minimum size=5pt,scale=0.3,
			inner sep=0pt]

			\draw (0,0.2) node (2s1) [label=below:$$] {};
			\draw (0.3,0.5) node (2s2) [label=below:$$] {};
			\draw (-0.3,0.5) node (2s3) [label=below:$$] {};
			\draw (0,0.8) node (2s4) [label=below:$$] {};

			\draw (2f1) -- (2s1) -- (2f2) -- (2s2) -- (2f4) -- (2s4) -- (2f3) -- (2s3) -- (2f1);
			\draw (2s1) -- (2s4);
			\draw (2s2) -- (2s3);
	\end{tikzpicture}},\hspace{-0.08cm}
	\raisebox{-1ex}{\begin{tikzpicture}[scale=0.3]
			\tikzstyle{every node}=[draw,circle,fill=black,minimum size=10pt,scale=0.3,
			inner sep=0pt]

			\draw (1.5,0) node (3f1) [label=below:$$] {};
			\draw (1.5,1) node (3f2) [label=below:$$] {};

			\tikzstyle{every node}=[draw,circle,fill=black,minimum size=5pt,scale=0.3,
			inner sep=0pt]

			\draw (1,0.5) node (3s1) [label=below:$$] {};
			\draw (2,0.5) node (3s2) [label=below:$$] {};
			
			\draw (3s1) -- (3s2);
			\draw (3f1) -- (3s1) -- (3f2) -- (3s2) -- (3f1);
	\end{tikzpicture}}\big\}-line Hoffman graph.
\end{proof}

Koolen et al.\cite[Theorem 3.7]{KYY19} showed the following result, which relates $\mathcal{M}(t)$ and $\Go(t)$.
Two matrices $B_1$ and $B_2$ are \emph{equivalent} if there exists a permutation matrix $P$ such that $P^T B_1 P = B_2$.

\begin{thm}[{\cite[Theorem 3.7]{KYY19}}]\label{line graphhjk} Let $t$ be a positive integer and $\ho$ a $t$-fat Hoffman graph.
	 If the special matrix $S(\ho)$ does not contain any principal submatrix equivalent to an element of $\mathcal{M}(t)$, then $\ho$ is a $t$-fat $\Go(t)$-line Hoffman graph.
	
\end{thm}

\section{Associated Hoffman graphs of $\mu$-bounded graphs }\label{3}

%In this section, we first introduce the quasi-clique and define associated Hoffman graphs, refered to \cite{KKY}. Especially, We also will give a simplified definition for the associated Hoffman graph of a $\mu$-bounded graph.
In this section, we first give a simplified definition for the associated Hoffman graph of a $\mu$-bounded graph. For a more general definition of associated Hoffman graphs by using quasi-cliques, we refer readers to \cite{KKY}. For $\mu$-bounded graphs, these two definitions are equivalent if the order of each maximal clique in the graph is large enough.
%A {\em maximal clique} is a clique that cannot be extended by including an additional adjacent vertex.

% Let $n\geq (m+1)^2$ be a positive integer. Let $\mathcal{C}(n)$ := $\{C\mid$ $C$ is a maximal clique of $G$ with at least $n$ vertices$\}$. Define the relation $\equiv_n^m$ on $\mathcal{C}(n)$ by $C_1\equiv_n^m C_2$ if each vertex $x\in C_1$ has at most $m-1$ non-neighbours in $C_2$ and each vertex $y\in C_2$ has at most $m-1$ non-neighbours in $C_1$. Then $\equiv_n^m$ is an equivalence relation as $n\geq (m+1)^2$ (see \cite[Lemma 3.1]{KKY}).

%Let $[C]_n^m$ denote the equivalence class of $\mathcal{C}(n)$ of $G$ under the equivalence relation $\equiv_n^m$ containing the maximal clique $C$ of $\mathcal{C}(n)$. We define the quasi-clique with respect to the pair $(m, n)$, $Q([C]_n^m)$, as the induced subgraph of $G$ on the set $\{x\in V(G)\mid$ $x$ has at most $m-1$ non-neighbours in $C\}$.

%\begin{de}
% Let $n\geq (m+1)^2$ be a positive integer and $[C_1]_n^m, \dots, [C_r]_n^m$ equivalence classes of maximal cliques under $\equiv_n^m$. The \emph{associated Hoffman graph} $\go=\go(G,m,n)$ is the Hoffman graph satisfying the following conditions:
%	\begin{enumerate}
%		\item $V_{\rm slim}(\go) = V(G)$, $V_{\rm fat}(\go)=\{F_1,\dots,F_r\}$;
%		\item the slim graph of $\go$ equals $G$;
%		\item the fat vertex $F_i$ is adjacent precisely to all vertices in $Q([C_i]_n^m)$ for each $i$.
%	\end{enumerate}
	
%\end{de}

%Note that a quasi clique with respect to the pair $(m,n)$ is maximal clique if $m=1$. For the convenience, let $\go(G,n):=\go(G,1,n)$.

 \begin{de}\label{asso}
%	Let $c$ be a positive integer. Let $q$ be a positive integer at least $c+2$ and $G$ a $\mu$-bounded graph with parameter $c$. 
	Let $q\geq2$ be an integer and $G$ a $\mu$-bounded graph. Let $\mathcal{C}(q)$ := $\{C_1,\dots,C_r\}$ be the set of all maximal cliques in $G$ with at least $q$ vertices. The \emph{associated Hoffman graph $\go=\go(G,q)$} is the Hoffman graph satisfying the following conditions:
	\begin{enumerate}
		\item $V_{\rm slim}(\go) = V(G)$ and $V_{\rm fat}(\go)=\{f_1,\dots,f_r\}$,
		\item the slim graph of $\go$ is equal to $G$,
		\item $f_i$ is adjacent precisely to all vertices in $C_i$ for $i\in[r]$.
	\end{enumerate} 	
\end{de}
	Note that if two distinct slim vertices $u$ and $v$ in an associated Hoffman graph have a common fat neighbor, then $u$ is adjacent to $v$.
	For a $\mu$-bounded graph $G$ with parameter $c$, the following proposition characterizes the adjacency matrix of $G$ if $\go(G,q)$ is a $2$-fat \big\{\raisebox{-1ex}{\begin{tikzpicture}[scale=0.3]
			
			\tikzstyle{every node}=[draw,circle,fill=black,minimum size=10pt,scale=0.3,
			inner sep=0pt]
			
			\draw (-2.1,0) node (1f1) [label=below:$$] {};
			\draw (-1.6,0) node (1f2) [label=below:$$] {};
			\draw (-1.1,0) node (1f3) [label=below:$$] {};

			\tikzstyle{every node}=[draw,circle,fill=black,minimum size=5pt,scale=0.3,
			inner sep=0pt]

			\draw (-1.6,1) node (1s1) [label=below:$$] {};
			
			\draw (1f1) -- (1s1) -- (1f2);
			\draw (1f3) -- (1s1);
	\end{tikzpicture}},\hspace{-0.08cm}
	\raisebox{-1ex}{\begin{tikzpicture}[scale=0.3]
			\tikzstyle{every node}=[draw,circle,fill=black,minimum size=10pt,scale=0.3,
			inner sep=0pt]

			\draw (1.5,0) node (3f1) [label=below:$$] {};
			\draw (1.5,1) node (3f2) [label=below:$$] {};

			\tikzstyle{every node}=[draw,circle,fill=black,minimum size=5pt,scale=0.3,
			inner sep=0pt]

			\draw (1,0.5) node (3s1) [label=below:$$] {};
			\draw (2,0.5) node (3s2) [label=below:$$] {};
			
			\draw (3s1) -- (3s2);
			\draw (3f1) -- (3s1) -- (3f2) -- (3s2) -- (3f1);
	\end{tikzpicture}}\big\}-line Hoffman graph, and thus $\lambda_{\min}(G)\geq -3$.
% then the following two propositions can describe the structure for its local graphs and adjacency matrix which implies $c\leq 18$.

We denote by ${\bf j}$ the all one vector.
For a vector ${\bf v}$, we denote by $\operatorname*{supp}({\bf v}):=\{r\mid ({\bf v})_r\neq 0 \}$ its support.
For two vectors ${\bf a}=(a_1,\ldots,a_m)^T$ and ${\bf b}=(b_1,\ldots,b_n)^T$, we denote by ${\bf a}\oplus{\bf b}$ the vector $(a_1,\ldots,a_m,b_1,\ldots,b_n)^T$.
	\begin{pro}\label{muassle}
		Let $c$ and $q$ be positive integers.
		 % Let $q$ be a positive integer at least $c+2$ and 
		Let $G$ be a $\mu$-bounded graph with parameter $c$.
		If the associated graph $\go(G,q)$ is a $2$-fat \big\{\raisebox{-1ex}{\begin{tikzpicture}[scale=0.3]
				
				\tikzstyle{every node}=[draw,circle,fill=black,minimum size=10pt,scale=0.3,
				inner sep=0pt]
				
				\draw (-2.1,0) node (1f1) [label=below:$$] {};
				\draw (-1.6,0) node (1f2) [label=below:$$] {};
				\draw (-1.1,0) node (1f3) [label=below:$$] {};

				\tikzstyle{every node}=[draw,circle,fill=black,minimum size=5pt,scale=0.3,
				inner sep=0pt]

				\draw (-1.6,1) node (1s1) [label=below:$$] {};
				
				\draw (1f1) -- (1s1) -- (1f2);
				\draw (1f3) -- (1s1);
		\end{tikzpicture}},\hspace{-0.08cm}
		\raisebox{-1ex}{\begin{tikzpicture}[scale=0.3]
				\tikzstyle{every node}=[draw,circle,fill=black,minimum size=10pt,scale=0.3,
				inner sep=0pt]

				\draw (1.5,0) node (3f1) [label=below:$$] {};
				\draw (1.5,1) node (3f2) [label=below:$$] {};

				\tikzstyle{every node}=[draw,circle,fill=black,minimum size=5pt,scale=0.3,
				inner sep=0pt]

				\draw (1,0.5) node (3s1) [label=below:$$] {};
				\draw (2,0.5) node (3s2) [label=below:$$] {};
				
				\draw (3s1) -- (3s2);
				\draw (3f1) -- (3s1) -- (3f2) -- (3s2) -- (3f1);
		\end{tikzpicture}}\big\}-line Hoffman graph, 
	then there exists a $(0,\pm 1)$-matrix $N$ whose columns are indexed by the vertex set $V(G)$ and satisfying the following properties:
		\begin{enumerate}
			\item $A(G)+3I=N^TN$ and $\lambda_{\min}(G)\geq -3$,
			\item $N_v^TN_v=3$ and ${\bf j}^TN_v\in\{1,3\}$ for each vertex $v\in V(G)$,
			\item if  ${\bf j}^TN_v=1$, then there is a vertex $u\in V(G)$
			 such that $N_v^TN_u=1$ and $\operatorname*{supp}(N_v)=\operatorname*{supp}(N_u)$.
		\end{enumerate} 

	\end{pro}

\begin{proof}
 By definition \ref{line},
 there is a Hoffman graph $\ho=\uplus_{i=1}^{t}\ho^i$
 containing $\go(G,q)$ as an induced Hoffman subgraph such that 
 $V_{\rm slim}(\ho)=V_{\rm slim}(\go(G,q))=V(G)$, where $\ho^i$ is isomorphic to a member of \big\{\raisebox{-1ex}{\begin{tikzpicture}[scale=0.3]
 		
 		\tikzstyle{every node}=[draw,circle,fill=black,minimum size=10pt,scale=0.3,
 		inner sep=0pt]
 		
 		\draw (-2.1,0) node (1f1) [label=below:$$] {};
 		\draw (-1.6,0) node (1f2) [label=below:$$] {};
 		\draw (-1.1,0) node (1f3) [label=below:$$] {};

 		\tikzstyle{every node}=[draw,circle,fill=black,minimum size=5pt,scale=0.3,
 		inner sep=0pt]

 		\draw (-1.6,1) node (1s1) [label=below:$$] {};
 		
 		\draw (1f1) -- (1s1) -- (1f2);
 		\draw (1f3) -- (1s1);
 \end{tikzpicture}},\hspace{-0.08cm}
 \raisebox{-1ex}{\begin{tikzpicture}[scale=0.3]
 		
 		\tikzstyle{every node}=[draw,circle,fill=black,minimum size=10pt,scale=0.3,
 		inner sep=0pt]
 		
 		\draw (-2.1,0) node (1f1) [label=below:$$] {};
 		\draw (-1.1,0) node (1f3) [label=below:$$] {};

 		\tikzstyle{every node}=[draw,circle,fill=black,minimum size=5pt,scale=0.3,
 		inner sep=0pt]

 		\draw (-1.6,1) node (1s1) [label=below:$$] {};
 		
 		\draw (1f1) -- (1s1) -- (1f3);
 \end{tikzpicture}},\hspace{-0.08cm}
 \raisebox{-1ex}{\begin{tikzpicture}[scale=0.3]
 		\tikzstyle{every node}=[draw,circle,fill=black,minimum size=10pt,scale=0.3,
 		inner sep=0pt]

 		\draw (1.5,0) node (3f1) [label=below:$$] {};
 		\draw (1.5,1) node (3f2) [label=below:$$] {};

 		\tikzstyle{every node}=[draw,circle,fill=black,minimum size=5pt,scale=0.3,
 		inner sep=0pt]

 		\draw (1,0.5) node (3s1) [label=below:$$] {};
 		\draw (2,0.5) node (3s2) [label=below:$$] {};
 		
 		\draw (3s1) -- (3s2);
 		\draw (3f1) -- (3s1) -- (3f2) -- (3s2) -- (3f1);
 \end{tikzpicture}}\big\} for $i\in[t]$.
By Definition \ref{directsummatrix}, $S(\ho)=A(G)-D^TD$ is a block diagonal matrix with blocks in $\{S(\raisebox{-1ex}{\begin{tikzpicture}[scale=0.3]
		
		\tikzstyle{every node}=[draw,circle,fill=black,minimum size=10pt,scale=0.3,
		inner sep=0pt]
		
		\draw (-2.1,0) node (1f1) [label=below:$$] {};
		\draw (-1.6,0) node (1f2) [label=below:$$] {};
		\draw (-1.1,0) node (1f3) [label=below:$$] {};

		\tikzstyle{every node}=[draw,circle,fill=black,minimum size=5pt,scale=0.3,
		inner sep=0pt]

		\draw (-1.6,1) node (1s1) [label=below:$$] {};
		
		\draw (1f1) -- (1s1) -- (1f2);
		\draw (1f3) -- (1s1);
\end{tikzpicture}})$,
$S(\raisebox{-1ex}{\begin{tikzpicture}[scale=0.3]
		
		\tikzstyle{every node}=[draw,circle,fill=black,minimum size=10pt,scale=0.3,
		inner sep=0pt]
		
		\draw (-2.1,0) node (1f1) [label=below:$$] {};
		\draw (-1.1,0) node (1f3) [label=below:$$] {};

		\tikzstyle{every node}=[draw,circle,fill=black,minimum size=5pt,scale=0.3,
		inner sep=0pt]

		\draw (-1.6,1) node (1s1) [label=below:$$] {};
		
		\draw (1f1) -- (1s1) -- (1f3);
\end{tikzpicture}})$,
$S(\raisebox{-1ex}{\begin{tikzpicture}[scale=0.3]
		\tikzstyle{every node}=[draw,circle,fill=black,minimum size=10pt,scale=0.3,
		inner sep=0pt]

		\draw (1.5,0) node (3f1) [label=below:$$] {};
		\draw (1.5,1) node (3f2) [label=below:$$] {};

		\tikzstyle{every node}=[draw,circle,fill=black,minimum size=5pt,scale=0.3,
		inner sep=0pt]

		\draw (1,0.5) node (3s1) [label=below:$$] {};
		\draw (2,0.5) node (3s2) [label=below:$$] {};
		
		\draw (3s1) -- (3s2);
		\draw (3f1) -- (3s1) -- (3f2) -- (3s2) -- (3f1);
\end{tikzpicture}}\hspace{-0.02cm})\}$, where
$S(\raisebox{-1ex}{\begin{tikzpicture}[scale=0.3]
		
		\tikzstyle{every node}=[draw,circle,fill=black,minimum size=10pt,scale=0.3,
		inner sep=0pt]
		
		\draw (-2.1,0) node (1f1) [label=below:$$] {};
		\draw (-1.6,0) node (1f2) [label=below:$$] {};
		\draw (-1.1,0) node (1f3) [label=below:$$] {};

		\tikzstyle{every node}=[draw,circle,fill=black,minimum size=5pt,scale=0.3,
		inner sep=0pt]

		\draw (-1.6,1) node (1s1) [label=below:$$] {};
		
		\draw (1f1) -- (1s1) -- (1f2);
		\draw (1f3) -- (1s1);
\end{tikzpicture}})=(-3)$, $S(\raisebox{-1ex}{\begin{tikzpicture}[scale=0.3]
		
		\tikzstyle{every node}=[draw,circle,fill=black,minimum size=10pt,scale=0.3,
		inner sep=0pt]
		
		\draw (-2.1,0) node (1f1) [label=below:$$] {};
		\draw (-1.1,0) node (1f3) [label=below:$$] {};

		\tikzstyle{every node}=[draw,circle,fill=black,minimum size=5pt,scale=0.3,
		inner sep=0pt]

		\draw (-1.6,1) node (1s1) [label=below:$$] {};
		
		\draw (1f1) -- (1s1) -- (1f3);
\end{tikzpicture}})=(-2)$, and
$S(\raisebox{-1ex}{\begin{tikzpicture}[scale=0.3]
		\tikzstyle{every node}=[draw,circle,fill=black,minimum size=10pt,scale=0.3,
		inner sep=0pt]

		\draw (1.5,0) node (3f1) [label=below:$$] {};
		\draw (1.5,1) node (3f2) [label=below:$$] {};

		\tikzstyle{every node}=[draw,circle,fill=black,minimum size=5pt,scale=0.3,
		inner sep=0pt]

		\draw (1,0.5) node (3s1) [label=below:$$] {};
		\draw (2,0.5) node (3s2) [label=below:$$] {};
		
		\draw (3s1) -- (3s2);
		\draw (3f1) -- (3s1) -- (3f2) -- (3s2) -- (3f1);
\end{tikzpicture}}\hspace{-0.02cm})=
\begin{pmatrix}
	-2 & -1\\
	-1 & -2
\end{pmatrix}$.

 For $v\in V(G)$, let %$\ho_v$ be the Hoffman graph in $\{\ho^i\}_{i=1}^t$ containing the vertex $v$, and
  $D_v$ be the column of $D$ indexed by the vertex $v$.
   Let ${\bf e}_v$ be the vector indexed by the vertices of $V(G)$ such that the $v$-entry is one and the others are zero.
 
%	 Assume that $V_{\rm slim}(\ho)=\{s_1,\ldots,s_n\}$, and  contains $s_j$ for $j\in[n]$.
   Let $\ho'\in\{\ho^i\}_{i=1}^t$.
 If $\ho'=$ \raisebox{-1ex}{\begin{tikzpicture}[scale=0.3]
 		
 		\tikzstyle{every node}=[draw,circle,fill=black,minimum size=10pt,scale=0.3,
 		inner sep=0pt]
 		
 		\draw (-2.1,0) node (1f1) [label=below:$$] {};
 		\draw (-1.6,0) node (1f2) [label=below:$$] {};
 		\draw (-1.1,0) node (1f3) [label=below:$$] {};

 		\tikzstyle{every node}=[draw,circle,fill=black,minimum size=5pt,scale=0.3,
 		inner sep=0pt]

 		\draw (-1.6,1) node (1s1) [label=below:$$] {};
 		
 		\draw (1f1) -- (1s1) -- (1f2);
 		\draw (1f3) -- (1s1);
 \end{tikzpicture}} and $V_{\rm slim}(\ho')=\{v\}$, then let $N_v=D_v$.
 If $\ho'=$ \raisebox{-1ex}{\begin{tikzpicture}[scale=0.3]
 		
 		\tikzstyle{every node}=[draw,circle,fill=black,minimum size=10pt,scale=0.3,
 		inner sep=0pt]
 		
 		\draw (-2.1,0) node (1f1) [label=below:$$] {};
 		\draw (-1.1,0) node (1f3) [label=below:$$] {};

 		\tikzstyle{every node}=[draw,circle,fill=black,minimum size=5pt,scale=0.3,
 		inner sep=0pt]

 		\draw (-1.6,1) node (1s1) [label=below:$$] {};
 		
 		\draw (1f1) -- (1s1) -- (1f3);
 \end{tikzpicture}} and $V_{\rm slim}(\ho')=\{v\}$, then let $N_v=D_v\oplus{\bf e}_{v}$.
 If $\ho'=$  \raisebox{-1ex}{\begin{tikzpicture}[scale=0.3]
 		\tikzstyle{every node}=[draw,circle,fill=black,minimum size=10pt,scale=0.3,
 		inner sep=0pt]

 		\draw (1.5,0) node (3f1) [label=below:$$] {};
 		\draw (1.5,1) node (3f2) [label=below:$$] {};

 		\tikzstyle{every node}=[draw,circle,fill=black,minimum size=5pt,scale=0.3,
 		inner sep=0pt]

 		\draw (1,0.5) node (3s1) [label=below:$$] {};
 		\draw (2,0.5) node (3s2) [label=below:$$] {};
 		
 		\draw (3s1) -- (3s2);
 		\draw (3f1) -- (3s1) -- (3f2) -- (3s2) -- (3f1);
 \end{tikzpicture}}  and $V_{\rm slim}(\ho')=\{u,v\}$, then let $N_v=D_v\oplus {\bf e}_{v}$ and $N_u=D_u\oplus(-{\bf e}_{v})$.
 Let $N$ be the matrix such that the column indexed by $v$ is $N_v$ for $v\in V(G)$.
 Therefore, $N$ is a $(0,\pm 1)$-matrix satisfying (ii) and (iii).

 Let $u$ and $v$ be two distinct vertices in $V(G)$.
 If $u,v$ are not contained in the same member of $\{\ho^i\}_{i=1}^t$,
 then $(N_u)^TN_v=(D_u)^TD_v$ and $(S(\ho))_{uv}=0$ by Lemma \ref{combi}.
 It follows that $(N_u)^TN_v=(D_u)^TD_v=(A(G)-S(\ho))_{uv}=(A(G))_{uv}$.
 If $u,v$ are contained in the member $\ho'$ of $\{\ho^i\}_{i=1}^t$, then
  $\ho'=  \raisebox{-1ex}{\begin{tikzpicture}[scale=0.3]
 		\tikzstyle{every node}=[draw,circle,fill=black,minimum size=10pt,scale=0.3,
 		inner sep=0pt]

 		\draw (1.5,0) node (3f1) [label=below:$$] {};
 		\draw (1.5,1) node (3f2) [label=below:$$] {};

 		\tikzstyle{every node}=[draw,circle,fill=black,minimum size=5pt,scale=0.3,
 		inner sep=0pt]

 		\draw (1,0.5) node (3s1) [label=below:$$] {};
 		\draw (2,0.5) node (3s2) [label=below:$$] {};
 		
 		\draw (3s1) -- (3s2);
 		\draw (3f1) -- (3s1) -- (3f2) -- (3s2) -- (3f1);
 \end{tikzpicture}} $.
 It follows that $(N_u)^TN_v=(D_u)^TD_v-1$ and $(S(\ho))_{uv}=-1$.
Furthermore, $(N_u)^TN_v=(D_u)^TD_v-1=(A(G)-S(\ho))_{uv}-1=(A(G))_{uv}$.
Therefore, $A(G)+3I=N^TN$, and thus $\lambda_{\min}(G)\geq -3$.
\end{proof}

For a $\mu$-bounded graph $G$ with parameter $c$, the following proposition characterizes the structure of $G$ if $\go(G,q)$ is a $2$-fat \big\{\raisebox{-1ex}{\begin{tikzpicture}[scale=0.3]
		
		\tikzstyle{every node}=[draw,circle,fill=black,minimum size=10pt,scale=0.3,
		inner sep=0pt]
		
		\draw (-2.1,0) node (1f1) [label=below:$$] {};
		\draw (-1.6,0) node (1f2) [label=below:$$] {};
		\draw (-1.1,0) node (1f3) [label=below:$$] {};

		\tikzstyle{every node}=[draw,circle,fill=black,minimum size=5pt,scale=0.3,
		inner sep=0pt]

		\draw (-1.6,1) node (1s1) [label=below:$$] {};
		
		\draw (1f1) -- (1s1) -- (1f2);
		\draw (1f3) -- (1s1);
\end{tikzpicture}},\hspace{-0.08cm}
\raisebox{-1ex}{\begin{tikzpicture}[scale=0.3]
		\tikzstyle{every node}=[draw,circle,fill=black,minimum size=10pt,scale=0.3,
		inner sep=0pt]

		\draw (1.5,0) node (3f1) [label=below:$$] {};
		\draw (1.5,1) node (3f2) [label=below:$$] {};

		\tikzstyle{every node}=[draw,circle,fill=black,minimum size=5pt,scale=0.3,
		inner sep=0pt]

		\draw (1,0.5) node (3s1) [label=below:$$] {};
		\draw (2,0.5) node (3s2) [label=below:$$] {};
		
		\draw (3s1) -- (3s2);
		\draw (3f1) -- (3s1) -- (3f2) -- (3s2) -- (3f1);
\end{tikzpicture}}\big\}-line Hoffman graph.

\begin{pro}\label{muassle2}
		Let $c$ and $q$ be positive integers.
	% Let $q$ be a positive integer at least $c+2$ and 
	Let $G$ be a $\mu$-bounded graph with parameter $c$.
	If the associated graph $\go(G,q)$ is a $2$-fat \big\{\raisebox{-1ex}{\begin{tikzpicture}[scale=0.3]
			
			\tikzstyle{every node}=[draw,circle,fill=black,minimum size=10pt,scale=0.3,
			inner sep=0pt]
			
			\draw (-2.1,0) node (1f1) [label=below:$$] {};
			\draw (-1.6,0) node (1f2) [label=below:$$] {};
			\draw (-1.1,0) node (1f3) [label=below:$$] {};

			\tikzstyle{every node}=[draw,circle,fill=black,minimum size=5pt,scale=0.3,
			inner sep=0pt]

			\draw (-1.6,1) node (1s1) [label=below:$$] {};
			
			\draw (1f1) -- (1s1) -- (1f2);
			\draw (1f3) -- (1s1);
	\end{tikzpicture}},\hspace{-0.08cm}
	\raisebox{-1ex}{\begin{tikzpicture}[scale=0.3]
			\tikzstyle{every node}=[draw,circle,fill=black,minimum size=10pt,scale=0.3,
			inner sep=0pt]

			\draw (1.5,0) node (3f1) [label=below:$$] {};
			\draw (1.5,1) node (3f2) [label=below:$$] {};

			\tikzstyle{every node}=[draw,circle,fill=black,minimum size=5pt,scale=0.3,
			inner sep=0pt]

			\draw (1,0.5) node (3s1) [label=below:$$] {};
			\draw (2,0.5) node (3s2) [label=below:$$] {};
			
			\draw (3s1) -- (3s2);
			\draw (3f1) -- (3s1) -- (3f2) -- (3s2) -- (3f1);
	\end{tikzpicture}}\big\}-line Hoffman graph, 
	then there exists a set of cliques $\Co$ in $G$ satisfying the following:
	\begin{enumerate}
		\item $E(G)=\bigcup_{C\in\Co} E(C)$,% and $|V(C)\cap V(C')|\leq 2$ for two distinct cliques $C,C'\in\Co$,
		\item for every vertex $u\in V(G)$, there are at most three cliques in $\Co$ containing $u$, where two of them are maximal cliques with order at least $q$,
		\item if there are distinct cliques $C,C'\in\Co$ such that $V(C)\cap V(C')\supset\{u,v\}$, then $C$ and $C'$ are the only two cliques in $\Co$ containing $u$ (resp. $v$) and $V(C)\cap V(C')=\{u,v\}$.
	\end{enumerate}
\end{pro}
\begin{proof}
By definition \ref{line},
there is a Hoffman graph $\ho=\uplus_{i=1}^{t}\ho^i$
containing $\go(G,q)$ as an induced Hoffman subgraph such that 
$V_{\rm slim}(\ho)=V_{\rm slim}(\go(G,q))=V(G)$, where $\ho^i$ is isomorphic to a member of \big\{\raisebox{-1ex}{\begin{tikzpicture}[scale=0.3]
		
		\tikzstyle{every node}=[draw,circle,fill=black,minimum size=10pt,scale=0.3,
		inner sep=0pt]
		
		\draw (-2.1,0) node (1f1) [label=below:$$] {};
		\draw (-1.6,0) node (1f2) [label=below:$$] {};
		\draw (-1.1,0) node (1f3) [label=below:$$] {};

		\tikzstyle{every node}=[draw,circle,fill=black,minimum size=5pt,scale=0.3,
		inner sep=0pt]

		\draw (-1.6,1) node (1s1) [label=below:$$] {};
		
		\draw (1f1) -- (1s1) -- (1f2);
		\draw (1f3) -- (1s1);
\end{tikzpicture}},\hspace{-0.08cm}
\raisebox{-1ex}{\begin{tikzpicture}[scale=0.3]
		
		\tikzstyle{every node}=[draw,circle,fill=black,minimum size=10pt,scale=0.3,
		inner sep=0pt]
		
		\draw (-2.1,0) node (1f1) [label=below:$$] {};
		\draw (-1.1,0) node (1f3) [label=below:$$] {};

		\tikzstyle{every node}=[draw,circle,fill=black,minimum size=5pt,scale=0.3,
		inner sep=0pt]

		\draw (-1.6,1) node (1s1) [label=below:$$] {};
		
		\draw (1f1) -- (1s1) -- (1f3);
\end{tikzpicture}},\hspace{-0.08cm}
\raisebox{-1ex}{\begin{tikzpicture}[scale=0.3]
		\tikzstyle{every node}=[draw,circle,fill=black,minimum size=10pt,scale=0.3,
		inner sep=0pt]

		\draw (1.5,0) node (3f1) [label=below:$$] {};
		\draw (1.5,1) node (3f2) [label=below:$$] {};

		\tikzstyle{every node}=[draw,circle,fill=black,minimum size=5pt,scale=0.3,
		inner sep=0pt]

		\draw (1,0.5) node (3s1) [label=below:$$] {};
		\draw (2,0.5) node (3s2) [label=below:$$] {};
		
		\draw (3s1) -- (3s2);
		\draw (3f1) -- (3s1) -- (3f2) -- (3s2) -- (3f1);
\end{tikzpicture}}\big\} for $i\in[t]$.
 By Definition \ref{directsummatrix}, $S(\ho)=A(G)-D^TD$ is a block diagonal matrix with blocks in $\{S(\raisebox{-1ex}{\begin{tikzpicture}[scale=0.3]
		
		\tikzstyle{every node}=[draw,circle,fill=black,minimum size=10pt,scale=0.3,
		inner sep=0pt]
		
		\draw (-2.1,0) node (1f1) [label=below:$$] {};
		\draw (-1.6,0) node (1f2) [label=below:$$] {};
		\draw (-1.1,0) node (1f3) [label=below:$$] {};

		\tikzstyle{every node}=[draw,circle,fill=black,minimum size=5pt,scale=0.3,
		inner sep=0pt]

		\draw (-1.6,1) node (1s1) [label=below:$$] {};
		
		\draw (1f1) -- (1s1) -- (1f2);
		\draw (1f3) -- (1s1);
\end{tikzpicture}})$,
$S(\raisebox{-1ex}{\begin{tikzpicture}[scale=0.3]
		
		\tikzstyle{every node}=[draw,circle,fill=black,minimum size=10pt,scale=0.3,
		inner sep=0pt]
		
		\draw (-2.1,0) node (1f1) [label=below:$$] {};
		\draw (-1.1,0) node (1f3) [label=below:$$] {};

		\tikzstyle{every node}=[draw,circle,fill=black,minimum size=5pt,scale=0.3,
		inner sep=0pt]

		\draw (-1.6,1) node (1s1) [label=below:$$] {};
		
		\draw (1f1) -- (1s1) -- (1f3);
\end{tikzpicture}})$,
$S(\raisebox{-1ex}{\begin{tikzpicture}[scale=0.3]
		\tikzstyle{every node}=[draw,circle,fill=black,minimum size=10pt,scale=0.3,
		inner sep=0pt]

		\draw (1.5,0) node (3f1) [label=below:$$] {};
		\draw (1.5,1) node (3f2) [label=below:$$] {};

		\tikzstyle{every node}=[draw,circle,fill=black,minimum size=5pt,scale=0.3,
		inner sep=0pt]

		\draw (1,0.5) node (3s1) [label=below:$$] {};
		\draw (2,0.5) node (3s2) [label=below:$$] {};
		
		\draw (3s1) -- (3s2);
		\draw (3f1) -- (3s1) -- (3f2) -- (3s2) -- (3f1);
\end{tikzpicture}}\hspace{-0.02cm})\}$, where
$S(\raisebox{-1ex}{\begin{tikzpicture}[scale=0.3]
		
		\tikzstyle{every node}=[draw,circle,fill=black,minimum size=10pt,scale=0.3,
		inner sep=0pt]
		
		\draw (-2.1,0) node (1f1) [label=below:$$] {};
		\draw (-1.6,0) node (1f2) [label=below:$$] {};
		\draw (-1.1,0) node (1f3) [label=below:$$] {};

		\tikzstyle{every node}=[draw,circle,fill=black,minimum size=5pt,scale=0.3,
		inner sep=0pt]

		\draw (-1.6,1) node (1s1) [label=below:$$] {};
		
		\draw (1f1) -- (1s1) -- (1f2);
		\draw (1f3) -- (1s1);
\end{tikzpicture}})=(-3)$, $S(\raisebox{-1ex}{\begin{tikzpicture}[scale=0.3]
		
		\tikzstyle{every node}=[draw,circle,fill=black,minimum size=10pt,scale=0.3,
		inner sep=0pt]
		
		\draw (-2.1,0) node (1f1) [label=below:$$] {};
		\draw (-1.1,0) node (1f3) [label=below:$$] {};

		\tikzstyle{every node}=[draw,circle,fill=black,minimum size=5pt,scale=0.3,
		inner sep=0pt]

		\draw (-1.6,1) node (1s1) [label=below:$$] {};
		
		\draw (1f1) -- (1s1) -- (1f3);
\end{tikzpicture}})=(-2)$, and
$S(\raisebox{-1ex}{\begin{tikzpicture}[scale=0.3]
		\tikzstyle{every node}=[draw,circle,fill=black,minimum size=10pt,scale=0.3,
		inner sep=0pt]

		\draw (1.5,0) node (3f1) [label=below:$$] {};
		\draw (1.5,1) node (3f2) [label=below:$$] {};

		\tikzstyle{every node}=[draw,circle,fill=black,minimum size=5pt,scale=0.3,
		inner sep=0pt]

		\draw (1,0.5) node (3s1) [label=below:$$] {};
		\draw (2,0.5) node (3s2) [label=below:$$] {};
		
		\draw (3s1) -- (3s2);
		\draw (3f1) -- (3s1) -- (3f2) -- (3s2) -- (3f1);
\end{tikzpicture}}\hspace{-0.02cm})=
\begin{pmatrix}
	-2 & -1\\
	-1 & -2
\end{pmatrix}$.

For a fat vertex $f$ in $\ho$, let $V_f=\{u\in V(G)\mid (D)_{fu}=1\}$.
%Assume that $C=(\vec{v_1}\ldots\vec{v_f})$ and $V(G)=\{u_1,\ldots,u_v\}$. Let $V_i=\{u_j\mid \vec{v_i}(j)=1\}$ and $C_i$ be a clique in $G$ induced by $V_i$ for each $i\leq f$.
%Assume that $V_{\rm fat}(\ho)=\{f_1,\ldots,f_m\}$.
 Let $C_f$ be the subgraph of $G$ induced by $V_f$ for $f\in V_{\rm fat}(\ho)$.
 Assume that there are two distinct non-adjacent vertices $u_1,u_2$ in $C_f$, then
 $u_1,u_2$ have one common fat neighbor.
 Hence, $u_1,u_2$ must be contained in a Hoffman graph $\ho'\in\{\ho^i\}_{i=1}^t$
 by Lemma \ref{combi} (iv).
  It follows that 
 $\ho'=\raisebox{-1ex}{\begin{tikzpicture}[scale=0.3]
 		\tikzstyle{every node}=[draw,circle,fill=black,minimum size=10pt,scale=0.3,
 		inner sep=0pt]

 		\draw (1.5,0) node (3f1) [label=below:$$] {};
 		\draw (1.5,1) node (3f2) [label=below:$$] {};

 		\tikzstyle{every node}=[draw,circle,fill=black,minimum size=5pt,scale=0.3,
 		inner sep=0pt]

 		\draw (1,0.5) node (3s1) [label=below:$$] {};
 		\draw (2,0.5) node (3s2) [label=below:$$] {};
 		
 		\draw (3s1) -- (3s2);
 		\draw (3f1) -- (3s1) -- (3f2) -- (3s2) -- (3f1);
 \end{tikzpicture}}$, which contradicts that $u_1,u_2$ are non-adjacent.
 Therefore, $C_f$ is a clique in $G$ for $f\in V_{\rm fat}(\ho)$. 

 Let $\Co:=\{C_f\mid f\in V_{\rm fat}(\ho)\}$. 
 Assume that there is an edge $uv$ not contained in any clique in $\Co$, then $(D_{u})^TD_{v}=0$. It follows that  $(S(\ho))_{uv}=(S(\ho))_{uv}+(D_{u})^TD_{v}=(A(G))_{uv}=1$.
 However, every off-diagonal entry of $S(\ho)$ is not equal to $1$.
  Hence, (i) holds.
  
 As $\go(G,q)$ is $2$-fat, each vertex in $V(G)$ is contained in at least two maximal cliques with order at least $q$ by Definition \ref{asso}.
 For $v\in V(G)$, the Hoffman graph in $\{\ho^i\}_{i=1}^t$ containing $v$
  is isomorphic to a member of \big\{\raisebox{-1ex}{\begin{tikzpicture}[scale=0.3]
 		
 		\tikzstyle{every node}=[draw,circle,fill=black,minimum size=10pt,scale=0.3,
 		inner sep=0pt]
 		
 		\draw (-2.1,0) node (1f1) [label=below:$$] {};
 		\draw (-1.6,0) node (1f2) [label=below:$$] {};
 		\draw (-1.1,0) node (1f3) [label=below:$$] {};

 		\tikzstyle{every node}=[draw,circle,fill=black,minimum size=5pt,scale=0.3,
 		inner sep=0pt]

 		\draw (-1.6,1) node (1s1) [label=below:$$] {};
 		
 		\draw (1f1) -- (1s1) -- (1f2);
 		\draw (1f3) -- (1s1);
 \end{tikzpicture}},\hspace{-0.08cm}
 \raisebox{-1ex}{\begin{tikzpicture}[scale=0.3]
 		
 		\tikzstyle{every node}=[draw,circle,fill=black,minimum size=10pt,scale=0.3,
 		inner sep=0pt]
 		
 		\draw (-2.1,0) node (1f1) [label=below:$$] {};
 		\draw (-1.1,0) node (1f3) [label=below:$$] {};

 		\tikzstyle{every node}=[draw,circle,fill=black,minimum size=5pt,scale=0.3,
 		inner sep=0pt]

 		\draw (-1.6,1) node (1s1) [label=below:$$] {};
 		
 		\draw (1f1) -- (1s1) -- (1f3);
 \end{tikzpicture}},\hspace{-0.08cm}
 \raisebox{-1ex}{\begin{tikzpicture}[scale=0.3]
 		\tikzstyle{every node}=[draw,circle,fill=black,minimum size=10pt,scale=0.3,
 		inner sep=0pt]

 		\draw (1.5,0) node (3f1) [label=below:$$] {};
 		\draw (1.5,1) node (3f2) [label=below:$$] {};

 		\tikzstyle{every node}=[draw,circle,fill=black,minimum size=5pt,scale=0.3,
 		inner sep=0pt]

 		\draw (1,0.5) node (3s1) [label=below:$$] {};
 		\draw (2,0.5) node (3s2) [label=below:$$] {};
 		
 		\draw (3s1) -- (3s2);
 		\draw (3f1) -- (3s1) -- (3f2) -- (3s2) -- (3f1);
 \end{tikzpicture}}\big\}. Therefore, each vertex is contained in at most three cliques in $\Co$, and (ii) holds.

If there are two cliques $C_{f_1}, C_{f_2}\in\Co$ such that $V_{f_1}\cap V_{f_2}$
 contains two distinct vertices $u,v$, then $u$ and $v$ have two common fat neighbors.
 Hence, $V_{f_1}\cap V_{f_2}$ must be contained in a Hoffman graph $\ho''\in\{\ho_i\}_{i=1}^t$ by Lemma \ref{combi} (iv).
 It follows that 
  $\ho''=\raisebox{-1ex}{\begin{tikzpicture}[scale=0.3]
		\tikzstyle{every node}=[draw,circle,fill=black,minimum size=10pt,scale=0.3,
		inner sep=0pt]

		\draw (1.5,0) node (3f1) [label=below:$$] {};
		\draw (1.5,1) node (3f2) [label=below:$$] {};

		\tikzstyle{every node}=[draw,circle,fill=black,minimum size=5pt,scale=0.3,
		inner sep=0pt]

		\draw (1,0.5) node (3s1) [label=below:$$] {};
		\draw (2,0.5) node (3s2) [label=below:$$] {};
		
		\draw (3s1) -- (3s2);
		\draw (3f1) -- (3s1) -- (3f2) -- (3s2) -- (3f1);
\end{tikzpicture}}$. This implies that $C_{f_1},C_{f_2}$ are the only two cliques in $\Co$ containing $u$ (resp. $v$)
 and $V_{f_1}\cap V_{f_2}=\{u,v\}$.
Therefore, (iii) holds.
\end{proof}

	As an application of Proposition \ref{muassle2}, the following lemma characterizes local graphs of $G$.

%\begin{lem}\label{r33}
%	Let $G$ be a graph.
%	Suppose that there exists a set $\Co$ of maximal cliques in $G$ satisfying the statements in Proposition \ref{muassle2}.
%	Let $\Co(x):=\{C\in\Co\mid x\in C\}$ for $x\in V(G)$.
%	Then $\{x\}\cup N(x)=\cup_{C\in\Co(x)}C$ satisfies one of the following statements for each vertex $x\in V(G)$:
%	\begin{enumerate}
%		\item 
%		There are two maximal cliques $C_1$ and $C_2$ with order at least $q$ such that $\Co(x)=\{C_1,C_2\}$ and $|V(C_1)\cap V(C_2)|\in\{1,2\}$.
%		\item There are three distinct cliques $C_1,C_2$, and $C_3$ such that $\Co(x)=\{C_1,C_2,C_3\}$, and $V(C_i)\cap V(C_j)=\{x\}$ for $i\neq j$, where $C_1$ and $C_2$ are maximal cliques with order at least $q$.
%	\end{enumerate}
%\end{lem}
%\begin{proof}
%	Let $x$ be a vertex in $G$.
%	Proposition \ref{muassle2} (i) shows that each pair of cliques in $\Co$ intersects in at most two vertices by.
%	Proposition \ref{muassle2} (ii) shows that $2\leq |\Co(x)|\leq 3$, and two cliques of $\Co(x)$ are maximal cliques with order at least $q$.
%	If $y\sim x$, then there is a clique $C\in\Co(x)$ such that $y\in V(C)$ by Proposition \ref{muassle2} (i). Therefore, $\{x\}\cup N(x)=\cup_{C\in\Co(x)}C$. 
	
%	If $|\Co(x)|=3$, then
%	any two cliques of $\Co(x)$ intersect in only one vertex by Proposition \ref{muassle2} (iii), and thus $\{x\}\cup N(x)$ satisfies the statement (ii).
%	Otherwise, $\{x\}\cup N(x)$ satisfies the statement (i) if $|\Co(x)|=2$.
%\end{proof}

\begin{cor}
		Let $c$ and $q$ be positive integers.
	Let $G$ be a $\mu$-bounded graph with parameter $c$.
	If the associated graph $\go(G,q)$ is a $2$-fat \big\{\raisebox{-1ex}{\begin{tikzpicture}[scale=0.3]
			
			\tikzstyle{every node}=[draw,circle,fill=black,minimum size=10pt,scale=0.3,
			inner sep=0pt]
			
			\draw (-2.1,0) node (1f1) [label=below:$$] {};
			\draw (-1.6,0) node (1f2) [label=below:$$] {};
			\draw (-1.1,0) node (1f3) [label=below:$$] {};

			\tikzstyle{every node}=[draw,circle,fill=black,minimum size=5pt,scale=0.3,
			inner sep=0pt]

			\draw (-1.6,1) node (1s1) [label=below:$$] {};
			
			\draw (1f1) -- (1s1) -- (1f2);
			\draw (1f3) -- (1s1);
	\end{tikzpicture}},\hspace{-0.08cm}
	\raisebox{-1ex}{\begin{tikzpicture}[scale=0.3]
			\tikzstyle{every node}=[draw,circle,fill=black,minimum size=10pt,scale=0.3,
			inner sep=0pt]

			\draw (1.5,0) node (3f1) [label=below:$$] {};
			\draw (1.5,1) node (3f2) [label=below:$$] {};

			\tikzstyle{every node}=[draw,circle,fill=black,minimum size=5pt,scale=0.3,
			inner sep=0pt]

			\draw (1,0.5) node (3s1) [label=below:$$] {};
			\draw (2,0.5) node (3s2) [label=below:$$] {};
			
			\draw (3s1) -- (3s2);
			\draw (3f1) -- (3s1) -- (3f2) -- (3s2) -- (3f1);
	\end{tikzpicture}}\big\}-line Hoffman graph, 
	then $c\leq 18$ . In particular, if $G$ is regular, then $c\leq 12$.
\end{cor}
\begin{proof}
	Let $\Co$ be a set of cliques satisfying (i)-(iii) of Proposition \ref{muassle2}.
		Let $\Co(v):=\{C\in\Co\mid v\in C\}$ for $v\in V(G)$.
	Let $u_1$ and $u_2$ be two non-adjacent vertices in $G$.
	
	Proposition \ref{muassle2} (ii) shows that $|\Co(u_i)|\leq 3$ for $i=1,2$, and $|V(A)\cap V(B)|\leq 2$ for $A\in\Co(u_1)$ and $B\in\Co(u_2)$.
	Hence, $|N(u_1)\cap N(u_2)|\leq \sum_{A\in\Co(u_1)}\sum_{B\in\Co(u_2)}|V(A)\cap V(B)|\leq 2|\Co(u_1)||\Co(u_2)|=2\times3\times3=18$.
	
	Assume that $G$ is $k$-regular, and   $|N(u_1)\cap N(u_2)|\geq 13$. It follows that $|\Co(u_1)|=|\Co(u_2)|=3$, otherwise $|N(u_1)\cap N(u_2)|\leq 2\times 2\times 3=12$. Let $\Co(u_1)=\{A_1,A_2,A_3\}$, $\Co(u_2)=\{B_1,B_2,B_3\}$,  $x_i:=|V(A_i)|-1$, and $y_i:=|V(B_i)|-1$ for $i=1,2,3$. %Without loss of generality, we assume that $a_1\geq a_2\geq a_3$, $b_1\geq b_2\geq b_3$ and $a_1+a_2+a_3=b_1+b_2+b_3=k$. 
	Note that $x_1+x_2+x_3=y_1+y_2+y_3=k$.
	As $|N(u_1)\cap N(u_2)|\geq 13$, there are at least 4 pairs of $(i,j)$ such that $|V(A_i)\cap V(B_j)|=2$ and $x_i+y_j=k+1$. It can be directly confirmed that there does not exist a solution $(x_1,x_2,x_3,y_1,y_2,y_3)$ satisfying these equations. 
\end{proof}

\begin{lem}\label{9}
		Let $c$ and $q$ be positive integers. % Let $q$ be a positive integer at least $c+2$ and 
	Let $G$ be a $\mu$-bounded graph with parameter $c$.
	If the associated graph $\go(G,q)$ is a $2$-fat \big\{\raisebox{-1ex}{\begin{tikzpicture}[scale=0.3]
			
			\tikzstyle{every node}=[draw,circle,fill=black,minimum size=10pt,scale=0.3,
			inner sep=0pt]
			
			\draw (-2.1,0) node (1f1) [label=below:$$] {};
			\draw (-1.6,0) node (1f2) [label=below:$$] {};
			\draw (-1.1,0) node (1f3) [label=below:$$] {};

			\tikzstyle{every node}=[draw,circle,fill=black,minimum size=5pt,scale=0.3,
			inner sep=0pt]

			\draw (-1.6,1) node (1s1) [label=below:$$] {};
			
			\draw (1f1) -- (1s1) -- (1f2);
			\draw (1f3) -- (1s1);
	\end{tikzpicture}},\hspace{-0.08cm}
	\raisebox{-1ex}{\begin{tikzpicture}[scale=0.3]
			\tikzstyle{every node}=[draw,circle,fill=black,minimum size=10pt,scale=0.3,
			inner sep=0pt]

			\draw (1.5,0) node (3f1) [label=below:$$] {};
			\draw (1.5,1) node (3f2) [label=below:$$] {};

			\tikzstyle{every node}=[draw,circle,fill=black,minimum size=5pt,scale=0.3,
			inner sep=0pt]

			\draw (1,0.5) node (3s1) [label=below:$$] {};
			\draw (2,0.5) node (3s2) [label=below:$$] {};
			
			\draw (3s1) -- (3s2);
			\draw (3f1) -- (3s1) -- (3f2) -- (3s2) -- (3f1);
	\end{tikzpicture}}\big\}-line Hoffman graph,
	then there exists a vertex $u$ in $G$ such that $|N(u)\cap N(v)|\leq 9$ for each vertex $v$ non-adjacent to $u$.
\end{lem}

\begin{proof}
	
 	Let $\Co$ be a set of cliques satisfying (i)-(iii) of Proposition \ref{muassle2}.
 Let $\Co(w):=\{C\in\Co\mid w\in C\}$ for $w\in V(G)$.
	
	If there are two cliques $C_1$ and $C_2$ in $\Co$ such that $V(C_1)\cap V(C_2)=\{u,u'\}$, then $u$ is the desired vertex.
	 We will show that $|N(u)\cap N(v)|\leq 9$ for each vertex $v$ non-adjacent to $u$.	
	Note that $\Co(u)=\{C_1,C_2\}$ by Proposition \ref{muassle2} (iii).
	Assume that there exists a vertex $v\notin N(u)$ such that $|N(u)\cap N(v)|\geq 10$.
	It follows that $|\Co(v)|=3$ , otherwise $|N(u)\cap N(v)|\leq 2\times 2\times 2=8$. Let $A_1,A_2,A_3$ be the cliques in $\Co(v)$. Let  $x_i:=|V(A_i)|-1$ and $y_j:=|V(C_j)|-1$ for $i=1,2,3$ and $j=1,2$. %Without loss of generality, we assume that $a_1\geq a_2\geq a_3$, $b_1\geq b_2\geq b_3$ and $a_1+a_2+a_3=b_1+b_2+b_3=k$. 
	Note that $x_1+x_2+x_3=k=y_1+y_2-1$.
	As $|N(u)\cap N(v)|\geq 10$, there are at least 4 pairs of $(i,j)$ such that $|V(A_i)\cap V(C_j)|=2$ and $x_i+y_j=k+1$. It can be directly confirmed that there does not exist a solution $(x_1,x_2,x_3,y_1,y_2)$ satisfying these equations.
	
	If each pair of distinct cliques in $\Co$ intersects in at most one vertex, then $c\leq 3 \times 3\times 1=9$ by Proposition \ref{muassle2} (ii). In other words, for each vertex $w$ in $G$, it shares at most $9$ common neighbors with every vertex non-adjacent to $w$. This shows the existence of $u$.
\end{proof}

%\begin{re}
%	In the above proposition, if $c\geq 12$, then $G$ can not be a regular graph.
%\end{re}

 The following result states that for a fixed smallest eigenvalue, each vertex outside a big enough clique has either a large number of neighbors or a large number of  non-neighbors in that clique. %We give this result without a proof, as it is a direct corollary of Greaves et al.\cite[Lemma 3.1]{Greaves2021}.

 \begin{lem}[{\cite[Lemma 1.2]{YK24}}]\label{Hat}
 	Let $\lambda\geq1$ be an integer and $G$ a graph with $\lambda_{\min}(G) \geq - \lambda$. Let $C$ be a clique in $G$ with order $\omega$. If $\omega\geq n_1(\lambda):=\lambda^4-2\lambda^3+3\lambda^2-3\lambda+3$, then each vertex outside $C$ has either at most $\lambda(\lambda-1)$ neighbors in $C$ or at least $\omega-(\lambda-1)^2$ neighbors in $C$.
 \end{lem}

 The following result is a version for $\mu$-bounded graphs of \cite[Proposition 4.1]{KKY}, which shows
 the relation between graphs and associated Hoffman graphs.
\begin{pro}\label{assohoff}
	Let $ \phi, \sigma$ and $p$ be non-negative integers.
	Let $\lambda, c$ be positive integers and $\tilde{c}:=\min\{c,\lambda(\lambda-1)\}$. Let $n_1(\lambda):=\lambda^4-2\lambda^3+3\lambda^2-3\lambda+3$.
	There exists a positive integer $n_2(\phi, \sigma,  p,\tilde{c}):=\tilde{c}(\sigma-1)+\tilde{c}(p+1)(\phi-1)+p+1$ such that for each integer $q \geq \max\{n_1(\lambda),n_2(\phi, \sigma,  p,\tilde{c}),(\lambda-1)^2+c+1\}$,
	every Hoffman graph $\ho$ with $\phi$ fat vertices and $\sigma$ slim vertices and every graph $G$ with $\lambda_{\min}(G)\geq -\lambda$,
	 the graph  $G(\ho, p)$ is an induced subgraph of $G$ if the graph  $G$ satisfies the following conditions:
	\begin{enumerate}
		\item the  graph $G$ is a $\mu$-bounded graph with parameter $c$,
		
		\item the associated Hoffman graph $\mathfrak{g}(G, q)$ contains $\ho$ as an induced Hoffman subgraph.
	
	\end{enumerate}
\end{pro}

\begin{proof}	
  % Without loss of generality, we may assume that $\ho$ is an induced Hoffman subgraph of  $\mathfrak{g}(G, q)$ with exactly $\phi$ fat vertices and $\sigma$ slim vertices. 
    For $\phi=0$, the proposition holds.
    Assume that the proposition holds for $\phi-1$, 
    we will show that the proposition holds for $\phi\geq 1$.
    
    Let $V_{\rm fat}(\ho)=\{f_i\}_{i=1}^{\phi}$ and $V_{\rm slim}(\ho)=\{s_i\}_{i=1}^{\sigma}$.
    As $\ho$ is an induced Hoffman subgraph of $\go(G,q)$, there is a set $\{C_i\}_{i=1}^{\phi}$ of the maximal cliques with order at least $q$ in $G$ corresponding to $\{f_i\}_{i=1}^{\phi}$.
   Additionally, the slim graph of $\ho$ is an induced subgraph of $G$ %with $V_{\rm slim}(\ho)\subset V_{\rm slim}(\go(G,q))= V(G)$, 
   as  the slim graph of $\go(G,q)$ is $G$.
  
  Let $\ho':=\ho-f_\phi$ be the Hoffman graph obtained by deleting the fat vertex $f_\phi$ from $\ho$. 
  As $n_2(\phi, \sigma,  p,\tilde{c})\geq n_2(\phi-1, \sigma,  p,\tilde{c})$, by induction, $G(\ho', p)$ is an induced subgraph of $G$.
   In other words,   
   there exist cliques $D_1,\ldots,D_{\phi-1}$ such that 
   \begin{enumerate}
   \item[($a_1$)] for $i\in[\phi-1]$, $V(D_i)\cap V_{\rm slim}(\ho)=\emptyset$,  $V(D_i)\subset V(C_i)$, and $|V(D_i)|\geq q$, 
   \item[($a_2$)]  for distinct $i,j\in[\phi-1]$, $V(D_i)\cap V(D_j)=\emptyset$, and no edge in $G$ exists between $D_i$ and $D_j$,
   \item[($a_3$)] for $\ell\in[\sigma]$ and $i\in[\phi-1]$, the slim vertex  $s_\ell$ is adjacent to all vertices of $D_i$ if $s_\ell$ is adjacent to the fat vertex $f_i$, and $s_\ell$ has no neighbors in $D_i$ if $s_\ell$ is not adjacent to $f_i$.
   \end{enumerate}

 As $G$ is a $\mu$-bounded graph with parameter $c$ and $|V(C_i)|\geq q\geq \max\{n_1(\lambda),(\lambda-1)^2+c+1\}$ for $i\in[\phi]$. Lemma \ref{Hat} implies that $|N_G(x)\cap V(C_i)|\leq \tilde{c}$ for every vertex $x$ outside $C_i$. Thus, $|V(C_i)\cap V(C_j)|\leq \tilde{c}$ for distinct $i,j\in [\phi]$.
 By Definition \ref{asso}, 
  $s_\ell\in V(C_i)$ if $s_\ell$ is adjacent to $f_i$, for $\ell\in[\sigma]$ and $ i\in[\phi]$. Hence, the slim vertex $s_\ell$ has at most $\tilde{c}$ neighbors in $C_i$ if $s_\ell$ is not adjacent to $f_i$.
  
   Let $A_\phi:=\{s\in V_{\rm slim}(\ho)\mid$  $s$ is adjacent to $f_\phi\}$ and $A_\phi^*:= V_{\rm slim}(\ho)-A_\phi$. Note that $|A_\phi|+\tilde{c}|A_\phi^*|\leq 1+\tilde{c}(\sigma-1)$ as $|A_\phi|+|A_\phi^*|=\sigma$ and $\tilde{c}\geq 1$.  
   It follows that  $|V(C_{\phi})-A_\phi-\bigcup_{x\in A_\phi^*}N_G(x)-\bigcup_{i\neq \phi}\bigcup_{y\in D_i}N_G(y)-\bigcup_{i\neq \phi}V(C_i)|\geq q-(1+\tilde{c}(\sigma-1))-\tilde{c} p(\phi-1)-\tilde{c}(\phi-1)\geq p $ holds.
   Therefore, there is a clique $D_\phi$ satisfying
   \begin{enumerate}
   	\item[($b_1)$] 	$|V(D_\phi)|=p$ and $V(D_\phi)\subset V(C_{\phi})-A_\phi-\bigcup_{x\in A_{\phi}^*}N_G(x)-\bigcup_{i\neq \phi}\bigcup_{y\in D_i}N_G(y)-\bigcup_{i\neq \phi}V(C_i)$,
   	\item[($b_2$)]    for $i\in[\phi-1]$,  $V(D_\phi)\cap V_{\rm slim}(\ho)=\emptyset$, $V(D_i)\cap V(D_\phi)=\emptyset$, and no edge in $G$ exists between $D_i$ and $D_\phi$, 
   	\item[($b_3$)] for $\ell\in[\sigma]$, the slim vertex $s_\ell$ is adjacent to all vertices of $D_\phi$ if $s_\ell$ is adjacent to the fat vertex $f_\phi$, and $s_\ell$ has no neighbors in $D_\phi$ if $s_\ell$ is not adjacent to $f_\phi$.
   \end{enumerate}
   It follows that the subgraph  in $G$ induced by {$(\bigcup_{i=1}^\phi V(D_i))\cup\{s_1,\dots,s_\phi\}$} is isomorphic to $G(\ho, p)$.
\end{proof}

\begin{thm}\label{forbiddenthm}
	Let $c,\tilde{c}$ and $q$ be positive integers such that $\tilde{c}=\min\{c,6\}$ and $q\geq \max\{c+5, 50\tilde{c}+16\}$. If $G$ is a $\mu$-bounded graph with parameter $c$ and $\lambda_{\min}(G)\geq -3$, then the associated Hoffman graph $\go(G,q)$ does not contain an induced Hoffman subgraph with special matrix in $\mathcal{M}(2)$.
\end{thm}

\begin{proof}
By Definition \ref{asso}, two distinct slim vertices $u$ and $v$ in $\go(G,q)$ are adjacent if they have at least one common fat neighbor.
Assume that $\go(G,q)$ contains an induced Hoffman subgraph with special matrix in $\mathcal{M}(2)$.
Lemma \ref{m2} shows that $\go(G,q)$ contains an induced Hoffman subgraph in $\Ho=\{\ho_{1,-2},\ho_{3,1},\ho_{3,-1},\ho_{4,-2},\ho_{5},\ho_6,\ho_7,\ho_{8}^{(1)},\ho_8^{(2)}\}$.
%Therefore, we just need to show that $\go(G,q)$ does not contain an induced subgraph isomorphic to a member of $\Ho$.

Since $\lambda_{\min}(G)\geq-3$, by Proposition \ref{cal}, $G$ does not contain a member of $\{G(\ho_{1,-2},7)$, $G(\ho_{3,1},7)$, $G(\ho_{3,-1},13)$, $G(\ho_{4,-2},5)$, $G(\ho_{5},11)$, $G(\ho_{6},5)$, $G(\ho_{7},15)$, $G(\ho_{8}^{(1)},8)$, $G(\ho_{8}^{(2)},11)\}$ as an induced subgraph.
Note that $q\geq \max\{c+5, 50\tilde{c}+16\}\geq  \max\{n_2(4,1,7,\tilde{c})$, $n_2(5,2,7,\tilde{c})$, $n_2(3,2,13,\tilde{c})$, $n_2(3,2,5,\tilde{c})$, $n_2(2,3,11,\tilde{c})$, $n_2(4,3,5,\tilde{c})$, $n_2(4,3,15,\tilde{c})$, $n_2(6,3,8,\tilde{c})$, $n_2(5,3,11,\tilde{c})\}$, and $50\tilde{c}+16\geq n_1(3)=48$. It follows that $\go(G,q)$ does not contain an induced subgraph isomorphic to a member of $\Ho$ by Proposition \ref{assohoff}, which contradicts with Lemma \ref{m2}.
\end{proof}

\section{Main theorems}
In this section, we complete the proof of Theorem~\ref{intro2}. The following key lemma is inspired by the Bose--Laskar method, which was first introduced by Bose and Laskar in \cite{BL1967} in 1967. Using this approach, they established fundamental relationships between claws and cliques in edge-regular graphs, leading to a characterization of tetrahedral graphs. Subsequently, Huang \cite{Huang1987} applied this method to characterize bilinear forms graphs. Metsch later refined the Bose--Laskar method and applied it to improve Bruck's completion theorem \cite{Metsch1991}. Based on Metsch's refinement, further characterizations were obtained, including those of Grassmann graphs \cite{KLG2025,Metsch1995} and bilinear forms graphs \cite{Metsch1999}. The Bose--Laskar method is also closely related to the claw bound. This bound was first proved by Metsch \cite[Lemma~1.1]{Metsch1991}, and was later rediscovered by Godsil \cite[Lemma~2.3]{Godsil1993} and by Koolen and Park \cite[Lemma~2]{KP2010}. The claw bound plays an important role in the characterization of strongly regular graphs; see, for example, \cite{KLMP2025,Neumaier1979}.

\begin{lem}\label{blhlem}
	Let $c$ and $r$ be positive integers. If $G$ is a $\mu$-bounded graph with parameter $c$ and $\lambda_{\min}(G)\geq -\lambda$,
	then each vertex $x\in V(G)$ lies in a maximal clique $C_1(x)$ of order at least $\frac{d(x)-\binom{\lfloor\lambda^2\rfloor}{2}(c-1)}{\lfloor\lambda^2\rfloor}+1$.    
	Moreover, if every maximal clique containing $x$ has order at most $d(x)-r$, then $x$ lies in another distinct maximal clique $C_2(x)$ of order at least $\frac{r-\binom{\lfloor\lambda^2\rfloor}{2}(c-1)+1}{\lfloor\lambda^2\rfloor-1}+1$.
\end{lem}
\begin{proof} 
	Since $\lambda_{\min}(K_{1,\lfloor\lambda^2+1\rfloor})=-\sqrt{\lfloor\lambda^2+1\rfloor}<-\lambda$, $G$ is induced $K_{1,\lfloor\lambda^2+1\rfloor}$-free
	 by interlacing theorem. Hence, the maximum order of independent sets in $N(x)$ is at most $\lfloor\lambda^2\rfloor$ for any vertex $x \in V(G)$.	
	Let $I=\{v_1,v_2,...,v_{s}\}$ be a maximum independent set in $N(x)$, where $s\leq \lfloor\lambda^2\rfloor $. 
	
	Let $W=\{y \in N(x)\mid |N(y) \cap I| \geq 2 \}$.
	 %then $W\cap(\cup_{i=1}^{s} P_i) = \emptyset$. 
	 As each pair of two distinct non-adjacent vertices has at most $c$ common neighbors, $|W|\leq {s \choose 2}(c-1)\leq {\lfloor\lambda^2\rfloor \choose 2}(c-1)$.
	
	Let $P_i:=\{y \in N(x)\mid v_i\in N(y)\}\cup\{v_i\}-W$ for $i\in[s]$. 
	It follows that $P_i\cap P_j= \emptyset $ for $i \neq j$.
	If there are two distinct non-adjacent  vertices $y_1,y_2\in P_i$, then $(I \backslash \{v_i\})\cup\{y_1,y_2\}$ is an independent set of order $s+1> |I|=s$. 
	Hence, the subgraph induced by $P_i$ is a clique containing $v_i$ for $i\in[s]$.
	
	 As the disjoint union $W\cup P_1 \cdots\cup P_{s}$ is exactly the set $N(x)$, $|W|+\sum_{i = 1}^{s} |P_i|=d(x)$ holds. Hence, $\sum_{i = 1}^{s} |P_i| \geq d(x)-{\lfloor\lambda^2\rfloor \choose 2}(c-1)$.
	Assume that $|P_1|\geq|P_2|\geq\dots\geq |P_s|$, then the subgraph induced by $P_1\cup\{x\}$ is a clique of order at least $ \frac{d(x)-{\lfloor\lambda^2\rfloor \choose 2}(c-1)}{s}+1 \geq \frac{d(x)-{\lfloor\lambda^2\rfloor \choose 2}(c-1)}{\lfloor\lambda^2\rfloor}+1$. 
	
	If every maximal clique containing $x$ has order at most $d(x)-r$,
	then $|P_1|\leq d(x)-r-1$ and the subgraph induced by $P_2\cup\{x\}$ is a clique of order at least $ \frac{d(x)-{\lfloor\lambda^2\rfloor \choose 2}(c-1)-(d(x)-r-1)}{s-1}+1 \geq \frac{r-{\lfloor\lambda^2\rfloor\choose 2}(c-1)+1}{\lfloor\lambda^2\rfloor-1}+1$. 
\end{proof}
	Without using Ramsey theory, the above lemma helps us to find large cliques in $\mu$-bounded graphs which are necessary for Hoffman theory.
	\begin{cor}\label{blhcor}
		Let $c$ and $q$ be positive integers. 
		Let $G$ be a $\mu$-bounded graph with parameter $c$ and $\lambda_{\min}(G)\geq -3$. If every clique $C$ in $G$ has order at most $\min_{x\in V(C)}d(x)-36c+45-8q$, then the associated Hoffman graph $\go(G,q)$ is $2$-fat.
	\end{cor}
	\begin{proof}
	 By Lemma \ref{blhlem}, for each vertex $v$, there are two distinct maximal cliques $C_1(v)$ and $C_2(v)$ containing $v$ such that $|V(C_1(v))|\geq\frac{d(v)-\binom{9}{2}(c-1)}{9}+1\geq \frac{d(v)-36c+45}{9}$ and $|V(C_2(v))|\geq\frac{1}{8}(d(v)-\min_{x\in V(C_2(v))}d(x)+36c-45+8q-\binom{9}{2}(c-1)+1)+1\geq \frac{1}{8}(36c-45+8q-\binom{9}{2}(c-1)+1)+1=q$.
	 Note that $d(v)\geq 9q+36c-45$ as $q\leq |V(C_2(v))|\leq \min_{x\in V(C_2(v))}d(x)-36c+45-8q$. It follows that $|V(C_1(v))|\geq \frac{d(v)-36c+45}{9}\geq q$.
	 In other words, each vertex is contained in two maximal cliques with order at least $q$, and thus $\go(G,q)$ is $2$-fat by Definition \ref{asso}.
	\end{proof}
	
	Now, we are in the position to prove Theorem \ref{intro2}.
	
	\begin{proof}[\bf Proof of Theorem \ref{intro2}]
		Let $q:= \max\{c+5, 50\tilde{c}+16\}$.
		By the condition (ii), every clique $C$ in $G$ has order at most $\min_{x\in V(C)}d(x)-36c+45-8q$.
		Thus, the associated Hoffman graph $\go(G,q)$ is $2$-fat  by  Corollary \ref{blhcor}.
		By Theorem \ref{forbiddenthm}, $\go(G,q)$ does not contain an induced Hoffman subgraph whose special matrix is a member of $\mathcal{M}(2)$. So, the associated Hoffman graph $\go(G,q)$ is
		a $2$-fat $\Go(2)$-line Hoffman graph by Theorem \ref{line graphhjk}. 

Notice that a $\Go(2)$-line Hoffman graph is a \big\{\raisebox{-1ex}{\begin{tikzpicture}[scale=0.3]
			
			\tikzstyle{every node}=[draw,circle,fill=black,minimum size=10pt,scale=0.3,
			inner sep=0pt]
			
			\draw (-2.1,0) node (1f1) [label=below:$$] {};
			\draw (-1.6,0) node (1f2) [label=below:$$] {};
			\draw (-1.1,0) node (1f3) [label=below:$$] {};

			\tikzstyle{every node}=[draw,circle,fill=black,minimum size=5pt,scale=0.3,
			inner sep=0pt]

			\draw (-1.6,1) node (1s1) [label=below:$$] {};
			
			\draw (1f1) -- (1s1) -- (1f2);
			\draw (1f3) -- (1s1);
	\end{tikzpicture}},\hspace{-0.08cm}
	\raisebox{-1ex}{\begin{tikzpicture}[scale=0.3]
			\tikzstyle{every node}=[draw,circle,fill=black,minimum size=10pt,scale=0.3,
			inner sep=0pt]

			\draw (-0.5,0) node (2f1) [label=below:$$] {};
			\draw (0.5,0) node (2f2) [label=below:$$] {};
			\draw (-0.5,1) node (2f3) [label=below:$$] {};
			\draw (0.5,1) node (2f4) [label=below:$$] {};

			\tikzstyle{every node}=[draw,circle,fill=black,minimum size=5pt,scale=0.3,
			inner sep=0pt]

			\draw (0,0.2) node (2s1) [label=below:$$] {};
			\draw (0.3,0.5) node (2s2) [label=below:$$] {};
			\draw (-0.3,0.5) node (2s3) [label=below:$$] {};
			\draw (0,0.8) node (2s4) [label=below:$$] {};

			\draw (2f1) -- (2s1) -- (2f2) -- (2s2) -- (2f4) -- (2s4) -- (2f3) -- (2s3) -- (2f1);
			\draw (2s1) -- (2s4);
			\draw (2s2) -- (2s3);
	\end{tikzpicture}},\hspace{-0.08cm}
	\raisebox{-1ex}{\begin{tikzpicture}[scale=0.3]
			\tikzstyle{every node}=[draw,circle,fill=black,minimum size=10pt,scale=0.3,
			inner sep=0pt]

			\draw (1.5,0) node (3f1) [label=below:$$] {};
			\draw (1.5,1) node (3f2) [label=below:$$] {};

			\tikzstyle{every node}=[draw,circle,fill=black,minimum size=5pt,scale=0.3,
			inner sep=0pt]

			\draw (1,0.5) node (3s1) [label=below:$$] {};
			\draw (2,0.5) node (3s2) [label=below:$$] {};
			
			\draw (3s1) -- (3s2);
			\draw (3f1) -- (3s1) -- (3f2) -- (3s2) -- (3f1);
	\end{tikzpicture}}\big\}-line Hoffman graph by 
Lemma \ref{b2linehoff}. This means that, by Definition \ref{line}, $\go(G,q)$ is a $2$-fat induced Hoffman subgraph of a Hoffman graph $\ho=\uplus_{i=1}^{v}\ho^i$, where $\ho^i$ is isomorphic to a $2$-fat induced Hoffman subgraph of a Hoffman graph in \big\{\raisebox{-1ex}{\begin{tikzpicture}[scale=0.3]
			
			\tikzstyle{every node}=[draw,circle,fill=black,minimum size=10pt,scale=0.3,
			inner sep=0pt]
			
			\draw (-2.1,0) node (1f1) [label=below:$$] {};
			\draw (-1.6,0) node (1f2) [label=below:$$] {};
			\draw (-1.1,0) node (1f3) [label=below:$$] {};

			\tikzstyle{every node}=[draw,circle,fill=black,minimum size=5pt,scale=0.3,
			inner sep=0pt]

			\draw (-1.6,1) node (1s1) [label=below:$$] {};
			
			\draw (1f1) -- (1s1) -- (1f2);
			\draw (1f3) -- (1s1);
	\end{tikzpicture}},\hspace{-0.08cm}
	\raisebox{-1ex}{\begin{tikzpicture}[scale=0.3]
			\tikzstyle{every node}=[draw,circle,fill=black,minimum size=10pt,scale=0.3,
			inner sep=0pt]

			\draw (-0.5,0) node (2f1) [label=below:$$] {};
			\draw (0.5,0) node (2f2) [label=below:$$] {};
			\draw (-0.5,1) node (2f3) [label=below:$$] {};
			\draw (0.5,1) node (2f4) [label=below:$$] {};

			\tikzstyle{every node}=[draw,circle,fill=black,minimum size=5pt,scale=0.3,
			inner sep=0pt]

			\draw (0,0.2) node (2s1) [label=below:$$] {};
			\draw (0.3,0.5) node (2s2) [label=below:$$] {};
			\draw (-0.3,0.5) node (2s3) [label=below:$$] {};
			\draw (0,0.8) node (2s4) [label=below:$$] {};

			\draw (2f1) -- (2s1) -- (2f2) -- (2s2) -- (2f4) -- (2s4) -- (2f3) -- (2s3) -- (2f1);
			\draw (2s1) -- (2s4);
			\draw (2s2) -- (2s3);
	\end{tikzpicture}},\hspace{-0.08cm}
	\raisebox{-1ex}{\begin{tikzpicture}[scale=0.3]
			\tikzstyle{every node}=[draw,circle,fill=black,minimum size=10pt,scale=0.3,
			inner sep=0pt]

			\draw (1.5,0) node (3f1) [label=below:$$] {};
			\draw (1.5,1) node (3f2) [label=below:$$] {};

			\tikzstyle{every node}=[draw,circle,fill=black,minimum size=5pt,scale=0.3,
			inner sep=0pt]

			\draw (1,0.5) node (3s1) [label=below:$$] {};
			\draw (2,0.5) node (3s2) [label=below:$$] {};
			
			\draw (3s1) -- (3s2);
			\draw (3f1) -- (3s1) -- (3f2) -- (3s2) -- (3f1);
	\end{tikzpicture}}\big\}
	 for $i\in[v]$.
	 
	 Assume that $\ho^i$ is isomorphic to a $2$-fat induced Hoffman subgraph of 	\raisebox{-1ex}{\begin{tikzpicture}[scale=0.3]
	 		\tikzstyle{every node}=[draw,circle,fill=black,minimum size=10pt,scale=0.3,
	 		inner sep=0pt]

	 		\draw (-0.5,0) node (2f1) [label=below:$$] {};
	 		\draw (0.5,0) node (2f2) [label=below:$$] {};
	 		\draw (-0.5,1) node (2f3) [label=below:$$] {};
	 		\draw (0.5,1) node (2f4) [label=below:$$] {};

	 		\tikzstyle{every node}=[draw,circle,fill=black,minimum size=5pt,scale=0.3,
	 		inner sep=0pt]

	 		\draw (0,0.2) node (2s1) [label=below:$$] {};
	 		\draw (0.3,0.5) node (2s2) [label=below:$$] {};
	 		\draw (-0.3,0.5) node (2s3) [label=below:$$] {};
	 		\draw (0,0.8) node (2s4) [label=below:$$] {};

	 		\draw (2f1) -- (2s1) -- (2f2) -- (2s2) -- (2f4) -- (2s4) -- (2f3) -- (2s3) -- (2f1);
	 		\draw (2s1) -- (2s4);
	 		\draw (2s2) -- (2s3);
	 \end{tikzpicture}}.
{It follows that $\ho^i$ belongs to the set $\mathcal{S}:=$}
 \big\{\raisebox{-1ex}{\begin{tikzpicture}[scale=0.3]
 		
 		\tikzstyle{every node}=[draw,circle,fill=black,minimum size=10pt,scale=0.3,
 		inner sep=0pt]
 		
 		\draw (-2.1,-0.4) node (1f1) [label=below:$$] {};
 		%	 	\draw (-1.6,0) node (1f2) [label=below:$$] {};
 		\draw (-1.1,-0.4) node (1f3) [label=below:$$] {};

 		\tikzstyle{every node}=[draw,circle,fill=black,minimum size=5pt,scale=0.3,
 		inner sep=0pt]

 		\draw (-1.6,0.6) node (1s1) [label=below:$$] {};
 		
 		\draw (1f1) -- (1s1) -- (1f3);
 		%	 	\draw (1f3) -- (1s1);
 \end{tikzpicture}},
	\raisebox{-1ex}{\begin{tikzpicture}[scale=0.3]
		\tikzstyle{every node}=[draw,circle,fill=black,minimum size=10pt,scale=0.3,
		inner sep=0pt]

		\draw (-0.5,0) node (2f1) [label=below:$$] {};
		\draw (0.5,0) node (2f2) [label=below:$$] {};
		\draw (-0.5,1) node (2f3) [label=below:$$] {};
		\draw (0.5,1) node (2f4) [label=below:$$] {};

		\tikzstyle{every node}=[draw,circle,fill=black,minimum size=5pt,scale=0.3,
		inner sep=0pt]

		%		\draw (0,0.2) node (2s1) [label=below:$$] {};
		\draw (0.3,0.5) node (2s2) [label=below:$$] {};
		\draw (-0.3,0.5) node (2s3) [label=below:$$] {};
		%		\draw (0,0.8) node (2s4) [label=below:$$] {};

		\draw (2f1) -- (2s3) -- (2f3);
		\draw  (2f4) -- (2s2) -- (2f2);
		\draw (2s2) -- (2s3);
		%		\draw (2s2) -- (2s3);
\end{tikzpicture}},
\raisebox{-1ex}{\begin{tikzpicture}[scale=0.3]
		\tikzstyle{every node}=[draw,circle,fill=black,minimum size=10pt,scale=0.3,
		inner sep=0pt]

		\draw (-0.5,0) node (2f1) [label=below:$$] {};
		
		\draw (-0.5,1) node (2f3) [label=below:$$] {};
		\draw (0.5,1) node (2f4) [label=below:$$] {};

		\tikzstyle{every node}=[draw,circle,fill=black,minimum size=5pt,scale=0.3,
		inner sep=0pt]

		\draw (-0.3,0.5) node (2s3) [label=below:$$] {};
		\draw (0,0.8) node (2s4) [label=below:$$] {};

		\draw (2f1) -- (2s3) -- (2f3) -- (2s4) -- (2f4);
\end{tikzpicture}},
\raisebox{-1ex}{\begin{tikzpicture}[scale=0.3]
		\tikzstyle{every node}=[draw,circle,fill=black,minimum size=10pt,scale=0.3,
		inner sep=0pt]

		\draw (-0.5,0) node (2f1) [label=below:$$] {};
		\draw (0.5,0) node (2f2) [label=below:$$] {};
		\draw (-0.5,1) node (2f3) [label=below:$$] {};
		\draw (0.5,1) node (2f4) [label=below:$$] {};

		\tikzstyle{every node}=[draw,circle,fill=black,minimum size=5pt,scale=0.3,
		inner sep=0pt]

		\draw (0.3,0.5) node (2s2) [label=below:$$] {};
		\draw (-0.3,0.5) node (2s3) [label=below:$$] {};
		\draw (0,0.8) node (2s4) [label=below:$$] {};

		\draw  (2f2) -- (2s2) -- (2f4) -- (2s4) -- (2f3) -- (2s3) -- (2f1);
		\draw (2s2) -- (2s3);
\end{tikzpicture}},
\raisebox{-1ex}{\begin{tikzpicture}[scale=0.3]
		\tikzstyle{every node}=[draw,circle,fill=black,minimum size=10pt,scale=0.3,
		inner sep=0pt]

		\draw (-0.5,0) node (2f1) [label=below:$$] {};
		\draw (0.5,0) node (2f2) [label=below:$$] {};
		\draw (-0.5,1) node (2f3) [label=below:$$] {};
		\draw (0.5,1) node (2f4) [label=below:$$] {};

		\tikzstyle{every node}=[draw,circle,fill=black,minimum size=5pt,scale=0.3,
		inner sep=0pt]

		\draw (0,0.2) node (2s1) [label=below:$$] {};
		\draw (0.3,0.5) node (2s2) [label=below:$$] {};
		\draw (-0.3,0.5) node (2s3) [label=below:$$] {};
		\draw (0,0.8) node (2s4) [label=below:$$] {};

		\draw (2f1) -- (2s1) -- (2f2) -- (2s2) -- (2f4) -- (2s4) -- (2f3) -- (2s3) -- (2f1);
		\draw (2s1) -- (2s4);
		\draw (2s2) -- (2s3);
\end{tikzpicture}}\hspace{-0.08cm}
\big\}.
{In the following, we will show that $\ho^i$ is isomorphic to either}
	\raisebox{-1ex}{\begin{tikzpicture}[scale=0.3]
			
			\tikzstyle{every node}=[draw,circle,fill=black,minimum size=10pt,scale=0.3,
			inner sep=0pt]
			
			\draw (-2.1,-0.4) node (1f1) [label=below:$$] {};
			%	 	\draw (-1.6,0) node (1f2) [label=below:$$] {};
			\draw (-1.1,-0.4) node (1f3) [label=below:$$] {};

			\tikzstyle{every node}=[draw,circle,fill=black,minimum size=5pt,scale=0.3,
			inner sep=0pt]

			\draw (-1.6,0.6) node (1s1) [label=below:$$] {};
			
			\draw (1f1) -- (1s1) -- (1f3);
			%	 	\draw (1f3) -- (1s1);
	\end{tikzpicture}}  or
	\raisebox{-1ex}{\begin{tikzpicture}[scale=0.3]
			\tikzstyle{every node}=[draw,circle,fill=black,minimum size=10pt,scale=0.3,
			inner sep=0pt]

			\draw (-0.5,0) node (2f1) [label=below:$$] {};
			\draw (0.5,0) node (2f2) [label=below:$$] {};
			\draw (-0.5,1) node (2f3) [label=below:$$] {};
			\draw (0.5,1) node (2f4) [label=below:$$] {};

			\tikzstyle{every node}=[draw,circle,fill=black,minimum size=5pt,scale=0.3,
			inner sep=0pt]

			%		\draw (0,0.2) node (2s1) [label=below:$$] {};
			\draw (0.3,0.5) node (2s2) [label=below:$$] {};
			\draw (-0.3,0.5) node (2s3) [label=below:$$] {};
			%		\draw (0,0.8) node (2s4) [label=below:$$] {};

			\draw (2f1) -- (2s3) -- (2f3);
			\draw  (2f4) -- (2s2) -- (2f2);
			\draw (2s2) -- (2s3);
			%		\draw (2s2) -- (2s3);
	\end{tikzpicture}}.
{It suffices to prove that each pair of slim vertices in $\ho^i$ is adjacent.}

%\textcolor{blue}{	where both  two	are  $2$-fat} 
%\raisebox{-1ex}{\begin{tikzpicture}[scale=0.3]
			
%			\tikzstyle{every node}=[draw,circle,fill=black,minimum size=10pt,scale=0.3,
%			inner sep=0pt]
			
%			\draw (-2.1,0) node (1f1) [label=below:$$] {};
%			\draw (-1.6,0) node (1f2) [label=below:$$] {};
%			\draw (-1.1,0) node (1f3) [label=below:$$] {};

%			\tikzstyle{every node}=[draw,circle,fill=black,minimum size=5pt,scale=0.3,
%			inner sep=0pt]

%			\draw (-1.6,1) node (1s1) [label=below:$$] {};
			
%			\draw (1f1) -- (1s1) -- (1f2);
%			\draw (1f3) -- (1s1);
%	\end{tikzpicture}}\textcolor{blue}{-line Hoffman graphs. 
 
 { If there are two distinct non-adjacent slim vertices $s_1,s_2$ in $\ho^i$, then they have a common fat neighbor $f$ in $\ho_i$, as $\ho_i\in\mathcal{S}$.
For $j=1,2$, since $s_j$ %({\em resp.} $s_2$)
 has exactly two fat neighbors in $\ho^i$, we may let $f_j$ %({\em resp.} $f_2$) 
 denote the other one.
By Lemma \ref{combi} (iii), the set $N_{\ho}^{{\rm fat}}(s_j)$ of fat neighbors of $s_j$ in $\ho$ is contained in $N_{\ho^i}^{{\rm fat}}(s_j)$ for $j=1,2$, thus $ N_{\ho}^{{\rm fat}}(s_j)= N_{\ho^i}^{{\rm fat}}(s_j)$.
By Definition \ref{line}, $\go(G,q)$ is a $2$-fat induced Hoffman subgraph of $\ho$. It follows that 
%$N_{\go(G,q)}^{{\rm fat}}(s_j)\subset N_{\ho}^{{\rm fat}}(s_j)\subset N_{\ho^i}^{{\rm fat}}(s_j)$.
%As $\go(G,q)$ is $2$-fat, we have 
$2\leq|N_{\go(G,q)}^{{\rm fat}}(s_j)|\leq |N_{\ho}^{{\rm fat}}(s_j)| = |N_{\ho^i}^{{\rm fat}}(s_j)|=2$, and 
thus $N_{\go(G,q)}^{{\rm fat}}(s_j)=N_{\ho^i}^{{\rm fat}}(s_j)=\{f,f_j\}$ for $j=1,2$.
Therefore the two non-adjacent slim vertices $s_1$ and $s_2$ have a common fat neighbor $f$ in $\go(G,q)$.
However, this contradicts Definition \ref{asso} (iii), which states that the set of slim neighbors of a fat vertex in  $\go(G,q)$ forms a clique in $G$.
Thus, $\ho^i$ is isomorphic to either}
\raisebox{-1ex}{\begin{tikzpicture}[scale=0.3]
	
	\tikzstyle{every node}=[draw,circle,fill=black,minimum size=10pt,scale=0.3,
	inner sep=0pt]
	
	\draw (-2.1,-0.4) node (1f1) [label=below:$$] {};
	%	 	\draw (-1.6,0) node (1f2) [label=below:$$] {};
	\draw (-1.1,-0.4) node (1f3) [label=below:$$] {};

	\tikzstyle{every node}=[draw,circle,fill=black,minimum size=5pt,scale=0.3,
	inner sep=0pt]

	\draw (-1.6,0.6) node (1s1) [label=below:$$] {};
	
	\draw (1f1) -- (1s1) -- (1f3);
	%	 	\draw (1f3) -- (1s1);
	\end{tikzpicture}}  or
	\raisebox{-1ex}{\begin{tikzpicture}[scale=0.3]
	\tikzstyle{every node}=[draw,circle,fill=black,minimum size=10pt,scale=0.3,
	inner sep=0pt]

	\draw (-0.5,0) node (2f1) [label=below:$$] {};
	\draw (0.5,0) node (2f2) [label=below:$$] {};
	\draw (-0.5,1) node (2f3) [label=below:$$] {};
	\draw (0.5,1) node (2f4) [label=below:$$] {};

	\tikzstyle{every node}=[draw,circle,fill=black,minimum size=5pt,scale=0.3,
	inner sep=0pt]

	%		\draw (0,0.2) node (2s1) [label=below:$$] {};
	\draw (0.3,0.5) node (2s2) [label=below:$$] {};
	\draw (-0.3,0.5) node (2s3) [label=below:$$] {};
	%		\draw (0,0.8) node (2s4) [label=below:$$] {};

	\draw (2f1) -- (2s3) -- (2f3);
	\draw  (2f4) -- (2s2) -- (2f2);
	\draw (2s2) -- (2s3);
	%		\draw (2s2) -- (2s3);
	\end{tikzpicture}}.

Note that 	\raisebox{-1ex}{\begin{tikzpicture}[scale=0.3]
		\tikzstyle{every node}=[draw,circle,fill=black,minimum size=10pt,scale=0.3,
		inner sep=0pt]

		\draw (-0.5,0) node (2f1) [label=below:$$] {};
		\draw (0.5,0) node (2f2) [label=below:$$] {};
		\draw (-0.5,1) node (2f3) [label=below:$$] {};
		\draw (0.5,1) node (2f4) [label=below:$$] {};

		\tikzstyle{every node}=[draw,circle,fill=black,minimum size=5pt,scale=0.3,
		inner sep=0pt]

		%		\draw (0,0.2) node (2s1) [label=below:$$] {};
		\draw (0.3,0.5) node (2s2) [label=below:$$] {};
		\draw (-0.3,0.5) node (2s3) [label=below:$$] {};
		%		\draw (0,0.8) node (2s4) [label=below:$$] {};

		\draw (2f1) -- (2s3) -- (2f3);
		\draw  (2f4) -- (2s2) -- (2f2);
		\draw (2s2) -- (2s3);
		%		\draw (2s2) -- (2s3);
\end{tikzpicture}} { is an induced Hoffman subgraph of}
\raisebox{-1ex}{\begin{tikzpicture}[scale=0.3]
		
		\tikzstyle{every node}=[draw,circle,fill=black,minimum size=10pt,scale=0.3,
		inner sep=0pt]
		
		\draw (-1.8,0) node (1f1) [label=below:$$] {};
		\draw (-1.2,0) node (1f2) [label=below:$$] {};
		\draw (-0.6,0) node (1f3) [label=below:$$] {};
		\draw (0,0) node (1f4) [label=below:$$] {};
		\draw (0.6,0) node (1f5) [label=below:$$] {};
		
		\tikzstyle{every node}=[draw,circle,fill=black,minimum size=5pt,scale=0.3,
		inner sep=0pt]

		\draw (-1.2,1) node (1s1) [label=below:$$] {};
		\draw (0,1) node (1s2) [label=below:$$] {};
		
		\draw (1f1) -- (1s1) -- (1f2) -- (1s1) -- (1f3) -- (1s2) --(1f4) -- (1s2) -- (1f5);
		\draw (1s1) -- (1s2);
\end{tikzpicture}} $=$
\raisebox{-1ex}{\begin{tikzpicture}[scale=0.3]
		
		\tikzstyle{every node}=[draw,circle,fill=black,minimum size=10pt,scale=0.3,
		inner sep=0pt]
		
		\draw (-2.1,0) node (1f1) [label=below:$$] {};
		\draw (-1.6,0) node (1f2) [label=below:$$] {};
		\draw (-1.1,0) node (1f3) [label=below:$$] {};

		\tikzstyle{every node}=[draw,circle,fill=black,minimum size=5pt,scale=0.3,
		inner sep=0pt]

		\draw (-1.6,1) node (1s1) [label=below:$$] {};
		
		\draw (1f1) -- (1s1) -- (1f2);
		\draw (1f3) -- (1s1);
\end{tikzpicture}} $\uplus$
\raisebox{-1ex}{\begin{tikzpicture}[scale=0.3]

\tikzstyle{every node}=[draw,circle,fill=black,minimum size=10pt,scale=0.3,
inner sep=0pt]

\draw (-2.1,0) node (1f1) [label=below:$$] {};
\draw (-1.6,0) node (1f2) [label=below:$$] {};
\draw (-1.1,0) node (1f3) [label=below:$$] {};

\tikzstyle{every node}=[draw,circle,fill=black,minimum size=5pt,scale=0.3,
inner sep=0pt]

\draw (-1.6,1) node (1s1) [label=below:$$] {};

\draw (1f1) -- (1s1) -- (1f2);
\draw (1f3) -- (1s1);
\end{tikzpicture}}.
It follows that both \raisebox{-1ex}{\begin{tikzpicture}[scale=0.3]
		
		\tikzstyle{every node}=[draw,circle,fill=black,minimum size=10pt,scale=0.3,
		inner sep=0pt]
		
		\draw (-2.1,-0.4) node (1f1) [label=below:$$] {};
		%	 	\draw (-1.6,0) node (1f2) [label=below:$$] {};
		\draw (-1.1,-0.4) node (1f3) [label=below:$$] {};

		\tikzstyle{every node}=[draw,circle,fill=black,minimum size=5pt,scale=0.3,
		inner sep=0pt]

		\draw (-1.6,0.6) node (1s1) [label=below:$$] {};
		
		\draw (1f1) -- (1s1) -- (1f3);
		%	 	\draw (1f3) -- (1s1);
\end{tikzpicture}}  and
\raisebox{-1ex}{\begin{tikzpicture}[scale=0.3]
		\tikzstyle{every node}=[draw,circle,fill=black,minimum size=10pt,scale=0.3,
		inner sep=0pt]

		\draw (-0.5,0) node (2f1) [label=below:$$] {};
		\draw (0.5,0) node (2f2) [label=below:$$] {};
		\draw (-0.5,1) node (2f3) [label=below:$$] {};
		\draw (0.5,1) node (2f4) [label=below:$$] {};

		\tikzstyle{every node}=[draw,circle,fill=black,minimum size=5pt,scale=0.3,
		inner sep=0pt]

		%		\draw (0,0.2) node (2s1) [label=below:$$] {};
		\draw (0.3,0.5) node (2s2) [label=below:$$] {};
		\draw (-0.3,0.5) node (2s3) [label=below:$$] {};
		%		\draw (0,0.8) node (2s4) [label=below:$$] {};

		\draw (2f1) -- (2s3) -- (2f3);
		\draw  (2f4) -- (2s2) -- (2f2);
		\draw (2s2) -- (2s3);
		%		\draw (2s2) -- (2s3);
\end{tikzpicture}} are $2$-fat \raisebox{-1ex}{\begin{tikzpicture}[scale=0.3]
	
	\tikzstyle{every node}=[draw,circle,fill=black,minimum size=10pt,scale=0.3,
	inner sep=0pt]
	
	\draw (-2.1,0) node (1f1) [label=below:$$] {};
	\draw (-1.6,0) node (1f2) [label=below:$$] {};
	\draw (-1.1,0) node (1f3) [label=below:$$] {};

	\tikzstyle{every node}=[draw,circle,fill=black,minimum size=5pt,scale=0.3,
	inner sep=0pt]

	\draw (-1.6,1) node (1s1) [label=below:$$] {};
	
	\draw (1f1) -- (1s1) -- (1f2);
	\draw (1f3) -- (1s1);
	\end{tikzpicture}}-line Hoffman graphs. 	 This shows that $\go(G,q)$ is a $2$-fat \big\{\raisebox{-1ex}{\begin{tikzpicture}[scale=0.3]
		
		\tikzstyle{every node}=[draw,circle,fill=black,minimum size=10pt,scale=0.3,
		inner sep=0pt]
		
		\draw (-2.1,0) node (1f1) [label=below:$$] {};
		\draw (-1.6,0) node (1f2) [label=below:$$] {};
		\draw (-1.1,0) node (1f3) [label=below:$$] {};

		\tikzstyle{every node}=[draw,circle,fill=black,minimum size=5pt,scale=0.3,
		inner sep=0pt]

		\draw (-1.6,1) node (1s1) [label=below:$$] {};
		
		\draw (1f1) -- (1s1) -- (1f2);
		\draw (1f3) -- (1s1);
		\end{tikzpicture}},\hspace{-0.08cm}
		\raisebox{-1ex}{\begin{tikzpicture}[scale=0.3]
		\tikzstyle{every node}=[draw,circle,fill=black,minimum size=10pt,scale=0.3,
		inner sep=0pt]

		\draw (1.5,0) node (3f1) [label=below:$$] {};
		\draw (1.5,1) node (3f2) [label=below:$$] {};

		\tikzstyle{every node}=[draw,circle,fill=black,minimum size=5pt,scale=0.3,
		inner sep=0pt]

		\draw (1,0.5) node (3s1) [label=below:$$] {};
		\draw (2,0.5) node (3s2) [label=below:$$] {};
		
		\draw (3s1) -- (3s2);
		\draw (3f1) -- (3s1) -- (3f2) -- (3s2) -- (3f1);
		\end{tikzpicture}}\big\}-line Hoffman graphs.
					\end{proof}
%		By Definition \ref{asso}, if distinct slim vertices $u$ and $v$ in $\go(G,q)$ have a common fat neighbor, then $u$ is adjacent to $v$. Note that there are only two Hoffman graphs $\mathfrak{f}_1$ and $\mathfrak{f}_3$ in $\Go(2)$ satisfy this property. Hence $G$ is the slim graph of a $2$-fat $\{\mathfrak{f}_1,\mathfrak{f}_3\}$-line Hoffman graph.

%			By Lemma \ref{combi}, if $y \notin V_{\rm slim}(\ho^1)$, then there are exactly one maximal clique containing $x$ and $y$ if $x$ and $y$ are adjacent.
%		 If $\ho^1$ is isomorphic to $\ho^{(3)}$, then there are exactly three maximal cliques in $G$ containing $x$ such that
%		 the union of these cliques is $\{x\}\cup N(x)$ and $x$ is the only element in the intersection of any two of these cliques.
%		  If $\ho^1$ is isomorphic to $\mathfrak{f}_3$, then there are exactly two maximal cliques $C_1$ and $C_2$ in $G$ containing $x$ such that $C_1\cup C_2=\{x\}\cup N(x)$ and $|C_1\cap C_2|=2$.
%		 If $\ho^1$ is isomorphic to  $\ho^{(2)}$, there are exactly three cliques $C_1,C_2$ and $C_3$ in $G$ containing $x$ such that the union of these cliques is $\{x\}\cup N(x)$ and $x$ is the only element in the intersection of any two of these cliques, which $C_1$ and $C_2$ are maximal cliques with order at least $\max\{c+5, 50\tilde{c}+16\}$ and $|C_3|<\max\{c+5, 50\tilde{c}+16\}$ ($C_3=\{x\}$ is allowed).

	As we mentioned in Section \ref{Forbiddenmatrice}, the smallest eigenvalue of each matrix in $\mathcal{M}(2)$ is at most $-2-\sqrt{2}$. We give the following result without a proof, as it can be proved by Proposition \ref{epsilon}, Proposition \ref{assohoff}, and a similar discussion of the proof for Theorem \ref{intro2}.
	\begin{thm}\label{corep}
			Let $c$ be a positive integer. For $\epsilon>0$, there exists a positive integer $K:=K(\epsilon,c)$
			 such that if a graph $G$ satisfies the following conditions:
		\begin{enumerate}
			\item $G$ is a $\mu$-bounded graph with parameter $c$,
			\item  every clique $C$ has order at most $\min_{x\in V(C)}d(x)-K$,
			\item $\lambda_{\min} (G) \geq -2-\sqrt{2}+\epsilon$,
		\end{enumerate}
		then the associated Hoffman graph $\go(G,q)$ is a $2$-fat \big\{\raisebox{-1ex}{\begin{tikzpicture}[scale=0.3]
			
			\tikzstyle{every node}=[draw,circle,fill=black,minimum size=10pt,scale=0.3,
			inner sep=0pt]
			
			\draw (-2.1,0) node (1f1) [label=below:$$] {};
			\draw (-1.6,0) node (1f2) [label=below:$$] {};
			\draw (-1.1,0) node (1f3) [label=below:$$] {};

			\tikzstyle{every node}=[draw,circle,fill=black,minimum size=5pt,scale=0.3,
			inner sep=0pt]

			\draw (-1.6,1) node (1s1) [label=below:$$] {};
			
			\draw (1f1) -- (1s1) -- (1f2);
			\draw (1f3) -- (1s1);
	\end{tikzpicture}},\hspace{-0.08cm}
	\raisebox{-1ex}{\begin{tikzpicture}[scale=0.3]
			\tikzstyle{every node}=[draw,circle,fill=black,minimum size=10pt,scale=0.3,
			inner sep=0pt]

			\draw (1.5,0) node (3f1) [label=below:$$] {};
			\draw (1.5,1) node (3f2) [label=below:$$] {};

			\tikzstyle{every node}=[draw,circle,fill=black,minimum size=5pt,scale=0.3,
			inner sep=0pt]

			\draw (1,0.5) node (3s1) [label=below:$$] {};
			\draw (2,0.5) node (3s2) [label=below:$$] {};
			
			\draw (3s1) -- (3s2);
			\draw (3f1) -- (3s1) -- (3f2) -- (3s2) -- (3f1);
	\end{tikzpicture}}\big\}
	-line Hoffman graph and $\lambda_{\min}(G)\geq -3$.
	\end{thm}
	
	\begin{re}
		The parameter $K=K(\epsilon,c)$ can be bounded by $\frac{500}{\epsilon}+55c$ when $\epsilon$ tends to zero.
	\end{re}
	
	The following proposition generalizes Proposition \ref{muassle2}.
%	Let $n_1(\lambda):=\lambda^4-2\lambda^3+3\lambda^2-3\lambda+3$ and $n_2(\phi, \sigma,  p,\tilde{c}):=\tilde{c}(\sigma-1)+\tilde{c}(p+1)(\phi-1)+p+1$ be the functions defined in Proposition \ref{assohoff}.
	
	\begin{pro}	Let $\lambda\geq2$ and $c$ be positive integers and $\tilde{c}:=\min\{c,\lambda(\lambda-1)\}$. 
		Let $n_1(\lambda):=\lambda^4-2\lambda^3+3\lambda^2-3\lambda+3$ and $n_2(\phi, \sigma,  p,\tilde{c}):=\tilde{c}(\sigma-1)+\tilde{c}(p+1)(\phi-1)+p+1$.
    Let $q\geq\max\{n_1(\lambda),n_2(\lambda+1,1,\lambda(\lambda-1)+1,\tilde{c}),{n_2(\lambda,2,(\lambda+1)(\lambda-1)^2+1,\tilde{c}),(\lambda-1)^2+c+1\}}$. 
		Let $R(c,\lambda,q):=\lambda^2(q-1)+(c-1)\binom{\lambda^2}{2}$.
		If $G$ is a $\mu$-bounded graph with parameter $c$ and $\lambda_{\min}(G)\geq-\lambda$, then there exists a set of maximal cliques $\Po=\{P_1,\ldots,P_s\}$ in $G$ satisfying the following properties:
		\begin{enumerate}
			\item $|N(x)-\bigcup_{i=1}^s\{y\mid x,y\in V(P_i)\}|< R(c,\lambda,q)$ for $x\in V(G)$,
			\item  $|\{i\mid x\in V(P_i)\}|\leq\lambda$ for $x\in V(G)$,{ and if the equality holds for $x$, then $N(x)=\bigcup_{i=1}^s\{y\mid x,y\in V(P_i)\}-\{x\}$,}
			\item  $|V(P_i)\cap V(P_j)|\leq \lambda-1$ for $i\neq j$,
			\item $|V(P_i)|\geq q$ for $i\in[s]$.
		\end{enumerate}
	\end{pro}
	
	\begin{proof}
		Let $\Po$ be { the} set of all maximal cliques with order at least $q$ in $G$. If there exists a vertex $x\in V(G)$ such that 
		$|N(x)-\bigcup_{i=1}^s\{y\mid x,y\in V(P_i)\}|\geq R(c,\lambda,q)$. 
%		\textcolor{blue}{Let $H$ be a graph such that $V(H)=V(G)$ and $E(H)=E(G)\setminus(\bigcup E(P_i))$. It follows that $H$ }
		By Lemma \ref{blhlem}, there exists a maximal clique $P_{s+1}\notin\Po$ containing $x$ of order at least $\frac{R(c,\lambda,q)-\binom{\lambda^2}{2}(c-1)}{\lambda^2}+1$.
		{Since $R(c,\lambda,q):=\lambda^2(q-1)+(c-1)\binom{\lambda^2}{2}$, it follows that the order of $P_{s+1}$ is at least $q$.}
		 This contradicts the definition of $\Po$. Therefore, the properties (i) and (iv) hold.
		
		Let $p_1:=\lambda(\lambda-1)+1$, $p_2:=(\lambda-1)(2\lambda-1)+1$, {and $p_3:=(\lambda+1)(\lambda-1)^2+1$.}
		It follows that
		{
		\begin{align*}
			&n_2(\lambda,2,p_3,\tilde{c})=\tilde{c}+\tilde{c}((\lambda+1)(\lambda-1)^2+2)(\lambda-1)+(\lambda+1)(\lambda-1)^2+2,\\
			&n_2(\lambda+1,1,p_1,\tilde{c})=\tilde{c}\lambda(\lambda^2-\lambda+2)+(\lambda^2-\lambda+2),\\
			&n_2(2,\lambda,p_2,\tilde{c})=2\tilde{c}(\lambda^2-\lambda+1)+2\lambda^2-3\lambda+3.
		\end{align*}}		
		For $\lambda\geq2$, $n_2(\lambda+1,1,p_1,\tilde{c})\geq n_2(2,\lambda,p_2,\tilde{c})$ holds as  $\tilde{c}$ is at least one.
		Thus, $q\geq \max\{n_1(\lambda)$, $n_2(\lambda+1,1,p_1,\tilde{c})$, $n_2(2,\lambda,p_2,\tilde{c})$, ${	n_2(\lambda,2,p_3,\tilde{c})}$, $(\lambda-1)^2+c+1\}$ holds by the condition.

		Let $\ho^{(\lambda+1)}$ be the Hoffman graph with a slim vertex adjacent to $\lambda+1$ fat vertices. Let $\ho_{(\lambda)}$ be the Hoffman graph with two fat vertices adjacent to $\lambda$ slim vertices such that its slim graph is a complete graph.
		{Let $\ho^{(\lambda)}_{(2)}$ be the Hoffman graph consisting of two adjacent slim vertices, one of which is adjacent to exactly $\lambda$ fat vertices, while the other has no fat neighbors.}
		By Proposition \ref{215}, $G$ does not contains $G(\ho^{(\lambda+1)},p_1)$,
		$G(\ho_{(\lambda)},p_2)$, {or $G(\ho^{(\lambda)}_{(2)},p_3)$,} as an induced subgraph.
		%{\color{red}Similarly, $G$ cannot contain $G(\ho^{(\lambda)}_{(2)},p_3)$ as an induced subgraph.}
		It follows that, by Proposition \ref{assohoff},  $\go(G,q)$ does not contain $\ho^{(\lambda+1)}$, { $\ho^{(\lambda)}_{(2)}$}, { or} $\ho_{(\lambda)}$ as an induced Hoffman subgraph
		 for $q\geq \max\{n_1(\lambda)$,{$ n_2(\lambda,2,p_3,\tilde{c})$,} $ n_2(\lambda+1,1,p_1,\tilde{c})$, $n_2(2,\lambda,p_2,\tilde{c})$, $(\lambda-1)^2+c+1\}$. 
		 Since $\go(G,q)$ does not contain $\ho^{(\lambda+1)}$ as an induced Hoffman subgraph, each vertex is contained in at most $\lambda$ maximal cliques in $\Bo$. 
	{Especially, 	$N(x)=\bigcup_{i=1}^{s}\{y\mid x,y\in V(P_i)\}-\{x\}$ if $x$ is contained in exactly $\lambda$ maximal cliques in $\Bo$, as $\go(G,q)$ does not contain $\ho^{(\lambda)}_{(2)}$ as an induced Hoffman subgraph.}
		 Additionally,  since  $\go(G,q)$ does not contain $\ho_{(\lambda)}$ as an induced Hoffman subgraph, any two maximal cliques in $\Bo$ intersect in at most $\lambda-1$ vertices.
	\end{proof}

		Lemma \ref{blhlem} can show that every vertex $x$ is contained in a large clique with order at least $\frac{d(x)}{\lambda^2}-f(\lambda,c)$ if the graph is $\mu$-bounded with parameter $c$ and the smallest eigenvalue at least $-\lambda$. However, as an application of Proposition \ref{assohoff} and Lemma \ref{blhlem}, the following theorem states that 
	every vertex $x$ is contained in a large clique with order at least $\frac{d(x)}{\lambda}-\ell(\lambda,c)$ under the same conditions of Lemma \ref{blhlem}. This also improves the result in \cite{YK24}.
	\begin{thm}	\label{c1}
		Let $c$ be a positive integer and $\lambda\geq 2$ be a real number.
		There exists a positive integer $\ell:=\ell(\lambda,c)$ such that if $G$ is a $\mu$-bounded graph with parameter $c$ and $\lambda_{\min}(G)\geq -\lambda$, then each vertex $x\in V(G)$ lies in a maximal clique $C(x)$ of order at least
		%$\frac{d(x)-\lfloor\lambda^2\rfloor(\lceil\lambda\rceil-1)\lceil\lambda\rceil-\binom{\lfloor\lambda^2\rfloor}{2}(c-1)}{\lceil\lambda\rceil}+1$.    
		$\frac{d(x)-\ell}{\lceil\lambda\rceil}+1$.
	\end{thm}	
	
	\begin{proof}%[\bf Proof of Theorem \ref{c1}]
	Let $p= 1+\lceil\lambda\rceil(\lceil\lambda\rceil-1)$.
	Let $\tilde{c}:=\min\{c,\lceil\lambda\rceil(\lceil\lambda\rceil-1)\}$ and $q:=\max\{n_1(\lceil\lambda\rceil),n_2(\lceil\lambda\rceil+1, 1,  p,\tilde{c}),(\lceil\lambda\rceil-1)^2+c+1\}= \max\{ \lceil\lambda\rceil^4-2\lceil\lambda\rceil^3+3\lceil\lambda\rceil^2-3\lceil\lambda\rceil+3,(\lceil\lambda\rceil^2-\lceil\lambda\rceil+2)(\tilde{c}\lceil\lambda\rceil+1)\}$
	, where $n_1(\lambda)$ and $n_2(\phi, \sigma,  p,\tilde{c})$ are defined in Proposition \ref{assohoff}. 
	By Proposition \ref{215}, $\lambda_{\rm min}(G(\ho^{(\lceil\lambda\rceil+1)},p))<-\lceil\lambda\rceil$ holds. 
	This implies that $G$ does not contain $G(\ho^{(\lceil\lambda\rceil+1)},p)$ as an induced subgraph.
	Thus, the associated Hoffman graph $\go(G,q)$ does not contain $\ho^{(\lceil\lambda\rceil+1)}$ as an induced Hoffman subgraph by  Proposition \ref{assohoff}. This means that for each vertex $x\in V(G)$, there are at most $\lceil\lambda\rceil$ distinct maximal cliques 
	 with order at least $q$ containing $x$.
	 
	 Let $\ell=\ell(\lambda,c):= \lfloor\lambda^2\rfloor(q-1)+\binom{\lfloor\lambda^2\rfloor}{2}(c-1)-1$.
	Let $\mathcal{C}=\{C_1(x),\ldots,C_s(x)\}$ be the set of all maximal cliques containing $x$ with order at least $q$. Let $A(x):=N(x)-\bigcup_{i=1}^sV(C_i(x))$.
	We claim that $|A(x)|\leq\ell$. Otherwise $|A(x)|>  \ell$,
	let $H$ be the subgraph of $G$ induced by $A(x)\cup\{x\}$.
	Note that in $H$, $x$ is a vertex with degree at least $\ell$. By Lemma \ref{blhlem}, $x$ lies in a clique $C$ in $H$ with order at least $q$. It is not hard to see that $C$ is not contained in a clique of $\mathcal{C}$, which contradicts the definition of $\mathcal{C}$. Therefore, $|A(x)|\leq \ell$ and $\max\{|V(C_1(x))|,\ldots,|V(C_s(x))|\}\geq \frac{d(x)-\ell}{\lceil\lambda\rceil}+1$ holds.
	\end{proof}

	\begin{re}\label{ell}
	The proof shows that $\ell(\lambda,c)= \lfloor\lambda^2\rfloor(q-1)+\binom{\lfloor\lambda^2\rfloor}{2}(c-1)-1$, where $q= \max\{c+1+(\lceil\lambda\rceil-1)^2, \lceil\lambda\rceil^4-2\lceil\lambda\rceil^3+3\lceil\lambda\rceil^2-3\lceil\lambda\rceil+3,(\lceil\lambda\rceil^2-\lceil\lambda\rceil+2)(\tilde{c}\lceil\lambda\rceil+1)\}$ and $\tilde{c}=\min\{c,\lceil\lambda\rceil(\lceil\lambda\rceil-1)\}$.
	\end{re}
	
\section{Applications on DRGs with classical parameters}\label{aDRG}
	In this section, we give some applications on distance-regular graphs with classical parameters and prove Corollary \ref{incor}, Theorem \ref{a5},
	and Theorem \ref{alphab}. 
	We denote by $\mathbb{N}$ the set of all non-negative integers

	\begin{proof}[\bf Proof of Corollary \ref{incor}]
		Let $G$ be a local graph of $\Gamma$. It follows that $G$ is an $a_1$-regular and $\mu$-bounded graph with parameter $c= c_2-1=3\alpha+2$ by Equation (\ref{drgc}).
	By Lemma \ref{ei}, the second largest eigenvalue of  $\Gamma$ is $\lambda_1=\frac{b_1}{2}-1$. This implies that the smallest eigenvalue of $G$ is at least $-3$	by Theorem \ref{bcnth}.

	Furthermore, $\beta\leq \min\{a_1-36c-400\tilde{c}-83, a_1-44c+5 \}$, 
	as $\alpha(2^D-2)\geq \max\{36c+400\tilde{c}+84, 44c-4 \}$ and $a_1=b_0-b_1-c_1=\beta-1+\alpha(2^D-2)$ by Equations (\ref{drgc}) and (\ref{drgb}).
	By Lemma \ref{ei}, the smallest eigenvalue of $\Gamma$ is $-2^D+1$. Additionally, by Lemma \ref{dels}, every clique in $G$ has order at most $\beta$, as the degree of $\Gamma$ is $k=b_0=\beta(2^D-1)$.
	Therefore,	 $\go(G,q)$ is a $2$-fat \big\{\raisebox{-1ex}{\begin{tikzpicture}[scale=0.3]
			
			\tikzstyle{every node}=[draw,circle,fill=black,minimum size=10pt,scale=0.3,
			inner sep=0pt]
			
			\draw (-2.1,0) node (1f1) [label=below:$$] {};
			\draw (-1.6,0) node (1f2) [label=below:$$] {};
			\draw (-1.1,0) node (1f3) [label=below:$$] {};

			\tikzstyle{every node}=[draw,circle,fill=black,minimum size=5pt,scale=0.3,
			inner sep=0pt]

			\draw (-1.6,1) node (1s1) [label=below:$$] {};
			
			\draw (1f1) -- (1s1) -- (1f2);
			\draw (1f3) -- (1s1);
	\end{tikzpicture}},\hspace{-0.08cm}
	\raisebox{-1ex}{\begin{tikzpicture}[scale=0.3]
			\tikzstyle{every node}=[draw,circle,fill=black,minimum size=10pt,scale=0.3,
			inner sep=0pt]

			\draw (1.5,0) node (3f1) [label=below:$$] {};
			\draw (1.5,1) node (3f2) [label=below:$$] {};

			\tikzstyle{every node}=[draw,circle,fill=black,minimum size=5pt,scale=0.3,
			inner sep=0pt]

			\draw (1,0.5) node (3s1) [label=below:$$] {};
			\draw (2,0.5) node (3s2) [label=below:$$] {};
			
			\draw (3s1) -- (3s2);
			\draw (3f1) -- (3s1) -- (3f2) -- (3s2) -- (3f1);
	\end{tikzpicture}}\big\}-line Hoffman graph by Theorem \ref{intro2}.
	\end{proof}
	
	\begin{pro}\label{5}
			Let $\Gamma$ be a distance-regular graph with classical parameters $(D,b,\alpha,\beta)$ such that $b=2$ and $D\geq12$. If $\alpha\leq 9$ holds, then  $\alpha\in\{0,\frac{1}{3},\frac{2}{3},1,\frac{4}{3},2\}$.
	\end{pro}
	
	\begin{proof}
	By Equation (\ref{drgc}), $c_2=3+3\alpha$ and $c_3=7(1+3\alpha)$.
	Since $c_2$ is an integer, so is $3\alpha$.
		Furthermore, $3\alpha\geq 0$ as $c_3$ is a positive integer. 
		Thus, $\alpha\in\frac{1}{3}\mathbb{N}$.
		By Proposition \ref{417} with $(i,h)=(6,6)$ and Equation (\ref{drgc}),
		 $$p^{12}_{66}=\frac{c_7c_8c_9c_{10}c_{11}c_{12}}{c_1c_2c_3c_4c_5c_6}=\frac{230674393235(1+63\alpha)(1+127\alpha)(1+255\alpha)(1+511\alpha)(1+1023\alpha)(1+2047\alpha)}
		 {(1+\alpha)(1+3\alpha)(1+7\alpha)(1+15\alpha)(1+31\alpha)}$$
		 is a non-negative integer. This proposition can be directly verified by using Wolfram Mathematica 13.3, under the conditions that $\alpha\in\frac{1}{3}\mathbb{N}$, $\alpha\leq 9$, and $p^{12}_{66}$ is a non-negative integer.
	\end{proof}
	
%	\begin{re}
%		In fact, Proposition \ref{417} with $(i,h)=(3,3)$ implies that $\alpha\leq 5$ if $\alpha<9$ and $D\geq 6$.
%	\end{re}
	
%	\begin{lem}\label{beta}\cite{}
%		Let $\Gamma$ be a distance-regular graph with classical parameters $(D,b,\alpha,\beta)$ such that $b=2$. Then $\beta\leq 2^5(2^D-1)$ holds.
%	\end{lem}
	
%	Let  $c_x(u):=\max\{|N_{\Gamma_x}(u)\cap N_{\Gamma_x}(v)|\mid v\in N_\Gamma(x)-N_\Gamma(u)\}$ for $x\in\Gamma$ and $u\in \Gamma_x$. The following lemma can bound $\alpha$ by $c_x(u)$.
%	\begin{lem}
%			Let $\Gamma$ be a distance-regular graph with classical parameters $(D,b,\alpha,\beta)$ such that $b=2$, $D\geq 9$ and $\alpha\neq 0$. Let $x\in V(\Gamma)$ and $u\in V(\Gamma_x)$. Then $\alpha\leq\max\{c_x(u),5\}$.
%	\end{lem}

The following theorem gives an upper bound for $\alpha$ in terms of $b$ when $D\geq 9$ and $b\geq 2$. 
\begin{thm}\label{beta}
	Let $D$ and $b$ be positive integers.
	If $\Gamma$ is a distance-regular graph with classical parameters $(D,b,\alpha,\beta)$ such that $D\geq 9$ and $b\geq2$, then $(b+1)\alpha$ is a non-negative integer and one of the following holds:
	\begin{enumerate}
		\item $\alpha\leq b(b+1)$,
		\item $b(b+1)<\alpha< b^2(b+1)+1$ and $\beta<(b(b+1)^2-\alpha)\Big[\substack{D\\ \\ 1}\Big]_b+(b+1)^6b+(\frac{(b+1)^3((b+1)^2+1)}{2}+1)\alpha-b$.
	\end{enumerate}
	Moreover, $\alpha\leq b^2(b+1)$ holds if $D\geq 10$.
\end{thm}

\begin{proof}
	Since $c_2=(b+1)(\alpha+1)$ is a positive integer, $(b+1)\alpha$ is an integer. Furthermore,  $(b+1)\alpha$ is also a non-negative integer  as $c_3=(1+b+b^2)(1+(b+1)\alpha)$ is a positive integer.
	Thus, $\alpha\in\frac{1}{b+1}\mathbb{N}$.
	 Assume that $b(b+1)<\alpha$ holds. It follows that $b(b+1)+\frac{1}{b+1}\leq \alpha$.

	Let $G$ be a local graph of $\Gamma$.
	It is not hard to see that $G$ has $n$ vertices with 	$n=b_0=\beta\frac{b^D-1}{b-1}$, and $G$ is  $w$-regular and $\mu$-bounded with parameter $c$, where $w=a_1=\beta-1+\alpha\frac{b^D-b}{b-1}$
	 and $c=c_2-1=(b+1)(\alpha+1)-1$.
	Note that $w\geq \alpha\frac{b+1}{b-1}(b^{D-1}-1)$ holds by Inequality (\ref{ie1}).
	By Theorem \ref{bcnth}, $\lambda_{\min}(G)\geq -b-1$ holds.
	Let $\ell(\lambda,c)$ be the function defined in Remark \ref{ell}, where $\lambda=b+1$.
%	Thus $\lambda=b+1$, $\tilde{c}=\min\{c,b(b+1)\}=b(b+1)$, $q= \max\{c+1+b^2, (b+1)^3(b-1)+3(b^2+b+1),((b+1)b+2)(b(b+1)^2+1)\}\leq(b+1)\alpha+((b+1)b+2)(b(b+1)^2+1)$
%	and $\ell(\lambda,c)=(b+1)^2(q-1)+\binom{(b+1)^2}{2}((b+1)(\alpha+1)-2)-1$.
	For $b\geq 2$, $\ell(\lambda,c)< (b+1)^6b+\frac{(b+1)^3((b+1)^2+1)}{2}\alpha$ holds.	
	By Theorem \ref{c1}, there exists a maximal clique $C$ with order $r\geq\frac{w-\ell(\lambda,c)}{b+1}+1$.
	It follows that $r-b^2\geq g(D,b,\alpha)$ as $w\geq \alpha\frac{b+1}{b-1}(b^{D-1}-1)$, where $g(D,b,\alpha):=\alpha(\frac{(b^{D-1}-1)}{b-1}-\frac{(b+1)^2((b+1)^2+1)}{2})-(b+1)^5b-(b-1)(b+1)$.

	We claim that each vertex outside $C$ has at most $b(b+1)$ neighbors in $C$.
	 By Lemma \ref{Hat},
	each vertex outside $C$ has either at most $b(b+1)$ neighbors in $C$ or at least $r-b^2$ neighbors in $C$. 
	As $\alpha\geq b(b+1)+\frac{1}{b+1}$,  by using Wolfram Mathematica 13.3 under the conditions $b\geq 2$ and $D\geq 9$, it can be obtained that $r-b^2\geq g(D,b,\alpha)\geq c_2=(1+\alpha)(1+b)$.
	Since $G$ is $\mu$-bounded with parameter $c_2-1$, each vertex outside $C$ has at most $b(b+1)$ neighbors in $C$.
	
	Let $e$ be the number of edges between $V(C)$ and $V(G)-V(C)$ in $G$. 
	It follows that $e=r(w-r+1)\leq b(b+1)(n-r)$. Hence,
	$$\beta(r-(b+1)b\frac{b^D-1}{b-1})\leq r(r-b(b+1)-\alpha\frac{b^D-b}{b-1}).$$

%It follows that $r-b^2-(b+1)(\alpha+1)\geq\alpha g(b)-(b+1)^5b-b(b+1)$, where $g(b):=\frac{(b^D-b)}{(b+1)(b-1)}-\frac{(b+1)^2((b+1)^2+1)}{2}-b-1>0$ for $b\geq 2$ and $D\geq 9$ by calculation. Thus $r-b^2-(b+1)(\alpha+1)> (b(b+1)+\frac{1}{b+1})g(b) -(b+1)^5b-b(b+1)>0$ for $b\geq 2$ and $D\geq 9$ by calculation.

We claim that $r< b(b+1)\frac{b^D-1}{b-1}$. Otherwise 
	 $r\geq b(b+1)\frac{b^D-1}{b-1}$, and thus $r\leq \beta$ holds by Lemma \ref{dels}. It follows that $r(r-(b+1)b\frac{b^D-1}{b-1})\leq r(r-b(b+1)-\alpha\frac{b^D-b}{b-1}).$ This implies that $\alpha\leq b(b+1)$, which contradicts  the assumption.
	 Hence, $r< b(b+1)\frac{b^D-1}{b-1}$ holds.

Since $ b(b+1)\frac{b^D-1}{b-1}\geq r\geq \frac{w-\ell(\lambda,c)}{b+1}+1$ and $w=\beta-1+\alpha\frac{b^D-b}{b-1}$, by Inequality (\ref{ie1}),
	  the following inequality holds
	\begin{align*}
		b(b+1)^2\frac{b^D-1}{b-1}+\ell(\lambda,c)-b-1
		\geq w= \beta-1+\alpha\frac{b^D-b}{b-1}
		\geq \alpha\frac{(b+1)(b^D-b)}{b(b-1)}.
	\end{align*}
	It follows that
	\begin{equation}\label{5.2beta}
		\beta<(b(b+1)^2-\alpha)\frac{b^D-1}{b-1}+(b+1)^6b+(\frac{(b+1)^3((b+1)^2+1)}{2}+1)\alpha-b,
	\end{equation} and 
	\begin{equation}\label{ab}
		\alpha< b^2(b+1)+f(D,b),
	\end{equation}
	where $f(D,b):=\frac{b(b+1)^7}{2b^{D-1}-b(b+1)^4}$.
%	In particular, Inequality (\ref{5.2beta}) is the upper bound for $\beta$ in (ii).
		Note that $f(D,b)$ is a monotonically decreasing function with respect to both $D$ and $b$ if $D\geq 9$ and $b\geq 2$.
		
	Now we show that $\alpha\leq b^2(b+1)$ in case of $D\geq 9$ and $2\leq b\leq 99$.
	Note that $f(9,b)<6$ in this case, verified by Wolfram Mathematica 13.3. 
	Thus, $f(D,b)\leq f(9,b)<6$
	and  $\alpha< b^2(b+1)+6$ in this case.
	By Equation (\ref{drgc}) and Proposition \ref{417} with $(i,h)=(4,4)$,
	$p^8_{44}$ is a non-negative integer depending on $\alpha$ and $b$.
	By using Wolfram Mathematica 13.3 under the conditions $\alpha< b^2(b+1)+6$, $2\leq b\leq  99$, and $D\geq 9$ to verify $p^8_{44}\in \mathbb{N}$, we have $\alpha\leq b^2(b+1)$ in this case.
	
	 To prove (ii), it is sufficient to show $\alpha<b^2(b+1)+1$ in case of $b\geq 100$ and $D\geq 9$.
	Note that $f(D,b)\leq f(D,100)<1$ in this case. It follows that $\alpha<b^2(b+1)+1$ for $b\geq 100$ and $D\geq 9$, and thus (ii) holds.
%	Additionally, $f(D,b)<6$ for $2\leq b\leq 14$ and $D\geq 9$, verified by Wolfram Mathematica 13.3. Thus,  $\alpha< b^2(b+1)+6$ for $2\leq b\leq  14$ and $D\geq 9$.
	
% Combining these, $\alpha<  b^2(b+1)+1$ for $b\geq 2$ and $D\geq 9$, and thus (ii) holds.
 
 To complete the proof, we will show that $\alpha\leq b^2(b+1)$  in case of $b\geq 100$ and $D\geq 10$.
 Note that $(b+1)f(D,b)$ is a  monotonically decreasing function with respect to both $D$ and $b$ if $D\geq 10$ and $b\geq 2$.
 Thus, $(b+1)f(D,b)\leq 101f(D,100)<1$ in this case.
 It follows that $f(D,b)<\frac{1}{b+1}$ in this case.
 This implies that $\alpha\leq   b^2(b+1)$ in this case by Inequality (\ref{ab}), as $\alpha\in\frac{1}{b+1}\mathbb{N}$.
 This completes the proof.
%	Additionally, $f(D,b)<3$ for $2\leq b\leq 19$,
%	verified by Wolfram Mathematica 13.3.
%	Thus, $\alpha<b^2(b+1)+3$ for $2\leq b\leq 19$ and $D\geq 10$.
%	It follows that $\alpha\leq b^2(b+1)$ for  $2\leq b\leq 19$ and $D\geq 10$, by using Wolfram Mathematica 13.3 under the conditions $\alpha< b^2(b+1)+3$, $2\leq b\leq  19$, $D\geq 10$, and $p^8_{44}$ is a non-negative integer only depending on $\alpha$ and $b$. Combining these, $\alpha\leq   b^2(b+1)$ for $b\geq 2$ and $D\geq 10$.
\end{proof}

\begin{proof}[\bf Proof of Theorem \ref{introalpha}]
	It can be directly derived from Theorem \ref{beta}.
\end{proof}

	Now, we are in the position to prove Theorem \ref{a5}.

\begin{proof}[\bf Proof of Theorem \ref{a5}]
	Let $c=3\alpha+2$ and $\tilde{c}:=\min\{c
	,6\}$. Assume that $\alpha>2$ holds. By Proposition \ref{5}, $\alpha\geq 9$ holds. Thus, $\alpha(2^D-2)\geq \max\{36c+400\tilde{c}+84, 44c-4 \}$.
	Let $G$ be a local graph of $\Gamma$. 
%	It follows that $G$ is $w$-regular with $n$ vertices, where $w=a_1=\beta-1+\alpha(2^D-2)$ and $n=b_0=\beta(2^D-1)$. 
		It is not hard to see that $G$ has $n$ vertices with 	$n=b_0=\beta(2^D-1)$, and $G$ is  $w$-regular and $\mu$-bounded with parameter $c$, where $w=a_1=\beta-1+\alpha(2^D-2)$.
	Furthermore,
 $\go(G,q)$ is a $2$-fat \big\{\raisebox{-1ex}{\begin{tikzpicture}[scale=0.3]
			
			\tikzstyle{every node}=[draw,circle,fill=black,minimum size=10pt,scale=0.3,
			inner sep=0pt]
			
			\draw (-2.1,0) node (1f1) [label=below:$$] {};
			\draw (-1.6,0) node (1f2) [label=below:$$] {};
			\draw (-1.1,0) node (1f3) [label=below:$$] {};

			\tikzstyle{every node}=[draw,circle,fill=black,minimum size=5pt,scale=0.3,
			inner sep=0pt]

			\draw (-1.6,1) node (1s1) [label=below:$$] {};
			
			\draw (1f1) -- (1s1) -- (1f2);
			\draw (1f3) -- (1s1);
	\end{tikzpicture}},\hspace{-0.08cm}
	\raisebox{-1ex}{\begin{tikzpicture}[scale=0.3]
			\tikzstyle{every node}=[draw,circle,fill=black,minimum size=10pt,scale=0.3,
			inner sep=0pt]

			\draw (1.5,0) node (3f1) [label=below:$$] {};
			\draw (1.5,1) node (3f2) [label=below:$$] {};

			\tikzstyle{every node}=[draw,circle,fill=black,minimum size=5pt,scale=0.3,
			inner sep=0pt]

			\draw (1,0.5) node (3s1) [label=below:$$] {};
			\draw (2,0.5) node (3s2) [label=below:$$] {};
			
			\draw (3s1) -- (3s2);
			\draw (3f1) -- (3s1) -- (3f2) -- (3s2) -- (3f1);
	\end{tikzpicture}}\big\}-line Hoffman graph by Corollary \ref{incor}. 

By Lemma \ref{9}, there exists a vertex $u$ in $G$ such that $|N_G(u)\cap N_G(v)|\leq 9$ for each vertex $v$ non-adjacent to $u$. 
		Let $x\in N_G(u)$ such that $|N_G(x)\cap N_G(u)|=\max\{|N_G(v)\cap N_G(u)|\mid v\in N_G(u)\}$.
		Let $e$ be the  number of edges between $N_G(u)-\{x\}$ and $V(G)-N_G(u)-\{u\}$ in $G$.
	It follows that
	\begin{equation}\label{e1}
	e\leq 9(n-w-1)=9(\beta-\alpha)(2^D-2).
	\end{equation}

 By Proposition \ref{muassle2} (i), there exists a set of cliques $\Bo=\{C_1,\ldots,C_s\}$ in $G$ such that $u,v$ are contained in at least one clique $C_i$  for each vertex $v\in N_G(u)-\{x\}$.
By Proposition  \ref{muassle2} (iii), if there are two cliques $C,C'\in\Bo$ such that $\{u,v\}\subset C\cap C'$, then $v$ must be the vertex $x$.
 Thus, for each vertex $v\in N_G(u)-\{x\}$, $u,v$ are contained in exactly one clique $C_i\in \Bo$.
 Let $\Co(w):=\{C\in\Bo| w\in C\}$ for each vertex $w$ in $G$.
  By Proposition \ref{muassle2} (ii) and (iii), $|N_G(u)\cap N_G(v)|\leq  |C_i|-2+\sum_{A\in\Co(u)-\{C_i\}}\sum_{B\in\Co(v)-\{C_i\}}|A\cap B|\leq |C_i|-2+2\times 2\times2$ for each vertex $v\in N_G(u)-\{x\}$. By Lemma \ref{ei}, the smallest eigenvalue of $\Gamma$ is $-2^D+1$. As the degree of $\Gamma$ is $\beta(2^D-1)$, by Lemma \ref{dels}, every clique in $G$ has order at most $\beta$. Thus, $|N_G(u)\cap N_G(v)|\leq \beta+6$ for each vertex $v\in N_G(u)-\{x\}$, and
 \begin{equation}\label{e2}
 	e\geq (w-1)(w-\beta-6)=(\beta-2+\alpha(2^D-2))(\alpha(2^D-2)-7).
 \end{equation}
	
	Combining Inequalities (\ref{e1}) and (\ref{e2}), $\alpha\leq 9+\frac{7}{2^D-2}-\frac{9(\alpha(2^D-1)-2)}{\beta-2+\alpha(2^D-2)}$ holds. 
	By Theorem \ref{beta}, $\beta\leq 14(2^D-1)$ holds for $\alpha\geq 9$ and $D>9$. It follows that $\frac{7}{2^D-2}<\frac{9(\alpha(2^D-1)-2)}{\beta-2+\alpha(2^D-2)}$. This implies $\alpha<9$, which contradicts the assumption.
\end{proof}

%\begin{re}
%We can prove that  $\alpha\in\{0,\frac{1}{3},\frac{2}{3},1,\frac{4}{3},2,\frac{11}{3},5,9\}$, by setting $(i,h)=(4,4)$ in Lemma \ref{417}.
%\end{re}

%For the general $b$, we have the following theorem.

%Similar as the method of Proposition \ref{5}, we check the possibility of $\alpha$ for $b\in[10]$ and find that $\alpha\leq b$ if $b\in\{3,5,6,8,10\}$ and $D\geq10$.
Using a method similar to the proof for Proposition \ref{5}, we examine the possible values of $\alpha$ for $b \in [100]$ and prove Theorem \ref{alphab}.
%\begin{cor}
%	Let $\Gamma$ be a distance-regular graph with classical parameters $(D,b,\alpha,\beta)$ such that $b\in\{3,5,6,8,10,11,12,13,14,15\}$ and $D\geq10$. Then  $\alpha\leq b$ holds. Moreover, if $b=7$ and $D\geq 14$ then $\alpha\leq b$ holds.
%\end{cor}
%\begin{proof}
%	By Proposition \ref{417}, $p^{10}_{55}$, $p^8_{44}$, and $p^6_{33}$ are integers.	By the program of Wolfram Mathematica 13.3, we have the following table.

%\centering	\begin{tabular}{|c|c|}
%		\hline
%		$b$ & $\alpha$\\
%		\hline
%		$3$ & $\frac{1}{2},1,\frac{3}{2},2,\frac{9}{4},3$\\
%		\hline
%		$4$ & $\frac{1}{5},\frac{3}{5},1,2,\frac{12}{5},3,\frac{16}{5},4,6$\\
%		\hline
%		$5$ & $\frac{2}{3},2,\frac{10}{3},4,\frac{25}{6},5$\\
%		\hline
%		$6$ & $\frac{5}{7},\frac{30}{7},5,\frac{36}{7},6$\\
%		\hline
%		$7$ & $\frac{3}{4},3,\frac{21}{4},6,\frac{49}{8},7,349$\\
%		\hline
%		$8$ & $\frac{7}{9},\frac{56}{9},7,\frac{64}{9},8$\\
%		\hline
%		$9$ & $\frac{4}{5},2,4,6,\frac{36}{5},8,\frac{81}{10},9,12$\\
%		\hline
%		$10$ & $\frac{9}{11},\frac{90}{11},9,\frac{100}{11},10$\\
%		\hline
%		$11$ & $\frac{5}{6},5,\frac{55}{6},10,\frac{121}{12},11$\\
%		\hline
%		$12$ & $\frac{11}{13},\frac{132}{13},11,\frac{144}{13},12$\\
%		\hline
%		$13$ & $\frac{6}{7},6,\frac{78}{7},12,\frac{169}{14},13$\\
%		\hline
%		$14$ & $\frac{13}{15},\frac{182}{15},13,\frac{196}{15},14$\\
%		\hline
%		$15$ & $\frac{7}{8},7,\frac{105}{8},14,\frac{225}{16},15$\\
%		\hline
%		$16$ & $\frac{15}{17},3,12,\frac{240}{17},15,\frac{256}{17},16,20$\\
%		\hline
%	\end{tabular}

%As $p^{14}_{77}$ is also an integer, $\alpha\leq b$ if $b=7$ and $D\geq 14$. 
%\end{proof}

\begin{proof}[\bf Proof of Theorem \ref{alphab}]
	Theorem \ref{beta} shows that $\alpha\in\frac{1}{b+1}\mathbb{N}$ and $\alpha\leq b^2(b+1)$.
	Furthermore, by Proposition \ref{417} and Equation (\ref{drgc}), $p^{14}_{77}$, $p^{12}_{66}$, $p^{10}_{55}$, $p^8_{44}$, and $p^6_{33}$ all are non-negative integers depending on $\alpha$ and $b$.  
	This result can be directly verified by using Wolfram Mathematica 13.3, under the conditions that $\alpha\in\frac{1}{b+1}\mathbb{N}$, $\alpha\leq b^2(b+1)$, $b\in[100]$, and $p^{14}_{77}$, $p^{12}_{66}$, $p^{10}_{55}$, $p^8_{44}$, $p^6_{33}$ all are non-negative integers only depending on $\alpha$ and $b$. 
\end{proof}

\section*{Acknowledgements}
{We are very grateful to the anonymous referee for comments and suggestions that have contributed a lot to improve the quality of our article.}
J.H. Koolen is partially supported by the National Key R. and D. Program of China (No. 2020YFA0713100), the National Natural Science Foundation of China (No. 12071454, No. 12471335), and the 
Anhui Initiative in Quantum Information Technologies (No. AHY150000). H.-J. Ge is supported by CSC scholarship program (No.  202506340038). Q. Yang is supported by the National Science Foundation of China (No. 12401460) and the Shanghai Sailing Program (No. 23YF1412500).

\bibliographystyle{plain}
\bibliography{KGLY2024}

\end{document}